\documentclass{article}

\usepackage[nonatbib,final]{neurips_2022}  

\usepackage[T1]{fontenc} 
\usepackage[latin9]{inputenc}
\usepackage{amstext}
\usepackage{amsthm}
\usepackage{amssymb}
\usepackage{dsfont}
\usepackage{color}
\usepackage[dvipsnames]{xcolor}
\makeatletter

\usepackage[T1]{fontenc}
\usepackage[american]{babel}
\usepackage{url}
\usepackage{booktabs}
\usepackage{amsfonts}
\usepackage{nicefrac}
\usepackage{csquotes}
\usepackage{enumitem}
\usepackage{etoolbox}
\usepackage{stmaryrd}
\usepackage{amstext}
\usepackage{amsmath}
\usepackage{amssymb}
\usepackage{xcolor}
 
\usepackage{appendix}
\usepackage{graphicx}
\usepackage{epstopdf}
\usepackage{subcaption}
\usepackage{blindtext}
\usepackage{algorithmic}
\usepackage{algorithm}
\usepackage[colorlinks,linkcolor={red!80!black},citecolor={blue},allbordercolors={1 1 1},urlcolor={blue!80!black},hypertexnames=false]{hyperref}
\usepackage[noabbrev,capitalize]{cleveref}  
\crefformat{equation}{(#2#1#3)}
\crefrangeformat{equation}{(#3#1#4) to~(#5#2#6)}
\crefmultiformat{equation}{(#2#1#3)}{ and~(#2#1#3)}{, (#2#1#3)}{ and~(#2#1#3)}
\crefname{appsec}{Appendix}{Appendices}

\usepackage{autonum}

\usepackage[numbers]{natbib}

\usepackage{automatic_resize}
\newdelimcommand{brakets}{[}{]}

\newcommand{\R}{\mathbb R}

\newcommand\xmone{{x_{k\text{--}1}}}
\newcommand\ymone{{y_{k\text{--}1}}}

\newcommand{\X}{\mathcal X}
\newcommand{\Y}{\mathcal Y}
\newcommand{\UL}{{\bf UL}\xspace}
\newcommand{\LL}{{\bf LL}\xspace}
\newcommand{\selection}{selection\xspace}
\newcommand{\Selection}{Selection\xspace} 

\newcommand{\eg}{{\textit{e.g.}}}

\newcommand\vsp{{\vspace*{-0.2cm}}}

\makeatletter
\newcommand\given{\@ifstar{\mathrel{}\middle|\mathrel{}}{\mid}}

\DeclareRobustCommand{\abs}{\@ifstar\@abs\@@abs}

\DeclareRobustCommand{\norm}{\@ifstar\@norm\@@norm}

\DeclareRobustCommand{\inner}{\@ifstar\@inner\@@inner}
\makeatother

\newcommand*\diff{\mathop{}\!\mathrm{d}}

\usepackage{mathtools}

\newtheorem{lem}{Lemma}
\newtheorem{thm}{Theorem}
\newtheorem{defn}{Definition}
\newtheorem{prop}{Proposition}
\newtheorem{corr}{Corollary}
\newtheorem{assump}{Assumption}

\newlist{assumplist}{enumerate}{1}
\setlist[assumplist]{label=(\textbf{\Alph*})}
\Crefname{assumplisti}{Assumption}{Assumptions}
\newlist{assumplist2}{enumerate}{2}
\setlist[assumplist2]{label=(\textbf{\alph*})}
\Crefname{assumplist2i}{Assumption}{Assumptions}

\newlist{assumplistobs}{enumerate}{3}
\setlist[assumplistobs]{label=(\textbf{$\mathcal{F}$-\alph*})} 
\Crefname{assumplistobs}{Assumption}{Assumptions}
\Crefname{assump}{Assumption}{Assumptions}
\makeatother

\usepackage{babel}

\crefname{lem}{Lemma}{Lemmas}
\Crefname{lem}{Lemma}{Lemmas}
\crefname{thm}{Theorem}{Theorems}
\Crefname{thm}{Theorem}{Theorems}
\crefname{prop}{Proposition}{Propositions}
\Crefname{prop}{Proposition}{Propositions}
\crefname{corr}{Corollary}{Corollaries}
\Crefname{corr}{Corollary}{Corollaries}

\title{Non-Convex Bilevel Games \\
with Critical Point Selection Maps}

\author{Michael Arbel and Julien Mairal \\
Univ. Grenoble Alpes, Inria, CNRS, Grenoble INP, LJK, 38000 Grenoble, France\\
 \small{\texttt{firstname.lastname@inria.fr}}} 

\date{
    $^1$ \emph{Inria Grenoble Rh\^{o}ne Alpes}\\%
}

\begin{document}

\maketitle

\begin{abstract}
Bilevel optimization problems involve two nested objectives, where an upper-level objective depends on a solution to a lower-level problem. When the latter is non-convex, multiple critical points may be present, leading to an ambiguous definition of the problem. In this paper, we introduce a key ingredient for resolving this ambiguity through the concept of a \emph{selection} map which allows one to choose a particular solution to the lower-level problem. Using such maps, we define a class of hierarchical games between two agents that resolve the ambiguity in bilevel problems. 
This new class of games requires introducing new analytical tools in Morse theory to extend \emph{implicit differentiation}, a technique used in bilevel optimization resulting from the implicit function theorem. In particular, we establish the validity of such a method even when the latter theorem is inapplicable due to degenerate critical points.
Finally, we show that algorithms for solving bilevel problems based on unrolled optimization solve these games up to approximation errors due to finite computational power. 
A simple correction to these algorithms is then proposed for removing these errors.
\end{abstract}

\section{Introduction}
Bilevel optimization has proven to be a major tool for solving machine learning problems that possess a nested structure such as hyper-parameter optimization \citep{Feurer:2019}, meta-learning \citep{Bertinetto:2018},  
reinforcement learning \citep{Hong:2020a,Liu:2021}, or  dictionary learning \citep{Mairal:2011}. 
Introduced in the field of economic game theory in~\cite{Stackelberg:1934}, a bilevel optimization problem can be understood as a game between a \emph{leader} and a \emph{follower} each of which optimizes their own objective function but where the leader can anticipate {follower}'s actions. 
In the context of machine learning, the leader typically optimizes a hyper-parameter over a validation loss while the follower optimizes the model parameter on a training loss \cite{Lorraine:2020}.  

Bilevel optimization introduces many challenges. In particular, when multiple optimal solutions are available to the follower, the leader would need to optimize a different objective depending on the follower's strategy to select an optimal solution. 
As a result, the bilevel problem becomes ambiguously defined without knowing the follower's strategy~\cite{Liu:2021b}. 
A large body of work on bilevel programs for machine learning gets around these considerations by assuming the follower to have a unique optimal choice, a situation that typically occurs when the follower's objective is strongly convex, leading to efficient and scalable algorithms \citep{Ablin:2020,Arbel:2021a,Blondel:2021,Domke:2012,Gould:2016,Liao:2018,Liu:2021,Shaban:2019}. 
However, in many machine learning applications, the strong convexity of the follower's objective is an unrealistic assumption. This is particularly the case in the context of deep learning, where the follower's objective, the training loss, can be highly non-convex in the parameters of the model and can have regions of flat optima due to symmetries and other degeneracies \cite{Draxler:2018,Li:2018d}. 

In the literature on mathematical optimization, the ambiguity in bilevel problems is often resolved by making an additional assumption on the follower's strategy for choosing their optimal solution. In particular, two problems are often considered: \emph{optimistic and pessimistic bilevel programs}, see~\cite{Dempe:2007}. Both problems rely on two assumptions: (i) the follower is using a strategy for selecting a solution to their problem that is either improving or degrading the leader's objective and (ii) the leader knows exactly what strategy the follower is using.    
These assumptions are strong from a game-theoretical perspective and often unrealistic for machine learning problems such as hyper-parameter optimization. 
Still, optimistic/pessimistic bilevel games are well defined and early works have proposed several algorithms to solve them with strong convergence guarantees~\citep{Ye:1997,Ye:1995,Ye:1997a}. Yet, these algorithms are often ill-suited to large-scale and high-dimensional problems arising in machine learning applications as they rely on second-order optimization methods such as Newton's method ~\cite{Guo:2015}.
For this reason, scalable first-order algorithms for such games have been proposed recently
\cite{Liu:2021f,Liu:2021b}. 

However, many of the best-performing approaches for hyper-parameter optimization rely neither on an optimistic nor a pessimistic formulation of the bilevel problem \citep{Vicol:2021}. 
Instead, they often rely on algorithms initially designed for bilevel problems with strongly convex lower objectives even though the convexity assumption does not hold \cite{Lorraine:2020}.
Consequently, these algorithms are solving a seemingly ill-defined bilevel program due to the ambiguity in the way the follower selects their solution. 
However, their ability to provide models with good empirical performance raises the question of whether these algorithms are solving another class of well-defined hierarchical problems beyond optimistic and pessimistic bilevel programs that are still relevant for machine learning.

In this work, we answer the above question by introducing 
\emph{Bilevel Games with Selection} (BGS), a class of games between two agents: a leader and a follower, where the leader uses a mechanism for anticipating the solution of the follower without knowing the exact follower's strategy.
We define such a mechanism using the notion of a \emph{selection}, which is simply a map for selecting a particular solution to the follower's objective given the current state of the game. 
In particular, BGS recovers a usual bilevel program when the follower's objective admits a unique solution.
By playing a BGS, the agents seek an equilibrium point for which each of their objectives ceases to vary. The equilibria are completely determined by the \emph{selection} thus resulting in a well-defined problem.

When the selection is differentiable, the equilibrium point can be characterized by a first-order optimality condition which enables gradient-based approximations.
More precisely, we show that \emph{implicit differentiation} \cite{Pedregosa:2016}, which, a priori, is only valid when the critical points of the follower's objective are non-degenerate, remains applicable for solving BGS even when these critical points are degenerate. 
To this end, we consider a general construction of the selection as the limit of a gradient flow of the follower's objective and prove the differentiability of such a selection near local minimizers, provided the follower's objective satisfies a generalization of the 
\emph{Morse-Bott property} \cite{Austin:1995,Feehan:2020a}.
We then characterize the differential of the selection as a solution to a linear system thus extending implicit differentiation to degenerate critical points. 
Finally, we leverage this characterization to show that popular algorithms based on iterative differentiation (ITD) \cite{Baydin:2018} find fixed points approximating the BGS's equilibria up to approximation errors. We then introduce a simple corrective term to these algorithms based on implicit differentiation to remove these errors.

\section{Related Work}

\paragraph{Iterative/Unrolled optimization (ITD)} is a class of methods approximating the lower-level solution map by a differentiable function obtained through successive gradient updates \cite{Baydin:2018}. 
When the lower-level objective is strongly convex, these algorithms solve a well-defined bilevel problem up to an error that is controlled by increasing the computational budget for the approximate solution~\cite{Ji:2021a}. 
Our analysis suggests a simple algorithmic correction to these approaches which can result in solutions to a bilevel game with a constant budget for the approximate solution.

\vsp
\paragraph{Approximate Implicit Differentiation (AID)}
is a class of methods approximating the variations of the lower-level solution map using the Implicit Function theorem \cite{Franceschi:2018,Ghadimi:2018,Pedregosa:2016,Rajeswaran:2019}. 
The non-degeneracy requirement under which the latter theorem holds restricts the applicability of AID to,  essentially, strongly convex lower-level objectives. 
These algorithms admit fixed points that match the solutions to the bilevel problem   
 \citep{Ghadimi:2018,Hong:2020a,Ji:2021,Ji:2021a}. As such, they typically require a smaller computational budget than ITD \cite{Arbel:2021a,Ji:2021a}. 
 Recently, \cite{Bolte:2021,Bolte:2022,Bolte:2022a} extended AID to non-smooth objectives while still requiring non-degenerate critical points. The present work is complementary to these works as it extends AID to smooth objectives that have possibly degenerate critical points. 

\vsp
\paragraph{Optimistic and pessimistic bilevel optimization.}  
When the lower-level objective is non-convex, the ambiguity of the problem arising from the multiplicity of the lower-level solutions can be resolved by optimizing the upper-level objective over all such possible solutions \cite{Wiesemann:2013,Zemkoho:2016}. 
The \emph{optimistic} and \emph{pessimistic} problems arise when either minimizing or maximizing the upper-level over all such lower-level solutions. 
Early works proposed to solve these problems using exact penalization \citep{Ye:1997a}, second-order optimization \citep{Ye:1997,Ye:1995} or smoothing method \citep{Xu:2014}. However, these approaches are hard to scale to the high dimensional problems arising in machine learning.
More recently, \cite{Liu:2021b,Liu:2021f} considered first-order methods based on unrolled optimization or interior-point methods for solving optimistic bilevel problems and provided approximation guarantees. 
However, as shown in \cite{Vicol:2021}, most practical applications to bilevel optimization rely on a formulation that goes beyond optimistic or pessimistic formulations. 
The present work departs from these approaches and instead introduces a bilevel game that is more tractable to solve. We show that popular bilevel algorithms, such as unrolled optimization, yield approximations of these games.

\section{Non-Convex Bilevel Optimization with Selection 
}
{\bf Notations.} Define $\X \!=\! \R^p$ and $\Y \!=\! \R^d$ for some positive integers $p$ and $d$. We consider two real valued functions $f$ and $g$ defined on $\X{\times} \Y$ and assume $g$ to be twice-continuously differentiable. 

\subsection{Background on Bilevel Optimization}\label{sec:background}
A bilevel program is an optimization problem where an upper-level objective $f$ defined over a set $\X\times\Y$ of variables $(x,y)$ is optimized in the first variable $x$ under the constraint that the second variable $y$ is optimal for a lower-level objective $y\mapsto g(x,y)$ depending on the upper-variable $x$.   
When $g(x,.)$ admits a unique minimizer denoted by $y^{\star}(x)$, which is the case if $y\mapsto g(x,y)$ is strongly convex, the bilevel problem is well-defined and can be expressed as:
\begin{align}\tag{BP}\label{eq:unique_bilevel_problem}
	\min_{x\in \X} f(x,y^{\star}(x)), \qquad	y^{\star}(x):=\arg\min_{y\in \Y} g(x,y).
\end{align}
When $g$ is non-convex, the set of minimizers $T(x) {:=} \arg\min_{y}g(x,y)$ may contain more than one element making \cref{eq:unique_bilevel_problem} ambiguous. 
A possible approach for resolving the ambiguity is to adopt a game-theoretical point of view, where a lower-level agent uses a particular strategy for selecting a solution in $T(x)$. 
For instance, in \emph{pessimistic} bilevel games, the lower agent chooses a minimizer of $g(x,.)$ that maximizes $f(x,.)$ while the upper agent minimizes the resulting worst-case loss $F$ in $x$:
\begin{align}\tag{pessimistic-BG}\label{eq:pessimistic}
   \text{(\UL):}~~~ \min_{x\in \X} F(x), \qquad \text{and}~~~     \text{(\LL):}~~~~F(x):=\max_{y \in \Y} f(x,y)~~~\text{s.t.}~~~ y \in T(x).
\end{align}
Similarly, an \emph{optimistic} bilevel game can be obtained by replacing maximization with minimization so that both agents cooperate.
While these approaches are highly relevant from a game-theoretical point of view, many machine learning applications do not rely on a pessimistic/optimistic bilevel formulation.
For instance, for hyper-parameter optimization, the lower agent may have access to training data, but it should not have access to the validation data processed (used in~$f$) by the upper agent. 
Instead, a popular approach consists of applying algorithms designed for bilevel
programs that admit unique solutions for the lower problems, even though this assumption may not hold in practice~\cite{Lorraine:2020}. 
In the next section, we introduce a class of games that allow characterizing the equilibrium points obtained by these popular algorithms while resolving the ambiguity of non-convex bilevel problems and bypassing the limitations of pessimistic/optimistic bilevel formulations.
\subsection{Bilevel Games with Selection (BGS)}\label{sec:BGS} 
We introduce a new class of nested games for bilevel optimization with two
agents, a \emph{leader} and a \emph{follower}.
The \emph{follower} minimizes the lower-level objective $g$ w.r.t. a variable $y$ in $\Y$. Similarly, the \emph{leader} minimizes the upper-level objective $f$ w.r.t. a variable $x\in \X$ while anticipating the \emph{follower}'s solution. 
More precisely, the \emph{leader} has access to a \emph{selection map}: $\phi: \X {\times} \Y \rightarrow \Y$ to choose a unique critical point $\phi(x,y)$ of $y\mapsto g(x,y)$ given the current state of the game $(x,y)\in\X {\times} \Y$ thus allowing the leader to anticipate the follower's solution. Typically, the selection $\phi(x,y)$ represents the critical point that is \emph{selected} by an optimization process of $g(x,.)$ starting from an initial condition $y$ (\eg, the limit of a gradient flow for a gradient descent algorithm). 
The  Bilevel Game with Selection (BGS) is therefore defined as the following interdependent optimization problems: 
\begin{align}\tag{BGS}\label{eq:ABG}
\text{(\UL):}\quad  \min_{x\in \X} \mathcal{L}_{\phi}(x,y):= f(x,\phi(x,y)), \qquad\qquad \text{(\LL):} \quad \min_{y\in\Y} g(x,y).
\end{align} 
Given a selection map $\phi$, the game \cref{eq:ABG} is well-defined and does not suffer from the ambiguity problem in \cref{eq:unique_bilevel_problem}.
The explicit dependence of $\phi(x,y)$ on the initialization $y$ might seem unnecessary at first, as one could simply fix $y$ to some value $y_0$ and consider only the dependence on the variable $x$. 
However, such a dependence on the variable $y$ allows performing \emph{warm-start} \citep{Vicol:2021}, where the lower-level problem is optimized starting from a previous state of the game, thus resulting in computational savings \cref{fig:BGS}. We provide below a formal definition for the selection map. 
\begin{defn}[{\bf \Selection map}]\label{def:selection}
   Given a continuously differentiable function $g: \X \times \Y \to \R$, the map $\phi : \X \times \Y \rightarrow \Y$ is a selection if it satisfies the following properties for any pair $(x,y)\in \X\times \Y$:
\begin{enumerate}
   \item {\bf Criticality:}  The element $y'=\phi(x,y)$ is a critical point of $g(x,.)$, i.e. $\partial_{y} g(x,y')=0$.
   \item {\bf Self-consistency:} If $y$ is a critical point of $g(x,.)$ i.e.  $\partial_y g(x,y)=0$, then $\phi(x,y)=y$.  
\end{enumerate}
\end{defn}
 \emph{Criticality} ensures the leader possesses a hierarchical advantage in that they know what are the optimal choices accessible to the follower. 
\emph{Self-consistency} implies that the leader makes a guess that is not contradicting the current choice $y$ of the follower. 
Both properties ensure the leader can rationally anticipate the follower's actions from the current state of the game $(x,y)$. 
We will see in \cref{sec:selection_morse_bott}, under mild assumptions on $g$, that it is always possible to define a selection $\phi$ as the limit of a continuous-time gradient flow of $y\mapsto g(x,y)$ initialized at $y$. Moreover, as we discuss later in \cref{sec:algorithms}, the selection does not need to be explicitly constructed for solving \cref{eq:ABG} in practice. It can be simply related to the implicit bias of the algorithm used for solving the follower's problem.
\vsp
\paragraph{Connection to \cref{eq:unique_bilevel_problem}.} When the lower-level objective $y\mapsto g(x,y)$ admits a unique minimizer $y^{\star}(x)$, 
it is easy to check that there exists a unique \selection map $\phi$ satisfies $\phi(x,y) {=}y^{\star}(x)$. Hence, \cref{eq:ABG} recovers the bilevel problem in \cref{eq:unique_bilevel_problem} as a particular case.

\vsp
\paragraph{Connection to \cref{eq:pessimistic} or the optimistic variant.}
Key differences between~\cref{eq:ABG} and pessimistic or optimistic games is that (i) the follower has never access to the upper function~$f$ with~\cref{eq:ABG}, which matches practical hyper-parameter optimization applications where $f$ relies on a validation dataset, whereas $g$ relies on a distinct training set; (ii) the leader in~\cref{eq:pessimistic} does not take into account the strategy used by the follower, whereas the leader in~\cref{eq:ABG} makes more rational choices by guessing the strategy of the follower through the selection map~$\phi$.

\vsp
\begin{figure}
\center
	\includegraphics[width=.42\linewidth,height=0.3\linewidth]{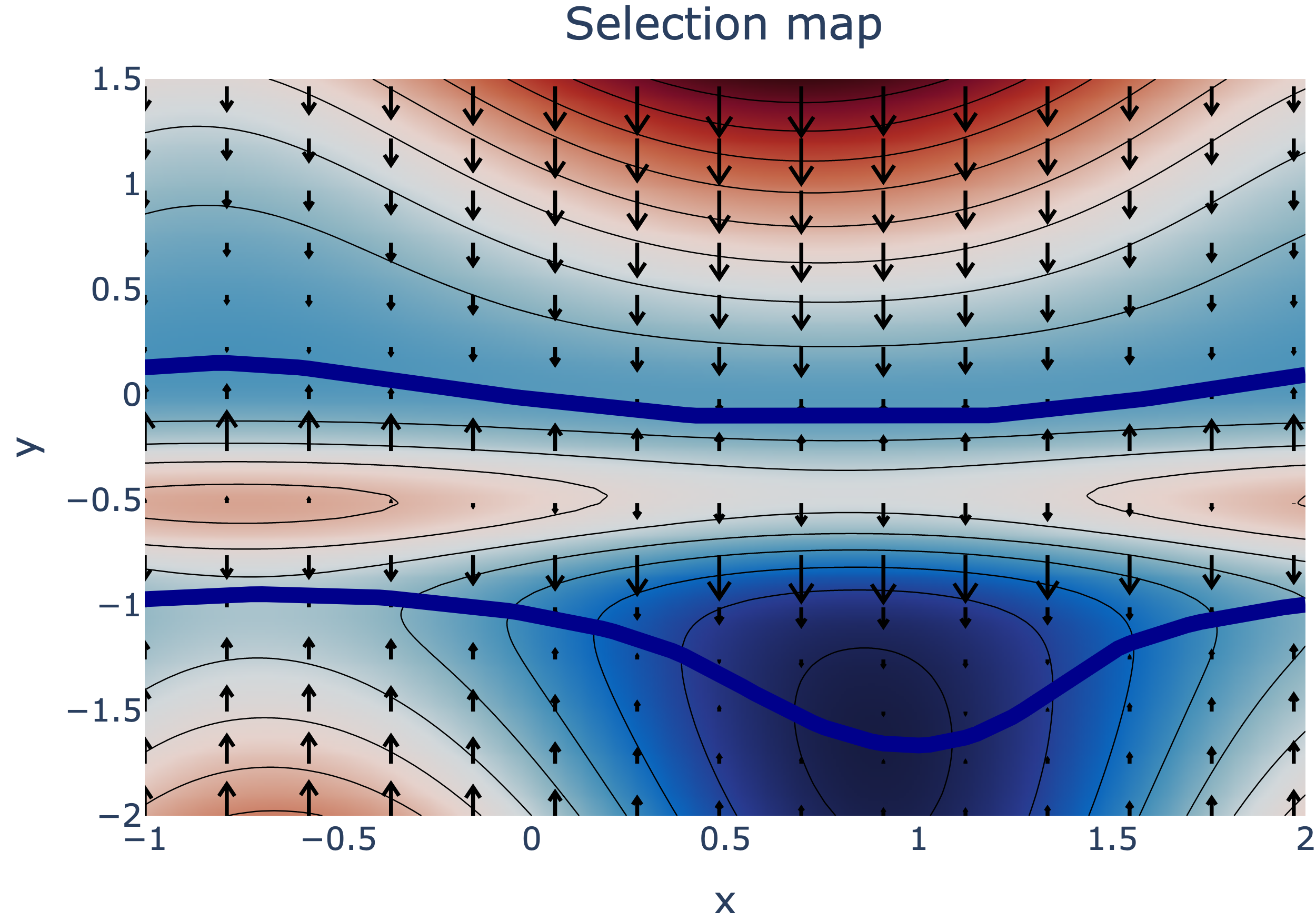}%
	\includegraphics[width=.42\linewidth]{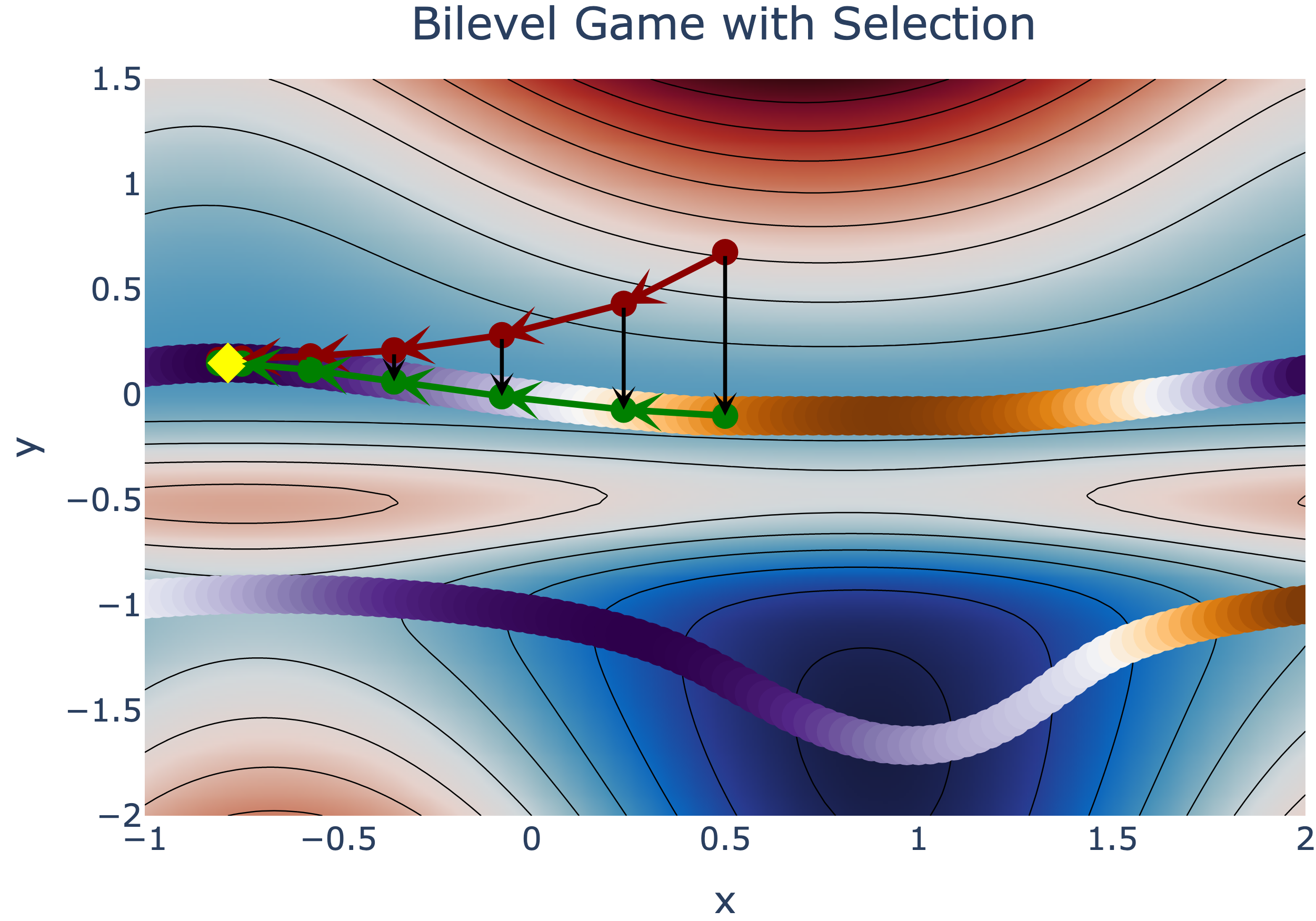}%
	\includegraphics[width=.16\linewidth]{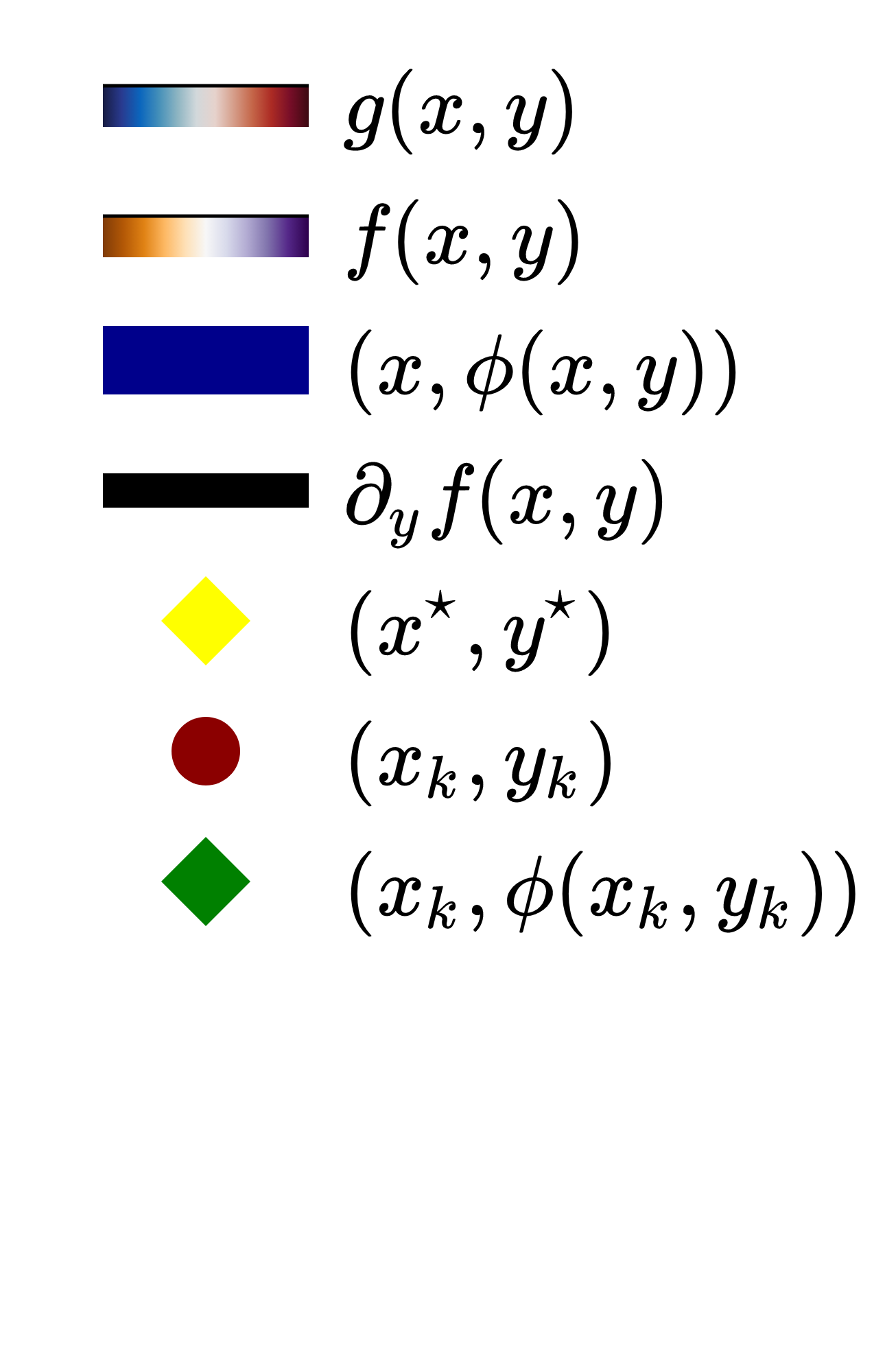}%
	\caption{\small Left: Heatmap of the lower-level objective $g(x,y)$. The local minimizers of $y\mapsto g(x,y)$ are represented by the 'critical lines' in blue. The selection map $\phi(x,y)$ is defined by following the vector field $\partial_y g(x,y)$, in black.
	Right: Iterates $(x_k,y_k)$ (in red) obtained by playing a BGS. The follower finds the next update $y_k$ by optimizing $y\mapsto g(x_k,y)$ starting from previous iterate $y_{k-1}$. The leader finds the next update $x_k$ by optimizing the upper-level objective $f$ along the 'critical lines' (iterates in green).  
	}
\label{fig:BGS}
\end{figure}
\paragraph{First-order equilibrium conditions.}
The agents can play the game \cref{eq:ABG} by successively taking actions $(x_k,y_k)$ to improve their own objectives   $x\mapsto \mathcal{L}_{\phi}(x,y_{k-1})$ and $y\mapsto g(x_k,y)$, by hoping the strategy will reach an equilibrium pair $(x^*,y^*)$ \cref{fig:BGS}(Right). 
In the case where $f$, $g$ and $\phi$ are differentiable at $(x^{*},y^*)$, the equilibrium pair is characterized by a first-order stationary condition: 
\begin{align}\tag{SC}\label{eq:SC}
	\partial_x \mathcal{L}_{\phi}(x^{\star},y^{\star})= \partial_x f(x^{\star},y^{\star}) + \partial_x\phi(x^{\star},y^{\star})\partial_y f(x^{\star},y^{\star}) = 0,\qquad \partial_y g(x^{\star},y^{\star})=0.
\end{align}
When $g$ is smooth and strongly convex in $y$, the implicit function theorem~\cite[Theorem 5.9]{Lang:2012} ensures that $\phi$ is differentiable and provides an expression of $\partial_x\phi(x^{\star},y^{\star})$ as a solution to a linear system which key for implicit differentiation.
This allows to devise efficient algorithms using estimates of the gradient~$\partial_x\mathcal{L}_{\phi}$, see, \eg, \cite{Arbel:2021a}.  
However, extensions of the implicit function theorem, such as the \emph{constant rank theorem} ~\citep[Theorem 4.12]{Lee:2003}, for cases where $g$ has possibly degenerate critical points require strong assumptions on $g$ which are unrealistic in machine learning. 
In the next section, we provide new analytical tools for extending  \emph{implicit differentiation} by studying the differentiability of a family of \selection maps corresponding to a large class of functions $g$. The resulting expression will be key for devising first-order methods to solve \cref{eq:ABG}, as discussed in \cref{sec:algorithms}.
\section{Selection Based on Gradient Flows for Parameteric Morse-Bott Functions}\label{sec:selection_morse_bott}
In this section, we extend implicit differentiation to a class of functions with possibly degenerate critical points. 
To this end, we consider a particular selection $\phi(x,y)$ obtained as the limit of a gradient flow $(\phi_t(x,y))_{t\geq 0}$ of $g(x,.)$ initialized at~$y$. 
We then study the \emph{differentiability} w.r.t. $x$ of the selection by analyzing the dynamics of such a gradient flow.
For general non-convex functions, the selection might be non-differentiable since a small perturbation to the parameter $x$ can change the geometry of the critical points of $g$, causing the perturbed flow to move away from the non-perturbed one (see \cref{fig:morse_bott}).  
We are therefore interested in functions $g$  preserving the local geometry near critical points as $x$ varies. In \cref{sec:param_morse-bott}, we introduce such a class of functions called parametric Morse-Bott functions, which covers many practical machine learning models. 
We then show, in \cref{sec:smoothness_selections}, that the selection resulting from such a function is differentiable near local minima.
\begin{figure}
	\includegraphics[width=.42\linewidth,height=.3\linewidth]{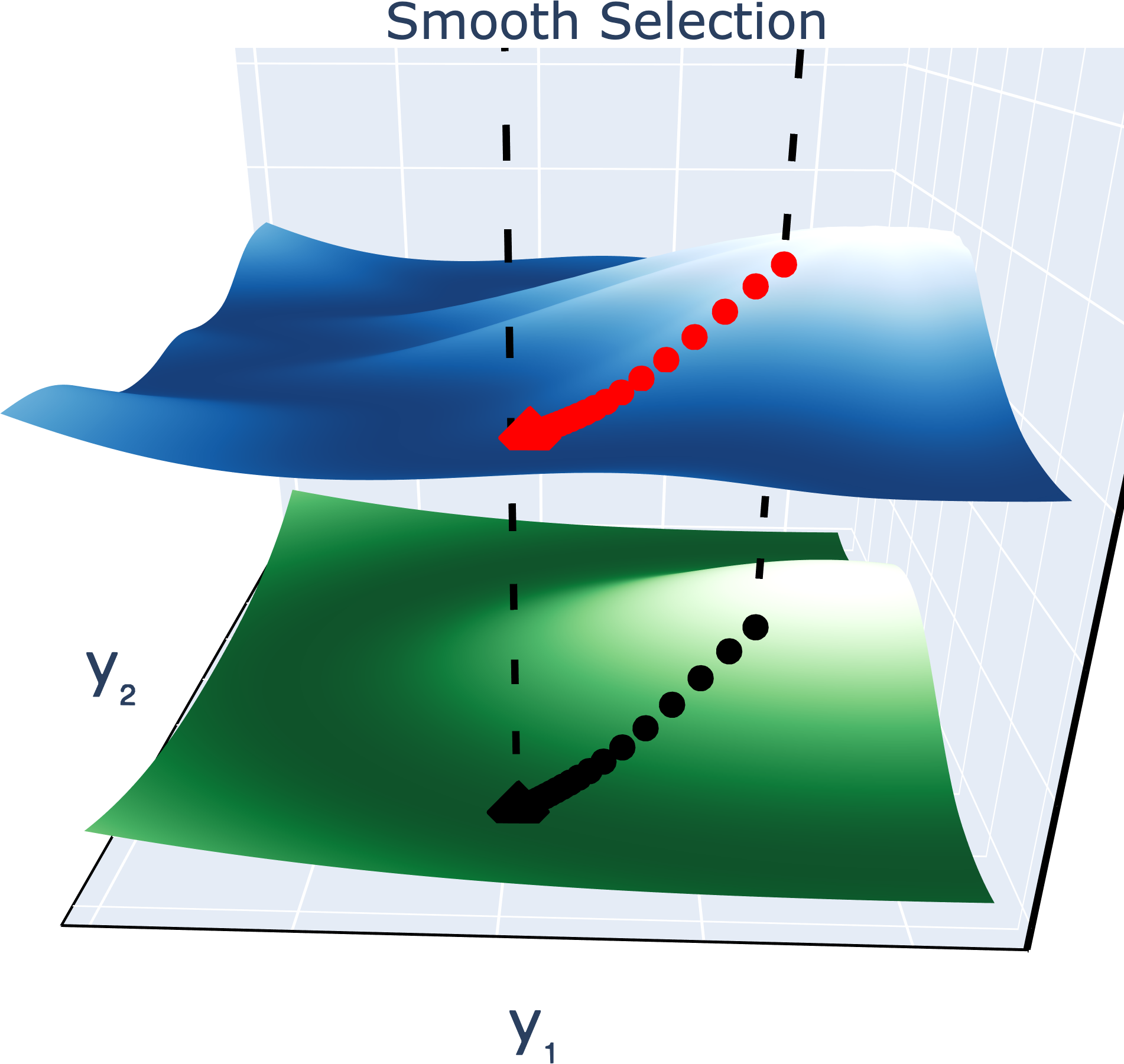}%
	\includegraphics[width=.42\linewidth,height=.3\linewidth]{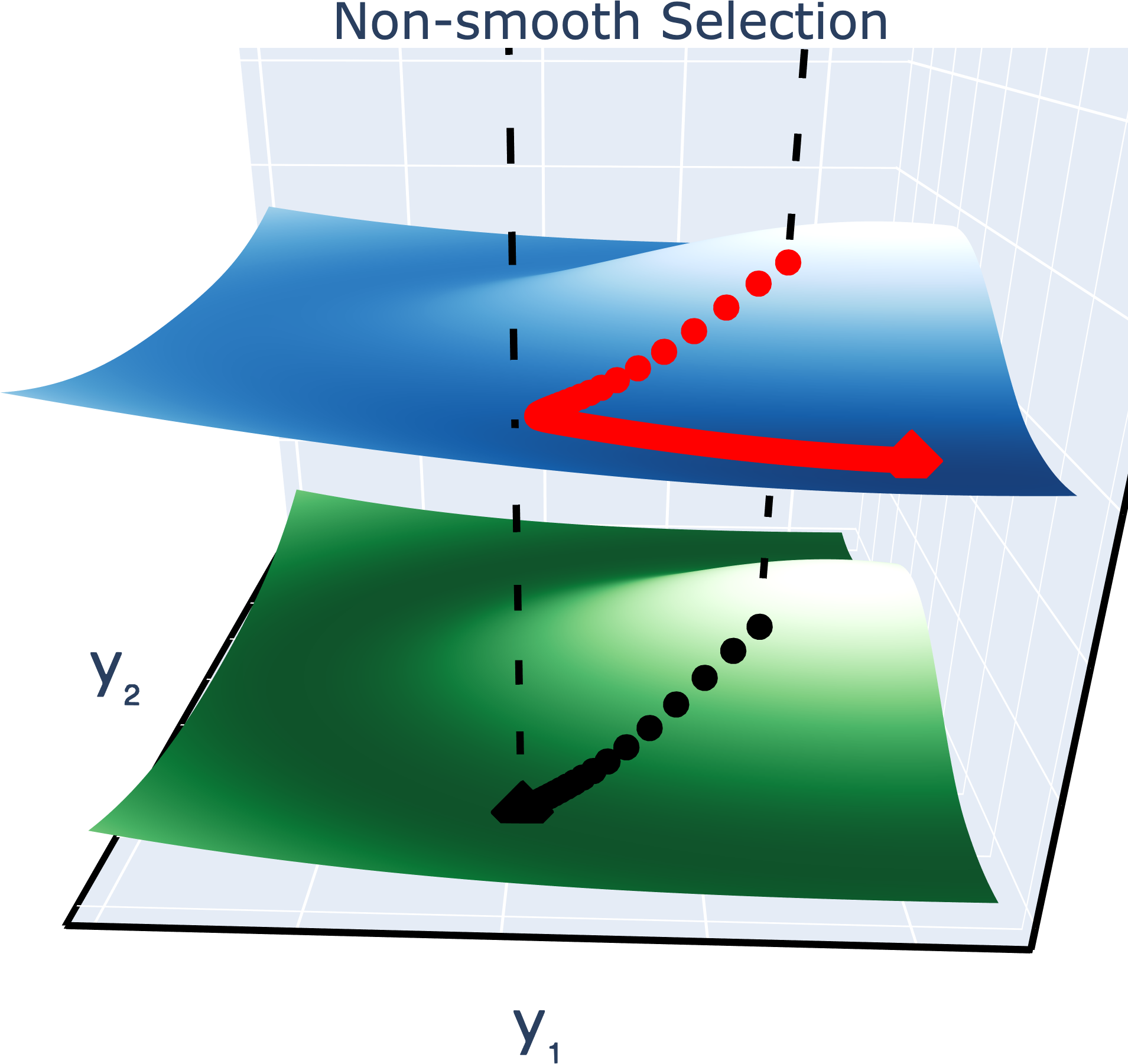}%
	\includegraphics[width=.16\linewidth]{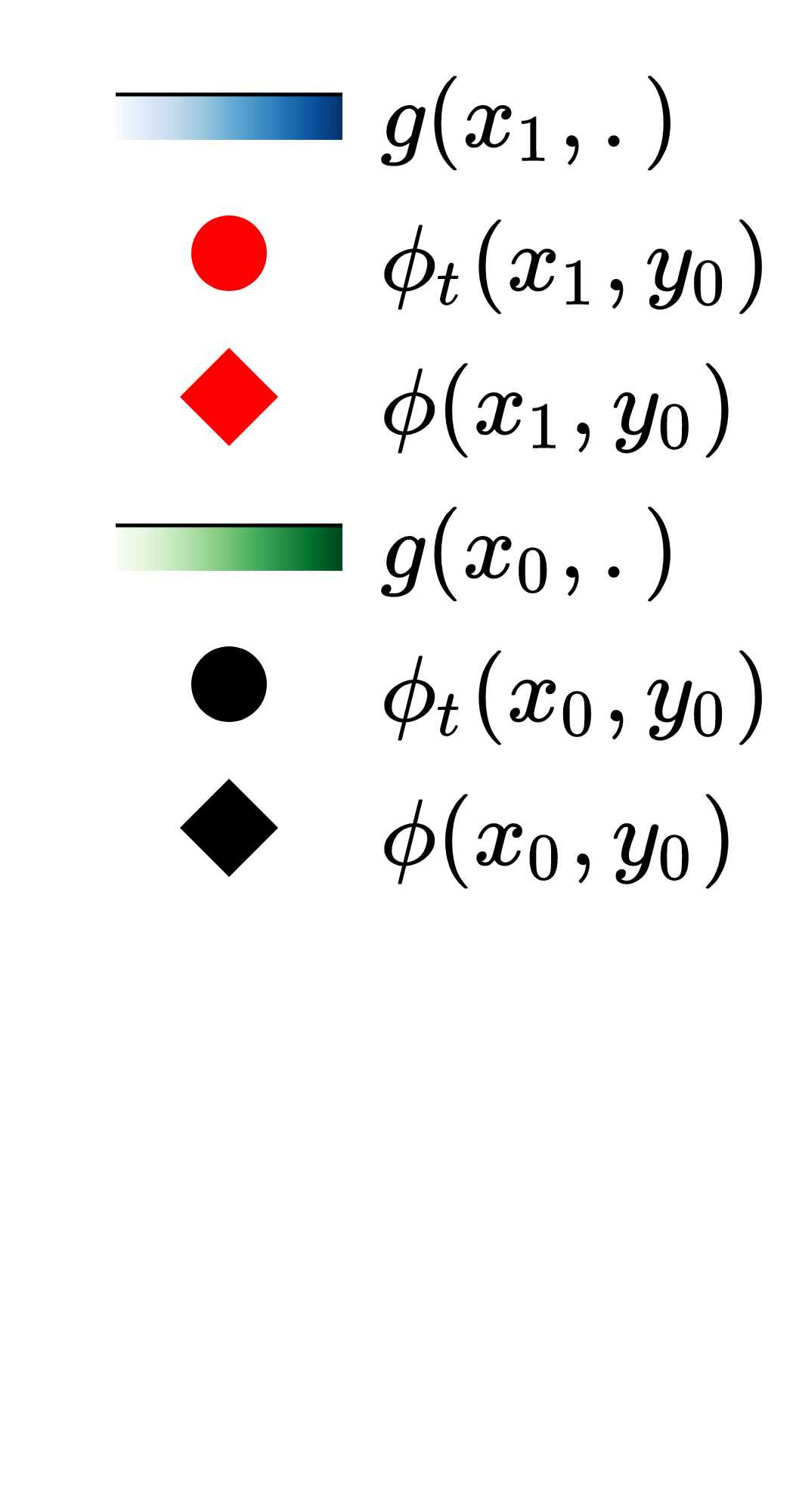}
	\caption{\small Two examples of functions $g$ with different behaviors of the gradient flow under perturbations of $x$. In both figures, the green surface represents a function $y\mapsto g(x_0,y)$ with $y\in \R^2$ resembling a \emph{Mexican hat} which has a manifold of (degenerate) local minimizers (in dark green). The blue surfaces represent \emph{deformed} versions of the Mexican hat function when the parameter $x$ is slightly perturbed $x_1 {\approx} x_0$. Depending on the deformation, the resulting function $y{\mapsto} g(x_1,y)$ can either preserve the same type of critical points as the unperturbed function, i.e. local minimizers remain local minimizers (Left), or change their type, i.e.: local minimizers can become saddle-points (Right). Left: the selection behaves smoothly as a function of the deformation. Right: the selection is discontinuous since the gradient flow is pushed away from $\phi(x_0,y_0)$ which is deformed into a saddle point.
	}
\label{fig:morse_bott}
\end{figure}
\subsection{Parameteric Morse-Bott Functions}\label{sec:param_morse-bott}
We introduce parametric Morse-Bott functions, a class of parametric functions $g:\X \times \Y\rightarrow \R$ with parameter $x$ in $\X$ extending the more familiar notion of  Morse-Bott functions (\cref{sec:Morse-Bott-functions},  \citep{Feehan:2020a}) to account for the effect of the parameter $x$ on the geometry of critical points.
\begin{defn}[\bfseries Parametric Morse-Bott function.]\label{def:parameteric_morse_bott}
Let $g:\X\times \Y$ be a real-valued twice continuousely differentiable function and define the set of \emph{augmented critical points } $\mathcal{M}$ as follows:
\begin{align}\label{eq:augmented_critical_points}
\mathcal{M} := \braces{ (x,y)\in \X\times \Y ~\middle|~ \partial_y g(x,y)=0} 
\end{align}
Let $(x_0,y_0)\in\mathcal{M}$. We say that $g$ is Morse-Bott at $y_0$ w.r.t. $x_0$, if there exists an open neighbordhood $\mathcal{V}$ of $(x_0,y_0)$ s.t. the intersection $\mathcal{M}\cap\mathcal{V}$ is a $C^2$-connected sub-manifold of $\X{\times} \Y$ of dimension:
\begin{align}
   \dim(\mathcal{M}\cap\mathcal{V}) = \dim(\X) + dim\parens{\text{Ker}(\partial^2_{yy} g(x_0,y_0))}.
\end{align}
$g$ is a parametric Morse-Bott function  if for any $(x_0,y_0){\in} \mathcal{M}$, $g$ is Morse-Bott at $y_0$ w.r.t. $x_0$.
\end{defn} 
The functions in \cref{def:parameteric_morse_bott} satisfy a condition that is stronger than simply satisfying the Morse-Bott property at any parameter value $x$ (\cref{def:morse_bott} of \cref{sec:Morse-Bott-functions}). Indeed, we show in \cref{prop:pointwise_morse_bott} of \cref{sec:properties_parameteric_morse_bott} that, for any $x_0\in \X$, the function $y\mapsto g(x_0,y)$ is a Morse-Bott function, meaning that the critical set $C(x_0)$ of $y\mapsto g(x_0,.)$ near a critical point $y_0$ is locally a  $C^2$ connected sub-manifold of $\Y$ of dimension equal to the dimension of the null-space of the Hessian $\partial^2_{yy} g(x_0,y_0)$.  
For conciseness, we introduce the following assumption which ensures $g$ satisfies the condition of \cref{def:parameteric_morse_bott} as well as possesses continuous third-order derivatives.
\begin{assump}[\bfseries Parameteric Morse-Bott property]\label{assumpt:morse-bott}
	The function $g$ is at least three-times continuously differentiable and is a parameteric Morse-Bott function as defined in \cref{def:parameteric_morse_bott}.
\end{assump}

\vsp
\paragraph{Examples of parametric Morse-Bott function.}
A notable class of parametric Morse-Bott functions is the one containing all twice-continuously differentiable functions that are strongly convex or, more generally, possess only non-degenerate critical points in the second variable as shown in \cref{prop:morse-functions-with-parameters} of \cref{sec:properties_parameteric_morse_bott}. %
Note that parametric Morse-Bott functions need not be convex and can have multiple (possibly degenerate) local minima, saddle-points, and local maxima. 

Another class of functions, this time with possibly degenerate critical points,  are those that can be expressed as a composition of some Morse-Bott function $h$ and a family $(\tau_x)_{x\in \X}$ of diffeomorphisms on $\Y$ parameterized by $x$, i.e. $g(x,y) {=} h(\tau_x(y))$. This particular form is relevant in generative modeling where the diffeomorphisms are defined using  normalizing flows of parameter $x$ \cite{Rezende:2015}.  

The condition in \cref{def:parameteric_morse_bott} ensures that the degree of freedom of the augmented critical set $\mathcal{M}$ is exactly determined by the degree of freedom of the parameter $x$ and the degree of degeneracy of the Hessian at a critical point $y$.
This condition is precisely what guarantees the stability of the local shape of critical points when the parameter $x$ varies as we formalize through the next theorem.
\begin{thm}[\bfseries Morse-Bott lemma with parameters]\label{thm:morse-bott_lemma}
   Let $g$ be a function satisfying \cref{assumpt:morse-bott}. Let $(x_0,y_0)$ in $\mathcal{M}$ be an augmented critical point of $g$. Denote by $\mathcal{K}$ the null space of the Hessian $A_0{:=}\partial^2_{yy} g(x_0,y_0)$ and by $\mathcal{K}^{\perp}$ its orthogonal complement in $\Y$. Let $J_0$ be a diagonal matrix with diagonal element given by the sign of the non-zero eigenvalues of $A_0$.    
	Then, there exists open neighborhoods $\mathcal{U}$ and $\mathcal{V}$ of $(x_0,0_{\mathcal{K}},0_{\mathcal{K}^{\perp}})$ and $(x_0,y_0)$  in $\X {\times} \mathcal{K}{\times} \mathcal{K}^{\perp}$ and $\X{\times} \Y$, and a diffeomorphism $\psi: \mathcal{U}\rightarrow \mathcal{V}$ preserving the first variable, i.e. $\psi(x,r,w) {=} (x,y)$ for any $(x,r,w)\in \mathcal{U}$, with $\psi(x_0,0_{\mathcal{K}},0_{\mathcal{K}^{\perp}}){=}(x_0,y_0)$ such that $g$ admits the representation:
	\begin{align}
		g(\psi(x,r,w)) = g(\psi(x,0_{\mathcal{K}},0_{\mathcal{K}^{\perp}})) + \frac{1}{2}w^{\top}J_0w, \qquad \forall (x,r,w)\in \mathcal{U}.
	\end{align}
\end{thm}
\cref{thm:morse-bott_lemma}, which is proven in \cref{sec:proof_morse-bott_lemma}, shows that, near an augmented critical point $(x_0,y_0)$, $g$ looks like a quadratic function up to an additive term that depends only on the parameter $x$. Moreover, slightly varying the parameter $x$ does not change the quadratic function and thus preserves the local shape near critical points. 
\cref{thm:morse-bott_lemma} is an extension of the \emph{Morse-Bott lemma} \cite[Theorem 2.10]{Feehan:2020a} to the case when there is a dependence on a parameter $x$. It can also be seen as an extension of the \emph{Morse lemma with parameters}  \cite[Theorem 4]{Feehan:2020a} which allows dependence to a parameter $x$ but requires the critical points to be non-degenerate (invertible matrix $A_0$). To our knowledge, \cref{thm:morse-bott_lemma} is the first result in the literature providing a decomposition of parametric functions with degenerate critical points into the sum of a quadratic non-degenerate term and a singular term depending only on the parameter $x$. %
We present now a corollary of \cref{thm:morse-bott_lemma} which is a strengthened version of the standard {\L}ojasiewicz inequality~\citep{Lojasiewicz:1982} that will be essential for our subsequent analysis.
\begin{prop}[\bfseries Locally Uniform {\L}ojasiewicz gradient inequality]\label{prop:uniform_KL}
	Let $g$ be a function satisfying \cref{assumpt:morse-bott}  and let $(x_0,y_0)$ be in $\mathcal{M}$ the augmented critical set defined in \cref{def:parameteric_morse_bott}. Then, there exists an open neighborhood $\mathcal{U}$ of $(x_0,y_0)$ and a positive number $\mu>0$ such that $y\mapsto g(x,y)$ is constant on the set $\mathcal{M}\cap \mathcal{U}$ with some common value $G(x):= g(x,y)$ and the following  holds:
		\begin{align}
			\mu\verts{ g(x,y)-G(x) }\leq \frac{1}{2}\Verts{\partial_y g(x,y)}^2,\qquad \forall (x,y)\in \mathcal{U}.
		\end{align}
\end{prop}
\cref{prop:uniform_KL}, which is proven in \cref{sec:proof_morse-bott_lemma}, ensures that the {\L}ojasiewicz gradient inequality holds uniformly on $(x,y)$ near any augmented critical point $(x_0,y_0)$. This result  will be essential in \cref{sec:smoothness_selections} for defining a selection $\phi$ obtained as limits of gradient flows and to obtain a locally uniform control of these flows in the parameter $x$. This in turn  will allow us to obtain the differentiability of the selection in the parameter $x$ whenever $\phi(x,y)$ is a local minimum.

\subsection{Smoothness of Selections Based on Gradient Flows of a Parametric Morse-Bott Function}\label{sec:smoothness_selections}
We consider a construction for the selection $\phi$ in \cref{def:selection} as a limit of a continuous-time gradient flow of $g$. More precisely, we define a continuous-time trajectory $(\phi_t(x,y))_{t\geq 0}$ in $\Y$ initialized at $\phi_0(x,y)=y$ and driven by the differential equation:
\begin{align}\tag{GF}\label{eq:gradient_flow}
	\frac{d \phi_t(x,y)}{dt} = -\partial_y g(x,\phi_t(x,y)).
\end{align}
Provided $\phi_t(x,y)$  converges towards some element  $\phi(x,y)$ as $t{\rightarrow}{+}\infty$, we can expect such a limit to  satisfy both conditions of \cref{def:selection}, therefore constituting a valid selection.
However, for general non-convex functions, $\phi_t(x,y)$ might not always converge  \cite{Lojasiewicz:1982}. 
To guarantee the existence and convergence of the flow, we make the following assumptions on the function $g$.   
\begin{assump}[\bfseries Smoothness]\label{assumpt:smootness} There exists $L{>}0$ such that $y{\mapsto}\partial_y g(x,y)$ is $L$-Lipschitz for any $x{\in} \X$.
\end{assump}
\begin{assump}[\bfseries Coercivity]\label{assumpt:Coercivity}
For any $x\in \X$, it holds that $g(x,y)\rightarrow +\infty$ as $\Verts{ y}\rightarrow +\infty$. 
\end{assump} 
The smoothness assumption in \cref{assumpt:smootness} is standard and guarantees the existence of the flow by the Cauchy-Lipschitz theorem. 
The coercivity condition in \cref{assumpt:Coercivity} guarantees that $\phi_t(x,y)$ cannot escape to infinity. It can be easily enforced by adding a small $\ell_2$-penalty to a non-negative loss (such as cross-entropy or mean-squared loss) which is already a common practice in machine learning. 
These assumptions, along with \cref{assumpt:morse-bott}  
ensure that the limit $\phi(x,y)$ always exists as we summarize in the following proposition, which is proven in \cref{sec:asymptotic_properties_flow}.
\begin{prop}\label{prop:well_defined_limit}
	Under \cref{assumpt:smootness,assumpt:morse-bott,assumpt:Coercivity}, and for any $(x,y)\in \X\times \Y$,  the gradient flow \cref{eq:gradient_flow} always converges towards a critical point $\phi(x,y)$ of $y\mapsto g(x,y)$ and the map $(x,y)\mapsto \phi(x,y)$ is a selection map as defined in \cref{def:selection}.  We call  $\phi$ the \emph{flow selection} relatively to $g$.   
\end{prop}
\cref{prop:well_defined_limit} is a consequence of a general result that holds for functions satisfying a {\L}ojasiewicz gradient inequality \citep{Attouch:2013,Merlet:2013} which is the case here by \cref{prop:uniform_KL}. From now on, we restrict our attention to the selection $\phi$ defined in \cref{prop:well_defined_limit}. 
Even though $\phi$ satisfies the implicit equation $\partial_y g(x,\phi(x,y))=0$, we cannot rely anymore on the implicit function theorem for studying the differentiability of $\phi(x,y)$ in $x$ since $g$ can have degenerate critical points. Instead, we propose to characterize the differentiability of $\phi$ by studying the limit of  $U_t(x,y):= \partial_x \phi_t(x,y)$ which is formally driven by a linear differential equation of the form:
\begin{align}\label{eq:dynamics_U}
   -\frac{dU_t(x,y)}{dt} = \partial^2_{xy} g(x,\phi_t(x,y)) + U_t(x,y)\partial^2_{yy} g(x,\phi_t(x,y)). 
\end{align}
Had we known in advance that $\phi(x,y)$ is differentiable in $x$, the limit $U_{\infty}(x,y)$ of $U_t(x,y)$ as $t\rightarrow +\infty$, whenever defined, would be a promising candidate for the differential of $\phi(x,y)$ in $x$. Such a limit is indeed expected to satisfy the following linear equation:
\begin{align}\label{eq:limit_U}
   0= \partial^2_{xy} g(x,\phi(x,y)) + U_{\infty}(x,y)\partial^2_{y y} g(x,\phi(x,y)).
\end{align}
A first challenge is to ensure that $U_t$ does not diverge. For critical points $\phi(x,y)$ that are not local minima, it is easy to see that the Hessian $\partial^2_{yy} g(x,\phi_t(x,y))$ must have a negative eigenvalue for $t$ large enough, therefore causing the system \cref{eq:dynamics_U} to diverge. 
Intuitively, unless $\phi(x,y)$ is a local minimum, there is no reason to expect $\phi(x,y)$ to be differentiable or even continuous in $x$, simply  because $\phi(x,y)$ would be an unstable fixed-point of the flow $\phi_t(x,y)$, so that any change in $x$ might cause a large variation in $\phi(x,y)$. The possible non-differentiability of $\phi(x,y)$ for critical points that are not local minima is not problematic in practice, since for almost all initial conditions $y$ of the flow $\phi_t(x,y)$, the limit $\phi(x,y)$ is guaranteed to be a local minimizer~\citep{Panageas:2016}. In addition, we show in \cref{prop:stability_of_local_minima} of \cref{sec:continuity_flow_selection} that if $\phi(x_0,y)$ is a local minimum, then $\phi(x,y)$ must also be a local minimum in a neighborhood of $x_0$. 

Nevertheless, even for local minima, if the Hessian $\partial^2_{yy} g(x,\phi(x,y))$ is non-invertible, \cref{eq:limit_U} might never hold if $\partial^2_{xy} g(x,\phi(x,y))$ does not belong to the image of the Hessian. However, we show in \cref{prop:exact_least_square_solutions}  of \cref{sec:properties_parameteric_morse_bott} that, for any pair $(x,y)$ of critical points, 
$\partial^2_{xy} g(x,y)$ must always belong to the span of the Hessian $\partial^2_{yy} g(x,y)$ as soon as $g$ satisfies \cref{assumpt:morse-bott}, therefore ensuring that \cref{eq:limit_U} admits a solution. The following theorem, which is proven in \cref{sec:differentiability-selection}, establishes the differentiability of $\phi$ at local minima and shows that $\partial_x \phi$ is exactly given by the limit $U_{\infty}$.
\begin{thm}[\bfseries Differentiability of the flow selection]\label{prop:diff_flow_selection}
	Let $g$ be a function satisfying \cref{assumpt:Coercivity,assumpt:morse-bott,assumpt:smootness} so that the flow selection $\phi$ is well-defined.  Let $(x_0,y_0)$ be in $\X{\times}\Y$.	 If $\phi(x_0,y_0)$ is a local minimizer of $y\mapsto g(x_0,y)$, then there exists a neighborhood $\mathcal{U}$ of $x_0$ on which $x\mapsto \phi(x,y_0)$ is differentiable with differential $\partial_x \phi(x,y_0) {=} U_{\infty}(x,y_0)$. Moreover, if $y_0$ is a local minimizer of $y{\mapsto} g(x_0,y)$, then, denoting by $\dagger$ the pseudo inverse operator, $\partial_x\phi(x_0,y_0)$ is exactly given by:
\begin{align}\label{eq:gradient_phi}
   \partial_x\phi(x_0,y_0) {=} - \partial_{xy} g(x_0,y_0)\parens{\partial_{yy} g(x_0,y_0)}^{\dagger}.
\end{align}
\end{thm}
The expression in \cref{eq:gradient_phi} is very similar to the one that would arise by application of the implicit function theorem to a strongly convex function $g$. However, the proof technique does not rely on such a theorem which would not be applicable here. The key technical challenges in proving the above result are: (i) showing that $\phi(x,y)$ must be continuous at~$x_0$ and (ii) controlling the error $\Verts{U_t(x,y)-U_{\infty}(x,y)}$ locally uniformly in $x$. The result follows by the application of classical uniform convergence results \cite[Theorem 7.17]{Rudin:1976}. The continuity of~$\phi$ is established in \cref{prop:continuity_at_point} of \cref{sec:continuity_flow_selection} and relies on a stability analysis of the flow $\phi_t$ performed in \cref{sec:stability_gradient_flow}. The uniform convergence of $U_t$ towards $U_{\infty}$ is shown in \cref{prop:convergence_U} of \cref{sec:differentiability-selection}  and relies on a local uniform convergence of the flow $\phi_t$ towards $\phi$ which is proven in \cref{prop:unif_convergence_flow} of \cref{sec:unif_convergence_flow}. 
It is worth noting that, even though we identified $\partial_x \phi$ to be~$U_{\infty}$, the latter is not fully characterized by \cref{eq:limit_U} as it might contain a non-zero component in the null-space of the Hessian. However, when $(x_0,y_0)$ is an augmented critical pair of $g$, such a component vanishes, and $\partial_x\phi(x_0,y_0)$ is exactly determined by the minimal norm solution in \cref{eq:gradient_phi}. The latter fact has practical implications when designing algorithms for solving \cref{eq:ABG} as we discuss next. 

\section{Algorithms}\label{sec:algorithms}

\subsection{Unrolled Optimization for BGS}
Unrolled optimization constructs a map $\varphi_T(x,y)$ approximating a critical point of the function $y\mapsto g(x,y)$ for any fixed $x$ by applying a finite number $T>0$ of gradient updates starting from some initial condition $y$. By convention, we set $\varphi_0(x,y) {=} y$. Hence, $\varphi_T$ can be understood as an approximation to the selection map defined in \cref{sec:smoothness_selections}. We emphasize that $\varphi_T$ is not a selection (\cref{def:selection}) since $\varphi_T(x,y)$ is not a critical point of $g$ in general. Nevertheless, it provides a tractable approximation to critical points which is key for constructing practical algorithms for bilevel optimization. The gradient of $\varphi_T(x,y)$ w.r.t. $x$ is then obtained by differentiating through the optimization steps and used to optimize the approximate upper-level objective:
\begin{align}
	\mathcal{L}_T(x,y) := f(x,\varphi_T(x,y)). 
\end{align}
Given the $k$-th upper-level iterate $x_k$ and an initial condition  $\tilde{y}_{k}$ for the unrolled optimization, these approaches compute an approximation $y_k {=} \varphi_T(\xmone,\tilde{y}_k)$ and find an update direction $d_k$ for the upper-level variable $x$ by differentiating $\mathcal{L}_T(x,\tilde{y}_k)$ in $x$ at the current iterate $\xmone$. The following iterate $x_{k}$ is obtained by applying an update procedure, such as $x_{k} {=} \xmone {-} \gamma d_k$ for positive small enough step-size $\gamma$. In \cref{alg:abg_alg}, we present several variants of these schemes, including a simple correction allowing them to solve~\cref{eq:ABG} instead of an approximation.

\begin{minipage}{0.38\linewidth}
\vspace{.3cm}
The initial condition $\tilde{y}_k$ is often computed using a  warm-start procedure $\tilde{y}_{k} {=} \mathcal{I}_M(\xmone,\ymone)$. The simplest procedure is to set $\tilde{y}_{k} {=} \ymone$ in which case $\mathcal{I}_0(x,y){=}y$. However, it is not uncommon to perform $M{>}0$ optimization steps to minimize the objective $y{\mapsto} g(\xmone,y)$ starting from $\ymone$. By doing so, gradient unrolling stops at $\tilde{y}_k$ and ignores the dependence of $\tilde{y}_k$ on $\ymone$, resulting in Truncated unrolled optimization \citep{Shaban:2019}. \cref{alg:abg_alg} summarizes these approaches when the binary variable  {\bf{AddCorrection}} is set to {\bf False}. 
To characterize the limit points of \cref{alg:abg_alg}, we make the following  assumptions on  $\mathcal{I}_M$, $\mathcal{\varphi}_T$.  
\end{minipage}
\hspace{0.2cm}
\begin{minipage}{0.55\linewidth}
	\begin{algorithm}[H]
\caption{BGS-Opt$(x_0,y_0)$}\label{alg:abg_alg}
	\begin{algorithmic}[1]
		\STATE Inputs: $x_0$, $y_0$, 
		\STATE Parameters:  $K$, $T$, $M$ , $\gamma$ {\bf {AddCorrection}} 		
		\FOR{ $k \in \{1,...,K+1\}$ }
                        \STATE $\tilde{y}_k \leftarrow  \mathcal{I}_M\parens{\xmone,\ymone}$. {\small\textcolor{Blue}{\# Warm-start.}}
                        \STATE $y_{k} {\leftarrow} \varphi_T\parens{\xmone, \tilde{y}_k}$ {\small\textcolor{Blue}{\# Unrolled optimization.}}
                        \STATE $d_k \leftarrow \partial_x \mathcal{L}_{T}\parens{\xmone, \tilde{y}_k}$
			\IF{{\bf {AddCorrection}$=$ True} }
			\STATE $v_k\leftarrow  \partial_y \mathcal{L}_{T}\parens{\xmone, \tilde{y}_k}$
			\STATE $\xi_k{\approx} -\parens{\partial_{yy}g(\xmone,y_k)}^{\dagger}v_k$ {\small\textcolor{Blue}{\# Approx. solver}}
			\STATE $d_k\leftarrow d_k + \textcolor{black}{\partial_{xy}g(\xmone,y_k)\xi_k}$ {\small\textcolor{Blue}{\# Grad. correction}}
			\ENDIF
			\STATE $x_{k} \leftarrow  \xmone -\gamma  d_{k}$ {\small\textcolor{Blue}{\# Updating $x$}} 
		\ENDFOR
		\STATE Return $(x_{K}, y_K)$.
	\end{algorithmic}
\end{algorithm}
\end{minipage}

\begin{assump}\label{assumpt:map_varphi}
For any non-negative integers $M,T\geq 0$, the maps $\mathcal{I}_M$  and $\mathcal{\varphi}_T$ are continuous on $\X\times \Y$ and take values in $\Y$, with $\varphi_T$ being continuously differentiable. 
		 Moreover, for any $(x,y)\in \X\times \Y$ s.t. $\partial_y g(x,y){=}0$ and $M,T\geq 0$, there exists a matrix $D$ such that:
	\begin{align}
           \mathcal{I}_M(x,y)=\varphi_T(x,y) = y,\quad  \partial_x \varphi_T(x,y) = \partial^2_{xy} g(x,y)D,\quad 
                \partial_y \varphi_T(x,y) = I+\partial^2_{yy} g(x,y)D. 
	\end{align}
	Finally, for any $(x,y){\in} \X{\times} \Y$, and $M,T\geq 0$ s.t. $T+M>0$, the equality $y {=} \varphi_T(x,\mathcal{I}_M(x,y))$ implies that $y$ is a critical point of $g$, i.e. $\partial_y g(x,y)=0$.
\end{assump}
\begin{assump}\label{assumpt:map_convergence_to_selection}
$\varphi_T$ converges to a selection $\phi$ 
 and $\partial_x\varphi_T$ converges uniformly near local minima.
\end{assump}
	\cref{assumpt:map_varphi} is satisfied by many mappings used in practice such as $T$-steps of the gradient descent or proximal point algorithms,  whenever $g$ is twice-continuousely differentiable and $L$-smooth as shown in \cref{prop:properties_varphi} of \cref{sec:proof_algorithms}. \cref{assumpt:map_convergence_to_selection} is a discrete-time version of the uniform convergence result in \cref{prop:convergence_U} of \cref{sec:differentiability-selection} but that we directly assume here for simplicity.  
	Under these assumptions we show that \cref{alg:abg_alg} can find equilibria of \cref{eq:ABG} up to an approximation error resulting from the fact that $\varphi_T$ is not an exact selection.
\begin{prop}\label{prop:critical_point_no_correction}
Let $M,T$ be non-negative numbers s.t. $M+T>0$ and let  $(x_k,y_k)$ be the iterates of \cref{alg:abg_alg} using the maps $\mathcal{I}_M$ and $\varphi_T$ and without any correction, i.e.  {\bf {AddCorrection}${=}$False}. If $(x_k,y_k)$ converges to a limit point $(x_T^{\star},y_T^{\star})$ then, under \cref{assumpt:map_varphi}:
	\begin{align}
		\partial_x \mathcal{L}_{T}(x_T^{\star},y_T^{\star})=0,\qquad \partial_y g(x_T^{\star},y_T^{\star})=0.
	\end{align}
Let $E$ be the set of limit points $(x_T^{\star},y_T^{\star})$  for $T\geq 0$. If $E$ is bounded and $y_T^{\star}$ is a local minimum of $g(x_T^{\star},.)$ for any $T\geq 0$, then, under \cref{assumpt:map_varphi,assumpt:map_convergence_to_selection}, the elements of $E$ are approximate equilibria for \cref{eq:ABG}:
\begin{align}
	\lim\sup_{T}\Verts{\partial_x\mathcal{L}_{\phi}(x_T^{\star},y_T^{\star})}=0, \qquad \partial_y g(x_T^{\star},y_T^{\star})=0,\quad (\forall T>0). 
\end{align}
\end{prop}
\cref{prop:critical_point_no_correction} shows that unrolled optimization algorithms approximately solve \cref{eq:ABG} in the limit where the number of unrolling steps $T$ of the $\varphi_T$ goes to infinity. This result is consistent with the ones obtained in \citep{Ji:2021a} for the case where $g$ is strongly convex and illustrates the high computational cost for solving \cref{eq:ABG} without correcting for the bias introduced by unrolling. Next, we show how to get rid of such a bias in light of \cref{prop:diff_flow_selection}. 
  
\subsection{Implicit Gradient Correction}
We propose to correct the bias of unrolling by exploiting the expression of the gradient $\partial_x\phi$ provided in \cref{prop:diff_flow_selection}.
The key idea is to obtain an expression for $\partial_x\mathcal{L}_{\phi}(x,y)$ in terms of $\mathcal{L}_{T}$ and the second-order derivatives of $g$ which holds for any local minimizer $y$ of $y\mapsto g(x,y)$ as shown by the proposition below.
\begin{prop}\label{prop:gradient_correction}
	Let $\phi$ be the selection defined in \cref{sec:smoothness_selections} and $(x,y)\in \X\times \Y$ be s.t. $y$ is a local minimum of $y\mapsto g(x,y)$. Then, under \cref{assumpt:Coercivity,assumpt:map_varphi,assumpt:morse-bott,assumpt:smootness}, $\partial_x\mathcal{L}_{\phi}(x,y)$ is given by the equation:
\begin{align}\label{eq:expression_grad_L}
   \partial_x\mathcal{L}_{\phi}(x,y) := \partial_x\mathcal{L}_T(x,y) - \partial^2_{xy} g(x,y)\parens{\partial^2_{yy} g(x,y)}^{\dagger}\partial_y \mathcal{L}_T(x,y).
\end{align}
\end{prop}
\cref{prop:gradient_correction}, which is proven in  \cref{sec:proof_algorithms}, suggests a simple correction for the gradient estimate $d_k$ in \cref{alg:abg_alg}. 
By doing so, the corrected algorithm would be performing an approximate gradient descent on each of the upper-level and lower-level objectives, suggesting that the algorithm may recover equilibrium points of \cref{eq:ABG} without having to increase the computation budget for the unrolling as we show later in \cref{prop:critical_point_correction}. 
A simple way to proceed would to compute  $c_k$ satisfying the approximate equation $c_k {\approx} -B_k(A_k)^{\dagger}v_k$, where $A_k {=} \partial^2_{yy} g(\xmone,y_k)$, $B_k {:=}\partial^2_{xy} g(\xmone,y_k) $ and $v_k {=} \partial_y \mathcal{L}_T(\xmone,\tilde{y}_k)$.  
More concretely, $c_k$ can be computed by setting $c_k {=} B_k\xi_k$ where $\xi_k$ approximates the minimum norm solution to the least squares problem:
\begin{align}\label{eq:min_norm_least_square}
	 \xi_k \approx \arg\min_{\xi}\Verts{\xi}^2, \qquad s.t. \quad  \xi\in \arg\min_{\xi}\Verts{ A_k \xi + v_k}^2, 
\end{align}
{\bf Approximate solution to \cref{eq:min_norm_least_square}.} 
It is possible to solve \cref{eq:min_norm_least_square} approximately using an iterative procedure by constructing $N$ iterates $\xi^t$ starting from $\xi^0 =0$ and performing (conjugate) gradient descent on the quadratic objective. 
This can be implemented efficiently using only Hessian vector products with the Hessian $A_k$ \citep{Lorraine:2020}. 
The constrained problem \cref{eq:min_norm_least_square} can also be expressed as an unconstrained one by re-parametrizing $\xi = A_k z$:
\begin{align}\label{eq:reparametrization}
	\xi_k \approx A_kz_k^{\star}, \qquad s.t.\quad z_k^{\star}\in \arg\min_{z}\Verts{A_k^2z+v_k}^2.
\end{align}
Eq.~\cref{eq:reparametrization} has the advantage that $z_k^{\star}$ solves an unconstrained problem. As such,  it is more amenable to applying a warm-start strategy, which can yield efficient approximation $z_k$ to $z_k^{\star}$ by exploiting previously computed approximation $z_{k-1}$ to $z_{k-1}^{\star}$ ~\citep{Arbel:2021a}. This strategy can be achieved using a standard iterative algorithm $\mathcal{P}$ for approximately solving the least-squares problems, such as a fixed number of conjugate gradient iterations, that takes as input the matrix $A_k$, vector $v_k$ and initialization $z_{k-1}\approx z_{k-1}^{\star}$ and returns the next iterate $z_k\approx z_k^{\star}$. 
More formally we view $\mathcal{P}$ as a continuous map of $(A,v,z)\mapsto\mathcal{P}(A,v,z)$ returning a vector $z'$ and such that the only fixed points are exact solutions to the least square problem $\min_{z}\Verts{A^2z+v}^2$. We refer to \cref{sec:warm-start} for examples of such maps. We can then define the iterates $z_k$ and  $\xi_k$ as follows:
\begin{align}\label{eq:warm_start}
	\xi_k = A_kz_k, \qquad z_k = \mathcal{P}(A_k,v_k,z_{k-1}).
\end{align}
The corrected algorithm is obtained by setting the variable {\bf{AddCorrection}${=}$True} in \cref{alg:abg_alg} and computing the $\xi_k$ using any approximate solver including, in particular, the ones based on a warm-start strategy as in \cref{eq:warm_start}. The following proposition, with proof in \cref{sec:proof_algorithms}, shows that the proposed correction indeed yields equilibrium points of \cref{eq:ABG}.  
\begin{prop}\label{prop:critical_point_correction}
	Let $(x_k,y_k)$ be the iterates obtained using \cref{alg:abg_alg} with {\bf{AddCorrection}${=}$True} and $T+M>0$ and assume that $\xi_k$ are computed using \cref{eq:warm_start}. If $(x_k,y_k,z_k)_{k\geq 0}$  converges to a limit point $(x^{\star},y^{\star},z^{\star})$, then $y^{\star}$ is a critical point of $y\mapsto g(x^{\star},y)$ and if, in addition, $y^{\star}$ is a local minimizer, then $(x^*,y^*)$ must be an equilibrium of \cref{eq:ABG} satisfying  \cref{eq:SC}:
	\begin{align}
           \partial_x \mathcal{L}_{\mathcal{\phi}}(x^{\star},y^{\star})=0~~~~\text{and}~~~~ \partial_y g(x^{\star},y^{\star})=0
	\end{align}
\end{prop}
\cref{prop:critical_point_correction} shows that the proposed correction allows to recover equilibria of \cref{eq:ABG} without having to increase the number of iterations $T$ of the unrolled algorithm. This is by contrast with \cref{prop:critical_point_no_correction} where $T$ must increase to infinity, which would be impractical. 
We discuss in \cref{sec:recovering_existing_alg}  how different choices for the parameters $T$ and $M$ recover known algorithms. In particular, that \cref{alg:abg_alg} with correction allows interpolating between two families of algorithms: (ITD) and (AID) while still recovering the correct equilibria. Numerical results illustrating the benefits of the correction are presented in \cref{sec:experiments}. 

\section{Discussion}
We have introduced a bilevel game that resolves the ambiguity in bilevel
optimization with non-convex objectives using the notion of selection maps. We have shown that many
algorithms for bilevel optimization approximately solve these games up to a
bias due to finite computational power.  Our study of the differentiability
properties of the selection maps has resulted in practical procedures for correcting
such a bias and required the development of new analytical tools. This study
opens the way for several avenues of research to understand the tradeoff
between unrolling and implicit gradient correction for designing efficient
algorithms. In future work, studying these algorithms in a non-smooth and stochastic setting
would also be of great theoretical and practical interest.

\paragraph*{Funding}
This project was supported by ANR 3IA MIAI@Grenoble Alpes (ANR-19-P3IA-0003).

\bibliographystyle{plainnat}   
\bibliography{extracted_bib}   

\begin{thebibliography}{58}
\providecommand{\natexlab}[1]{#1}
\providecommand{\url}[1]{\texttt{#1}}
\expandafter\ifx\csname urlstyle\endcsname\relax
  \providecommand{\doi}[1]{doi: #1}\else
  \providecommand{\doi}{doi: \begingroup \urlstyle{rm}\Url}\fi

\bibitem[Ablin et~al.(2020)Ablin, Peyr{\'e}, and Moreau]{Ablin:2020}
Pierre Ablin, Gabriel Peyr{\'e}, and Thomas Moreau.
\newblock Super-efficiency of automatic differentiation for functions defined
  as a minimum.
\newblock In \emph{International Conference on Machine Learning}, pages 32--41.
  PMLR, 2020.

\bibitem[Arbel and Mairal(2021)]{Arbel:2021a}
Michael Arbel and Julien Mairal.
\newblock {Amortized implicit differentiation for stochastic bilevel
  optimization}.
\newblock working paper or preprint, November 2021.
\newblock URL \url{https://hal.archives-ouvertes.fr/hal-03455458}.

\bibitem[Attouch et~al.(2013)Attouch, Bolte, and Svaiter]{Attouch:2013}
Hedy Attouch, J{\'e}r{\^o}me Bolte, and Benar~Fux Svaiter.
\newblock Convergence of descent methods for semi-algebraic and tame problems:
  proximal algorithms, forward--backward splitting, and regularized
  gauss--seidel methods.
\newblock \emph{Mathematical Programming}, 137\penalty0 (1):\penalty0 91--129,
  2013.

\bibitem[Austin and Braam(1995)]{Austin:1995}
David~M Austin and Peter~J Braam.
\newblock Morse-bott theory and equivariant cohomology.
\newblock In \emph{The Floer memorial volume}, pages 123--183. Springer, 1995.

\bibitem[Baydin et~al.(2018)Baydin, Pearlmutter, Radul, and
  Siskind]{Baydin:2018}
Atilim~Gunes Baydin, Barak~A Pearlmutter, Alexey~Andreyevich Radul, and
  Jeffrey~Mark Siskind.
\newblock Automatic differentiation in machine learning: a survey.
\newblock \emph{Journal of machine learning research}, 18, 2018.

\bibitem[Bertinetto et~al.(2018)Bertinetto, Henriques, Torr, and
  Vedaldi]{Bertinetto:2018}
Luca Bertinetto, Joao~F Henriques, Philip~HS Torr, and Andrea Vedaldi.
\newblock Meta-learning with differentiable closed-form solvers.
\newblock \emph{arXiv preprint arXiv:1805.08136}, 2018.

\bibitem[Blondel et~al.(2021)Blondel, Berthet, Cuturi, Frostig, Hoyer,
  Llinares-L{\'o}pez, Pedregosa, and Vert]{Blondel:2021}
Mathieu Blondel, Quentin Berthet, Marco Cuturi, Roy Frostig, Stephan Hoyer,
  Felipe Llinares-L{\'o}pez, Fabian Pedregosa, and Jean-Philippe Vert.
\newblock Efficient and modular implicit differentiation.
\newblock \emph{arXiv preprint arXiv:2105.15183}, 2021.

\bibitem[Bolte et~al.(2021)Bolte, Le, Pauwels, and Silveti-Falls]{Bolte:2021}
J{\'e}r{\^o}me Bolte, Tam Le, Edouard Pauwels, and Tony Silveti-Falls.
\newblock Nonsmooth implicit differentiation for machine-learning and
  optimization.
\newblock \emph{Advances in neural information processing systems},
  34:\penalty0 13537--13549, 2021.

\bibitem[Bolte et~al.(2022{\natexlab{a}})Bolte, Boustany, Pauwels, and
  Pesquet-Popescu]{Bolte:2022a}
J{\'e}r{\^o}me Bolte, Ryan Boustany, Edouard Pauwels, and B{\'e}atrice
  Pesquet-Popescu.
\newblock Nonsmooth automatic differentiation: a cheap gradient principle and
  other complexity results.
\newblock \emph{arXiv preprint arXiv:2206.01730}, 2022{\natexlab{a}}.

\bibitem[Bolte et~al.(2022{\natexlab{b}})Bolte, Pauwels, and
  Vaiter]{Bolte:2022}
J{\'e}r{\^o}me Bolte, Edouard Pauwels, and Samuel Vaiter.
\newblock Automatic differentiation of nonsmooth iterative algorithms.
\newblock \emph{arXiv preprint arXiv:2206.00457}, 2022{\natexlab{b}}.

\bibitem[Cohen(1991)]{Cohen:1991}
Ralph~L Cohen.
\newblock \emph{Topics in Morse theory}.
\newblock Stanford University Department of Mathematics, 1991.

\bibitem[Daneri and Savar{\'e}(2010)]{Daneri:2010}
Sara Daneri and Giuseppe Savar{\'e}.
\newblock Lecture notes on gradient flows and optimal transport.
\newblock \emph{arXiv preprint arXiv:1009.3737}, 2010.

\bibitem[Dempe et~al.(2007)Dempe, Dutta, and Mordukhovich]{Dempe:2007}
S~Dempe, J~Dutta, and BS~Mordukhovich.
\newblock New necessary optimality conditions in optimistic bilevel
  programming.
\newblock \emph{Optimization}, 56\penalty0 (5-6):\penalty0 577--604, 2007.

\bibitem[Domke(2012)]{Domke:2012}
Justin Domke.
\newblock Generic methods for optimization-based modeling.
\newblock In \emph{Artificial Intelligence and Statistics}, pages 318--326.
  PMLR, 2012.

\bibitem[Draxler et~al.(2018)Draxler, Veschgini, Salmhofer, and
  Hamprecht]{Draxler:2018}
Felix Draxler, Kambis Veschgini, Manfred Salmhofer, and Fred Hamprecht.
\newblock Essentially no barriers in neural network energy landscape.
\newblock In \emph{International conference on machine learning}, pages
  1309--1318. PMLR, 2018.

\bibitem[Feehan(2020)]{Feehan:2020a}
Paul Feehan.
\newblock On the morse--bott property of analytic functions on banach spaces
  with {\l}ojasiewicz exponent one half.
\newblock \emph{Calculus of Variations and Partial Differential Equations},
  59\penalty0 (2):\penalty0 1--50, 2020.

\bibitem[Feurer and Hutter(2019)]{Feurer:2019}
Matthias Feurer and Frank Hutter.
\newblock Hyperparameter optimization.
\newblock In \emph{Automated machine learning}, pages 3--33. Springer, Cham,
  2019.

\bibitem[Franceschi et~al.(2018)Franceschi, Frasconi, Salzo, Grazzi, and
  Pontil]{Franceschi:2018}
Luca Franceschi, Paolo Frasconi, Saverio Salzo, Riccardo Grazzi, and
  Massimiliano Pontil.
\newblock Bilevel programming for hyperparameter optimization and
  meta-learning.
\newblock In \emph{International Conference on Machine Learning}, pages
  1568--1577. PMLR, 2018.

\bibitem[Ghadimi and Wang(2018)]{Ghadimi:2018}
Saeed Ghadimi and Mengdi Wang.
\newblock Approximation methods for bilevel programming.
\newblock \emph{arXiv preprint arXiv:1802.02246}, 2018.

\bibitem[Gould et~al.(2016)Gould, Fernando, Cherian, Anderson, Cruz, and
  Guo]{Gould:2016}
Stephen Gould, Basura Fernando, Anoop Cherian, Peter Anderson, Rodrigo~Santa
  Cruz, and Edison Guo.
\newblock On differentiating parameterized argmin and argmax problems with
  application to bi-level optimization.
\newblock \emph{arXiv preprint arXiv:1607.05447}, 2016.

\bibitem[Guo et~al.(2015)Guo, Lin, and Ye]{Guo:2015}
Lei Guo, Gui-Hua Lin, and Jane~J Ye.
\newblock Solving mathematical programs with equilibrium constraints.
\newblock \emph{Journal of Optimization Theory and Applications}, 166\penalty0
  (1):\penalty0 234--256, 2015.

\bibitem[He et~al.(2015)He, Zhang, Ren, and Sun]{He:2015}
Kaiming He, Xiangyu Zhang, Shaoqing Ren, and Jian Sun.
\newblock Deep {Residual} {Learning} for {Image} {Recognition}.
\newblock \emph{2016 IEEE Conference on Computer Vision and Pattern Recognition
  (CVPR)}, pages 770--778, 2015.
\newblock \doi{10.1109/CVPR.2016.90}.

\bibitem[Hong et~al.(2020)Hong, Wai, Wang, and Yang]{Hong:2020a}
Mingyi Hong, Hoi-To Wai, Zhaoran Wang, and Zhuoran Yang.
\newblock A two-timescale framework for bilevel optimization: Complexity
  analysis and application to actor-critic.
\newblock \emph{arXiv preprint arXiv:2007.05170}, 2020.

\bibitem[Ji and Liang(2021)]{Ji:2021}
Kaiyi Ji and Yingbin Liang.
\newblock Lower bounds and accelerated algorithms for bilevel optimization.
\newblock \emph{arXiv preprint arXiv:2102.03926}, 2021.

\bibitem[Ji et~al.(2021)Ji, Yang, and Liang]{Ji:2021a}
Kaiyi Ji, Junjie Yang, and Yingbin Liang.
\newblock Bilevel optimization: Convergence analysis and enhanced design.
\newblock In \emph{International Conference on Machine Learning}, pages
  4882--4892. PMLR, 2021.

\bibitem[Kingma and Ba(2015)]{Kingma:2014}
Diederik~P. Kingma and Jimmy Ba.
\newblock Adam: {A} method for stochastic optimization.
\newblock In Yoshua Bengio and Yann LeCun, editors, \emph{3rd International
  Conference on Learning Representations, {ICLR} 2015, San Diego, CA, USA, May
  7-9, 2015, Conference Track Proceedings}, 2015.

\bibitem[Krizhevsky et~al.(2009)Krizhevsky, Hinton, et~al.]{Krizhevsky:2009}
Alex Krizhevsky, Geoffrey Hinton, et~al.
\newblock Learning multiple layers of features from tiny images.
\newblock 2009.

\bibitem[Lang(2012)]{Lang:2012}
Serge Lang.
\newblock \emph{Fundamentals of differential geometry}, volume 191.
\newblock Springer Science \& Business Media, 2012.

\bibitem[Lee(2003)]{Lee:2003}
John~M. Lee.
\newblock \emph{Introduction to {Smooth} {Manifolds}}.
\newblock Springer Science \& Business Media, 2003.
\newblock ISBN 978-0-387-95448-6.
\newblock Google-Books-ID: eqfgZtjQceYC.

\bibitem[Li et~al.(2018)Li, Xu, Taylor, Studer, and Goldstein]{Li:2018d}
Hao Li, Zheng Xu, Gavin Taylor, Christoph Studer, and Tom Goldstein.
\newblock Visualizing the loss landscape of neural nets.
\newblock \emph{Advances in neural information processing systems}, 31, 2018.

\bibitem[Li et~al.(2019)Li, Liu, and Xue]{Li:2019d}
Zhuchun Li, Yi~Liu, and Xiaoping Xue.
\newblock Convergence and stability of generalized gradient systems by
  {\l}ojasiewicz inequality with application in continuum kuramoto model.
\newblock \emph{Discrete \& Continuous Dynamical Systems}, 39\penalty0
  (1):\penalty0 345, 2019.

\bibitem[Liao et~al.(2018)Liao, Xiong, Fetaya, Zhang, Yoon, Pitkow, Urtasun,
  and Zemel]{Liao:2018}
Renjie Liao, Yuwen Xiong, Ethan Fetaya, Lisa Zhang, KiJung Yoon, Xaq Pitkow,
  Raquel Urtasun, and Richard Zemel.
\newblock Reviving and improving recurrent back-propagation.
\newblock In \emph{International Conference on Machine Learning}, pages
  3082--3091. PMLR, 2018.

\bibitem[Liu et~al.(2021{\natexlab{a}})Liu, Gao, Zhang, Meng, and
  Lin]{Liu:2021}
Risheng Liu, Jiaxin Gao, Jin Zhang, Deyu Meng, and Zhouchen Lin.
\newblock Investigating bi-level optimization for learning and vision from a
  unified perspective: A survey and beyond.
\newblock \emph{arXiv preprint arXiv:2101.11517}, 2021{\natexlab{a}}.

\bibitem[Liu et~al.(2021{\natexlab{b}})Liu, Liu, Zeng, Zhang, and
  Zhang]{Liu:2021f}
Risheng Liu, Xuan Liu, Shangzhi Zeng, Jin Zhang, and Yixuan Zhang.
\newblock Value-function-based sequential minimization for bi-level
  optimization.
\newblock \emph{arXiv preprint arXiv:2110.04974}, 2021{\natexlab{b}}.

\bibitem[Liu et~al.(2021{\natexlab{c}})Liu, Liu, Zeng, and Zhang]{Liu:2021b}
Risheng Liu, Yaohua Liu, Shangzhi Zeng, and Jin Zhang.
\newblock Towards gradient-based bilevel optimization with non-convex followers
  and beyond.
\newblock \emph{Advances in Neural Information Processing Systems}, 34,
  2021{\natexlab{c}}.

\bibitem[Lojasiewicz(1982)]{Lojasiewicz:1982}
Stanislaw Lojasiewicz.
\newblock Sur les trajectoires du gradient d'une fonction analytique.
\newblock \emph{Seminari di geometria}, 1983:\penalty0 115--117, 1982.

\bibitem[Lorraine et~al.(2020)Lorraine, Vicol, and Duvenaud]{Lorraine:2020}
Jonathan Lorraine, Paul Vicol, and David Duvenaud.
\newblock Optimizing millions of hyperparameters by implicit differentiation.
\newblock In \emph{International Conference on Artificial Intelligence and
  Statistics}, pages 1540--1552. PMLR, 2020.

\bibitem[Mairal et~al.(2011)Mairal, Bach, and Ponce]{Mairal:2011}
Julien Mairal, Francis Bach, and Jean Ponce.
\newblock Task-driven dictionary learning.
\newblock \emph{IEEE transactions on pattern analysis and machine
  intelligence}, 34\penalty0 (4):\penalty0 791--804, 2011.

\bibitem[Mart{\i}nez-Alfaro et~al.(2016)Mart{\i}nez-Alfaro, Meza-Sarmiento, and
  Oliveira]{Martinez-Alfaro:2016}
J~Mart{\i}nez-Alfaro, IS~Meza-Sarmiento, and R~Oliveira.
\newblock Topological classification of simple morse bott functions on
  surfaces.
\newblock \emph{Real and complex singularities}, 675:\penalty0 165--179, 2016.

\bibitem[Merlet and Nguyen(2013)]{Merlet:2013}
Beno{\^\i}t Merlet and Thanh~Nhan Nguyen.
\newblock Convergence to equilibrium for discretizations of gradient-like flows
  on riemannian manifolds.
\newblock \emph{Differential and Integral Equations}, 26\penalty0
  (5/6):\penalty0 571--602, 2013.

\bibitem[Panageas and Piliouras(2016)]{Panageas:2016}
Ioannis Panageas and Georgios Piliouras.
\newblock Gradient descent only converges to minimizers: Non-isolated critical
  points and invariant regions.
\newblock \emph{arXiv preprint arXiv:1605.00405}, 2016.

\bibitem[Pedregosa(2016)]{Pedregosa:2016}
Fabian Pedregosa.
\newblock Hyperparameter optimization with approximate gradient.
\newblock In \emph{International conference on machine learning}, pages
  737--746. PMLR, 2016.

\bibitem[Rajeswaran et~al.(2019)Rajeswaran, Finn, Kakade, and
  Levine]{Rajeswaran:2019}
Aravind Rajeswaran, Chelsea Finn, Sham~M Kakade, and Sergey Levine.
\newblock Meta-{Learning} with {Implicit} {Gradients}.
\newblock In H.~Wallach, H.~Larochelle, A.~Beygelzimer,
  F.~d{\textbackslash}textquotesingle Alch{\'e}-Buc, E.~Fox, and R.~Garnett,
  editors, \emph{Advances in {Neural} {Information} {Processing} {Systems} 32
  (NeurIPS)}. Curran Associates, Inc., 2019.

\bibitem[Rezende and Mohamed(2015)]{Rezende:2015}
Danilo Rezende and Shakir Mohamed.
\newblock Variational inference with normalizing flows.
\newblock In \emph{International conference on machine learning}, pages
  1530--1538. PMLR, 2015.

\bibitem[Robinson(2012)]{Robinson:2012}
Rex~Clark Robinson.
\newblock \emph{An introduction to dynamical systems: continuous and discrete},
  volume~19.
\newblock American Mathematical Soc., 2012.

\bibitem[Rudin et~al.(1976)]{Rudin:1976}
Walter Rudin et~al.
\newblock \emph{Principles of mathematical analysis}, volume~3.
\newblock McGraw-hill New York, 1976.

\bibitem[Shaban et~al.(2019)Shaban, Cheng, Hatch, and Boots]{Shaban:2019}
Amirreza Shaban, Ching-An Cheng, Nathan Hatch, and Byron Boots.
\newblock Truncated back-propagation for bilevel optimization.
\newblock In \emph{The 22nd International Conference on Artificial Intelligence
  and Statistics}, pages 1723--1732. PMLR, 2019.

\bibitem[Singh et~al.(2019)Singh, Sahani, and Gretton]{Singh:2019}
Rahul Singh, Maneesh Sahani, and Arthur Gretton.
\newblock Kernel {Instrumental} {Variable} {Regression}.
\newblock \emph{arXiv:1906.00232 [cs, econ, math, stat]}, June 2019.
\newblock URL \url{http://arxiv.org/abs/1906.00232}.
\newblock arXiv: 1906.00232.

\bibitem[Stackelberg(1934)]{Stackelberg:1934}
H.F.~Von Stackelberg.
\newblock \emph{MarktformundGleichgewicht}.
\newblock Springer, 1934.

\bibitem[Vicol et~al.(2021)Vicol, Lorraine, Duvenaud, and Grosse]{Vicol:2021}
Paul Vicol, Jonathan Lorraine, David Duvenaud, and Roger Grosse.
\newblock Implicit regularization in overparameterized bilevel optimization.
\newblock In \emph{ICML 2021 Beyond First Order Methods Workshop}, 2021.

\bibitem[Wang et~al.(2018{\natexlab{a}})Wang, Zhu, Torralba, and
  Efros]{Wang:2018a}
Tongzhou Wang, Jun-Yan Zhu, Antonio Torralba, and Alexei~A Efros.
\newblock Dataset distillation.
\newblock \emph{arXiv preprint arXiv:1811.10959}, 2018{\natexlab{a}}.

\bibitem[Wang et~al.(2018{\natexlab{b}})Wang, Sun, and Halgamuge]{Wang:2018}
Wei Wang, Yuan Sun, and Saman Halgamuge.
\newblock Improving mmd-gan training with repulsive loss function.
\newblock \emph{arXiv preprint arXiv:1812.09916}, 2018{\natexlab{b}}.

\bibitem[Wiesemann et~al.(2013)Wiesemann, Tsoukalas, Kleniati, and
  Rustem]{Wiesemann:2013}
Wolfram Wiesemann, Angelos Tsoukalas, Polyxeni-Margarita Kleniati, and
  Ber{\c{c}} Rustem.
\newblock Pessimistic bilevel optimization.
\newblock \emph{SIAM Journal on Optimization}, 23\penalty0 (1):\penalty0
  353--380, 2013.

\bibitem[Xu and Ye(2014)]{Xu:2014}
Mengwei Xu and Jane~J Ye.
\newblock A smoothing augmented lagrangian method for solving simple bilevel
  programs.
\newblock \emph{Computational Optimization and Applications}, 59\penalty0
  (1):\penalty0 353--377, 2014.

\bibitem[Ye and Ye(1997)]{Ye:1997}
JJ~Ye and XY~Ye.
\newblock Necessary optimality conditions for optimization problems with
  variational inequality constraints.
\newblock \emph{Mathematics of Operations Research}, 22\penalty0 (4):\penalty0
  977--997, 1997.

\bibitem[Ye and Zhu(1995)]{Ye:1995}
JJ~Ye and DL~Zhu.
\newblock Optimality conditions for bilevel programming problems.
\newblock \emph{Optimization}, 33\penalty0 (1):\penalty0 9--27, 1995.

\bibitem[Ye et~al.(1997)Ye, Zhu, and Zhu]{Ye:1997a}
JJ~Ye, DL~Zhu, and Qiji~Jim Zhu.
\newblock Exact penalization and necessary optimality conditions for
  generalized bilevel programming problems.
\newblock \emph{SIAM Journal on optimization}, 7\penalty0 (2):\penalty0
  481--507, 1997.

\bibitem[Zemkoho(2016)]{Zemkoho:2016}
Alain~B Zemkoho.
\newblock Solving ill-posed bilevel programs.
\newblock \emph{Set-Valued and Variational Analysis}, 24\penalty0 (3):\penalty0
  423--448, 2016.

\end{thebibliography}
\section*{Checklist}

The checklist follows the references.  Please
read the checklist guidelines carefully for information on how to answer these
questions.  For each question, change the default \answerTODO{} to \answerYes{},
\answerNo{}, or \answerNA{}.  You are strongly encouraged to include a {\bf
justification to your answer}, either by referencing the appropriate section of
your paper or providing a brief inline description. 
Please do not modify the questions and only use the provided macros for your
answers.  Note that the Checklist section does not count towards the page
limit.  In your paper, please delete this instructions block and only keep the
Checklist section heading above along with the questions/answers below.

\begin{enumerate}

\item For all authors...
\begin{enumerate}
  \item Do the main claims made in the abstract and introduction accurately reflect the paper's contributions and scope?
    \answerYes{}
  \item Did you describe the limitations of your work?
    \answerYes{}
  \item Did you discuss any potential negative societal impacts of your work?
    \answerNA{}
  \item Have you read the ethics review guidelines and ensured that your paper conforms to them?
    \answerYes{}
   \end{enumerate}

\item If you are including theoretical results...
\begin{enumerate}
  \item Did you state the full set of assumptions of all theoretical results?
    \answerYes{}
     \item Did you include complete proofs of all theoretical results?
    \answerYes{}
\end{enumerate}

\item If you ran experiments...
\begin{enumerate}
  \item Did you include the code, data, and instructions needed to reproduce the main experimental results (either in the supplemental material or as a URL)?
    \answerNA{}
  \item Did you specify all the training details (e.g., data splits, hyperparameters, how they were chosen)?
    \answerNA{} 
           \item Did you report error bars (e.g., with respect to the random seed after running experiments multiple times)?
    \answerNA{}
        \item Did you include the total amount of compute and the type of resources used (e.g., type of GPUs, internal cluster, or cloud provider)?
    \answerNA{}
\end{enumerate}

\item If you are using existing assets (e.g., code, data, models) or curating/releasing new assets...
\begin{enumerate}
  \item If your work uses existing assets, did you cite the creators?
    \answerNA{}
  \item Did you mention the license of the assets?
    \answerNA{}
  \item Did you include any new assets either in the supplemental material or as a URL?
    \answerNA{}
  \item Did you discuss whether and how consent was obtained from people whose data you're using/curating?
    \answerNA{}
  \item Did you discuss whether the data you are using/curating contains personally identifiable information or offensive content?
    \answerNA{}
\end{enumerate}

\item If you used crowdsourcing or conducted research with human subjects...
\begin{enumerate}
  \item Did you include the full text of instructions given to participants and screenshots, if applicable?
    \answerNA{}
  \item Did you describe any potential participant risks, with links to Institutional Review Board (IRB) approvals, if applicable?
    \answerNA{}
      \item Did you include the estimated hourly wage paid to participants and the total amount spent on participant compensation?
    \answerNA{}
\end{enumerate}

\end{enumerate}

\clearpage
\appendix

\section{Morse-Bott Lemma with Parameters}
\subsection{Background on Morse-Bott Functions}\label{sec:Morse-Bott-functions}
We recall the definition of classical Morse-Bott functions  \citep{Austin:1995,Feehan:2020a}, 
which we extend in \cref{sec:param_morse-bott} to the case where there is a dependence on some additional parameter $x$ in $\X$.
\begin{defn}[\bf Morse-Bott function]\label{def:morse_bott}
	Let $h:\Y\rightarrow \R$ be a real-valued twice continuousely differentiable function. Define $\mathcal{C}_h$ to be the set of critical points of $h$ and consider $y_0\in C_h$. 
	We say that $h$ is Morse-Bott at $y_0$, if there exists a open neighbordhood  $\mathcal{V}$ of $y_0$ such that $ C_h\cap \mathcal{V}$ is a connected sub-manifold of $\Y$ of dimension $\dim\parens{\text{Ker}(\partial_{yy}^2h(y_0))}$. 
	We say that $h$ is a \emph{Morse-Bott function}  if for any $y_0\in \mathcal{C}_h$, $h$ is Morse-Bott at $y_0$.
\end{defn}
Morse-Bott functions were introduced in the context of differential topology to analyze the geometry of a manifold by studying the properties of differentiable functions defined on that manifold \citep{Austin:1995}. Their main property is that all their critical points that are connected have the same type (same number of positive and negative eigenvalues for the Hessian), a fact expressed by the Morse-Bott lemma \citep[Theorem 2.10]{Feehan:2020a} that we generalize to the parametric setting in \cref{thm:morse-bott_lemma}. 
Morse-Bott functions form a \emph{generic} class of functions \citep{Martinez-Alfaro:2016}, meaning that any smooth function can always be slightly perturbed to become a smooth Morse-Bott function. 
Hence, in principle, requiring that $y\mapsto g(x,y)$ is a Morse-Bott function for any parameter $x\in\X$ is essentially a mild assumption. 
The Morse-Bott property allows characterizing the geometry of critical points of $g(x,.)$ for any $x$ and ensures that the selection map $\phi$ is well-defined \citep[Chapter 15]{Cohen:1991}. 
However, this condition does not provide any information about how the set of critical points evolves as the parameter $x$ varies, which is crucial for the study of smoothness of the \selection $\phi$. This is precisely why we introduced \emph{parametric Morse-Bott functions} in \cref{sec:param_morse-bott}.
\subsection{Properties of Parameteric Morse-Bott Functions.}\label{sec:properties_parameteric_morse_bott}
In this section, we describe some elementary properties of parametric Morse-Bott functions. In particular, \cref{prop:exact_least_square_solutions} shows that $\partial_{xy}^2 g(x,y)$ belongs to the range of $\partial_{yy}^2g(x,y)$ whenever $(x,y)$ is an augmented critical point of $g$, i.e. $\partial_y g(x,y)=0$. \cref{prop:pointwise_morse_bott} shows that any parametric Morse-Bott function $g$ satisfies a pointwise Morse-Bott property in the sense of \cref{def:morse_bott}.  Finally, \cref{prop:morse-functions-with-parameters,lem:morse_bott_diffeo} provide examples of functions that satisfy the parametric Morse-Bott property. Recall $\mathcal{M}$ the set of augmented critical points of $g$:
\begin{align}\label{eq:augmented_critical_points}
	\mathcal{M} = \braces{(x,y)\in \X\times \Y \middle| \partial_y g(x,y)=0}.
\end{align}
\begin{prop}[\bf Exact least square solution]\label{prop:exact_least_square_solutions}
Let $g$ be a parametric Morse-Bott function. 
Let $(x_0,y_0)$ be  an element in $\mathcal{M}$ defined in \cref{eq:augmented_critical_points} and define the matrices  $A := \partial_{yy}^2 g(x_0,y_0)$ and $B:= \partial_{xy}^2 g(x_0,y_0)$. Then, $B$ is in the range of $A$, i.e. there exists a matrix $U$ such that $B {=} UA$. 
\end{prop}
\begin{proof}
Recall that $\mathcal{M}$ is the set of augmented critical points of $g$. Since $g$ is a parametric Morse-Bott function, there exists a neighborhood $\mathcal{U}$ of $(x_0,y_0)$ such that the augmented critical set $\mathcal{M}\cap \mathcal{U}$ is a $C^2$ manifold of dimension $d_{\mathcal{M}} {=} \dim(\X) + \dim(Ker(\partial_{yy}^2 g(x_0,y_0)))$. We know that $\mathcal{M}\cap \mathcal{U}$ is characterized locally by the equation $\partial_y g(x_0,y_0)=0$, hence the tangent space $T{\mathcal{M}}_{(x_0,y_0)}$ of $\mathcal{M}\cap \mathcal{U}$ at point $(x_0,y_0)$ consist of the set of directions $(u,v)\in \X\times \Y$ for which $\partial_y g(x_0+\epsilon u,y_0+\epsilon v)= O(\epsilon^2)$. In other words $T{\mathcal{M}}_{(x_0,y_0)}$ is the set of vectors $(u,v)\in \X\times \Y$ of $\mathcal{M}\cap \mathcal{U}$  satisfying the equation:
	\begin{align}
		u^{\top}\partial_{xy}^2 g(x_0,y_0) + v^{\top} \partial_{yy}^2 g(x_0,y_0) {=} u^{\top}B+ v^{\top} A{=}0.
	\end{align}
	 Since $\mathcal{M}\cap \mathcal{U}$ is of dimension $d_{\mathcal{M}}$, the tangent space $T\mathcal{M}_{(x_0,y_0)}$ must also have dimension $d_{\mathcal{M}}$.  Therefore, by the rank theorem, it must hold that the matrix $D = (B,A)$ has a rank equal to $\dim(\mathcal{X}) + \dim(\mathcal{Y}) - d_{\mathcal{M}} = \text{rank}(A)$. On the other hand, we know that $0^{\top}B + v^{\top}A = v^{\top}A\in Range(A)$ for any $v\in \mathcal{Y}$, so that $Range(A)\subset Range(D)$. The two subspaces having the same dimension, the inclusion implies equality ($Range(A)=Range(D)$). Henceforth, there must exist a matrix $U$ such that $B$ can be written as $B{=} UA$.
\end{proof}

\begin{prop}[\bf Pointwise Morse-Bott property]\label{prop:pointwise_morse_bott}
	Let $g$ be a parametric Morse-Bott function. Then for any $x\in \X$, the function $y\mapsto g(x,y)$ is a Morse-Bott function in the following sense:
	For any $x_0$ and any critical point $y_0$ of $g(x_0,.)$,  there exists an open neighborhood $\mathcal{V}$ of $y_0$ so that $C_{x,y_0}:=\braces{y\in \Y\middle | \partial_y g(x_0,y)=0}\cap\mathcal{V}$ is a connected sub-manifold of dimension equal to the dimension of the null space of the Hessian $\partial_{yy}^2 g(x,y_0)$.
\end{prop}
\begin{proof}
	Let $(x_0,y_0)$ be in $\X\times \Y$ such that $\partial_yg(x_0,y_0)=0$. Then, since $g$ is a parameteric Morse-Bott function, there exists a neighborhood $\mathcal{U}$ of $(x_0,y_0)$ such that the augmented critical set $\mathcal{M}\cap \mathcal{U}$ is a $C^2$ manifold of dimension $d_{\mathcal{M}} {=} \dim(\X) + \dim(Ker(\partial_{yy}^2 g(x_0,y_0)))$.   On the other hand, we know that $\mathcal{M}\cap \mathcal{U}$ is characterized locally by the equation $\partial_y g(x,y)=0$, hence the tangent vectors $(u,v)\in \X\times \Y$ of $\mathcal{M}\cap \mathcal{U}$ at $(x_0,y_0)$ must satisfy the equation:
	\begin{align}
		u^{\top}\partial_{xy}^2 g(x_0,y_0) + v^{\top} \partial_{yy}^2 g(x_0,y_0) {=} 0.
	\end{align}
	For simplicity, we denote by $B = \partial_{xy}^2 g(x_0,y_0)$ and $A=\partial_{yy}^2 g(x_0,y_0)$. By \cref{prop:exact_least_square_solutions}, we know that $B$ can be written in the form $B=UA$ for some matrix. Hence, the tangent space of $\mathcal{M}$ at $(x_0,y_0)$ consists in vectors $(u,v)\in \X\times \Y$ satisfying 
 	\begin{align}
		\parens{u^{\top}U+v}A=0.
	\end{align}
	In particular, for any $u\in \X$, we can set $v= -u^{\top}U$ which ensures that $(u,v)$ is in the tangent space of $\mathcal{M}$ at $(x_0,y_0)$. Now consider the sub-manifold $\{x_0\}\times \Y$, its tangent space at $(x_0,y_0)$ is $\{0\}\times \Y$. 
	For any element $(x,y)\in \X\times \Y$, we have the decomposition $(x,y){=} (x,-x^{\top}U) + (0,y+x^{\top}U)$ where the first tuple belongs to the tangent space of $\mathcal{M}$ and the second one belongs to the tangent space of $\{x_0\}\times \Y$ at $(x_0,y_0)$. Hence, the tangent space of $\X\times \Y$ is generated by the both separate tangent spaces which means that both manifolds intersect transversally and that $\{x_0\}\times\mathcal{C}:=(\mathcal{M}\cap \mathcal{U})\cap (\{x_0\}\times \Y)$ is a sub-manifold of dimension $\dim\parens{Ker\parens{ \partial_{yy}^2 g(x_0,y_0)}}$ \citep[Theorem 6.30]{Lee:2003}. For a small enough open connected neighborhood $\mathcal{V}$ of $y_0$, we can ensure that $\mathcal{C}\cap \mathcal{V}$ is a connected sub-manifold of $\mathcal{Y}$. This precisely means that $y\mapsto g(x_0,y)$ is Morse-Bott at the point $y_0$ which concludes the proof.
\end{proof}

\begin{prop}[\bf Morse functions with parameters]\label{prop:morse-functions-with-parameters}
	Let $g:\X\times\Y$ be a three-times continuously differentiable function such that for any $(x,y)\in \X\times \Y$ for which $\partial_y g(x,y)=0$, the Hessian matrix $\partial_{yy}^2g(x,y)$ is invertible. Then $g$ is a parametric Morse-Bott function.
\end{prop}
\begin{proof}
	Let $(x_0,y_0)\in \X\times\Y$ be such that $y_0$ is a critical point of $g(x_0,.)$ ( i.e. $\partial_y g(x_0,y_0){=}0$). Since, by assumption, the Hessian is invertible, we can apply the implicit function theorem which guarantees the existence of a function $x\mapsto y(x)$ defined in a neighborhood $\mathcal{U}$ of $x_0$ and taking values in a neighborhood $\mathcal{V}$ of $y_0$, such that $y(x_0)=y_0$ and $y(x)$ is the unique critical point of $g(x,.)$ on $\mathcal{V}$, i.e.:
	\begin{align}
		\partial_y g(x,y(x))=0,\qquad \forall x\in \mathcal{U}.
	\end{align}
	 Moreover, $x\mapsto y(x)$ is twice continuously differentiable.  This ensures that $\mathcal{M}$ the set of augmented critical points of $g$ satisfies:
	\begin{align}
		\mathcal{M}\cap\parens{\mathcal{U}\times \mathcal{V}} = \braces{(x,y(x))\in \X\times \Y \middle| x\in \mathcal{U} } := \mathcal{S}.
	\end{align}
	We only need to show that $\mathcal{M}\cap\parens{\mathcal{U}\times \mathcal{V}}$ is a manifold of dimension $\dim(\X)$. For this, we will apply the regular level set theorem ~\citep[Corollary 5.14]{Lee:2003} to the function $G: (x,y)\mapsto \partial_y g(x,y)$ defined on $\mathcal{U}\times \mathcal{V}$. The pre-image of $0$ by $G$ is exactly equal to $\mathcal{M}\cap\parens{\mathcal{U}\times \mathcal{V}}$. Moreover, for any $(x,y)\in \mathcal{M}\cap\parens{\mathcal{U}\times \mathcal{V}}$, we have that $dG(x,y)$ is of maximal rank since $\partial_{yy}^2 g(x,y)$ is invertible. Hence, by application of the regular level set theorem theorem to the twice continuously differentiable ($C^2$) function $G$, it follows that $\mathcal{M}\cap\parens{\mathcal{U}\times \mathcal{V}}= G^{-1}(\braces{0})$ is a $C^2$ sub-manifold of $\X\times \Y$ of dimension $\dim(ker(dG(x,y)))=\dim(\X)$. We have shown that $\mathcal{M}\cap\parens{\mathcal{U}\times \mathcal{V}}$ is sub-manifold of dimension $\dim(\X)$, which proves the result.
\end{proof}

\begin{lem}\label{lem:morse_bott_diffeo}
	Let $h$ be a smooth Morse-Bott function defined  on $\Y$. Let $\mathcal{T}:\X\times \Y\rightarrow \Y$ be a smooth function, such that $y\mapsto \mathcal{T}(x,y) = \tau_x(y)$  is a diffeomorphism on $\Y$ for any $x\in \X$. Then the function $g(x,y) = h(\tau_x(y))$ is a parametric Morse-Bott function. 
\end{lem}
\begin{proof}
	Consider the function $G: (x,y)\mapsto \partial_y g(x,y)$. We have the following equivalence
	\begin{align}
		(x,y)\in G^{-1}(\braces{0}) \iff \partial_y h(\tau_x(h))\partial_y\tau_x(y)=0\iff  \tau_x(y)\in \partial_y h^{-1}(\braces{0}).
	\end{align}
Consider the map $\mathcal{T}:(x,y)\mapsto \tau_x(y)$, then we have shown that $G^{-1}(\braces{0}) = \mathcal{T}^{-1}\parens{\partial_y h^{-1}(\braces{0})}$.	
Let $(x_0,y_0)\in \X\times \Y$ be an augmented critical point of $g$. Set $\tilde{y} = \mathcal{T}(x_0,y_0)$ which is a critical point of $h$. Since $h$ is, by assumption, Morse-Bott at $\tilde{y}$, then there exists an open neighborhood $\tilde{\mathcal{V}}$ of $\tilde{y}$ such that $ \partial_y h^{-1}(\braces{0})\cap \tilde{\mathcal{V}}$ is a sub-manifold of dimension $\dim(Ker(\partial_{yy}^2h(\tilde{y})))$. By continuity of $\mathcal{T}$, we can always find open connected neighborhoods $\mathcal{U}$ and $\mathcal{V}$ of $(x_0,y_0)$ so that $\mathcal{V}' {:=} \mathcal{T}(\mathcal{U}\times \mathcal{V})\subset \tilde{\mathcal{V}}$.  Moreover, since for any $x$ $\mathcal{T}(x,.)$ is a diffeomorphism, it must be that $\mathcal{V}'$ is an open set. Therefore, $\mathcal{S}:= \partial_y h^{-1}(\braces{0})\cap \mathcal{V}'$ must be a sub-manifold as of dimension $\dim(Ker(\partial_{yy}^2h(\tilde{y})))$. It remains to show that $\mathcal{T}^{-1}(\mathcal{S})$ is a sub-manifold.  To see this, it suffice to note that the differential of $\mathcal{T}$ is surjective which ensures that $\mathcal{T}$ is transverse to $\mathcal{S}$ and that $\mathcal{T}^{-1}(\mathcal{S})$  is a sub-manifold \citep[Theorem 6.30]{Lee:2003}. Moreover, the dimension of such manifold is  equal to $\dim(\X)+ \dim(ker(\partial_{yy}^2h(\tilde{y}))) = \dim(\X)+ \dim(ker(\partial_{yy}^2g(x_0,y_0)))$. 

\end{proof}

\subsection{Proof of the Morse-Bott Lemma with Parameters}\label{sec:proof_morse-bott_lemma}
In this section, we provide a proof of the Morse-Bott lemma with parameters introduced in \cref{thm:morse-bott_lemma}. We then introduces two results in \cref{corr:similar_hessian,cor:existence_diffeomorphism} which are consequences of \cref{thm:morse-bott_lemma}. \cref{corr:similar_hessian} shows that near an augmented critical point $(x_0,y_0)$, the Hessian matrices of nearby augmented critical points are all similar. This result illustrates that the geometry near a critical point is preserved when the parameter $x$ is perturbed.  \cref{corr:similar_hessian} will be used later in  \cref{prop:boundedness_U} of \cref{sec:differentiability-selection} to show that the pseudo-inverse of the Hessian matrices of critical points near a local minimum are uniformly bounded.  Finally, \cref{cor:existence_diffeomorphism} shows that near any augmented critical point $(x_0,y_0)$ the function $y\mapsto g(x,y)$ can be expressed as a slight deformation of $y\mapsto g(x_0,y)$. This result, along with the stability result in \cref{sec:stability_gradient_flow} of the gradient flow to deformations will be key to prove the continuity of the selection map $x\mapsto \phi(x,y)$ near local minima. 

\begin{proof}[Proof of \cref{thm:morse-bott_lemma}]
Let $x_0\in \X$ and $y_0$ be a critical point of $g(x_0,.)$. Denote by $\mathcal{K}$ the null space of the Hessian $A_0 = \partial_{yy}^2 g(x_0,y_0)$ and by $\mathcal{K}^{\perp}$ its orthogonal complement in $\Y$. The function $g(x_0,.)$ is a Morse-Bott function by \cref{prop:pointwise_morse_bott}, therefore by the Morse-Bott lemma \cite[Theorem 2.10]{Feehan:2020a}, there exists three open neighborhoods $\mathcal{O}$, $\mathcal{O}^{\perp}$ and $\mathcal{V}$ of $0\in \mathcal{K}$, $0\in \mathcal{K}^{\perp}$ and $y_0\in \Y$ and a diffeomorphism $s: \mathcal{O}\times \mathcal{O}^{\perp}\rightarrow \mathcal{V}$ s.t. $s(0,0)=y_0$ and for any $r,w\in \mathcal{O}\times \mathcal{O}^{\perp}$ it holds that:
\begin{align}
	g(x_0,s(r,w)) = g(x_0,y_0) + \frac{1}{2}w^{\top}J_0w, \forall r,w\in \mathcal{O}\times \mathcal{O}^{\perp}.
\end{align} 
where $J_0$ is an invertible diagonal matrix whose diagonal elements are equal to the sign of the non-zero eigenvalues of the Hessian $\partial_{yy}^2 g(x_0,y_0)$. By convention $J_0 {=} 0$ in case the Hessian  $\partial_{yy}^2 g(x_0,y_0)=0$. 
Since, the function $h(x,r,w) := g(x,s(r,w))$ is such that $\partial_{w} h(x_0,0,0) = 0$ and the partial Hessian $\partial_{ww}^2 h(x_0,0,0) = J_0$ is invertible, we are in position to apply the Morse lemma with parameters \cite[Theorem 4]{Feehan:2020a}. 
The lemma ensures that  $\mathcal{O}$ and $\mathcal{O}^{\perp}$ can be chosen small enough so that there exits open neighborhoods $\mathcal{B}$ and $ \mathcal{O}^{\perp}_1$ of $x_0\in \X$ and $0\in \mathcal{K}^{\perp}$ and a diffeomorphism $\tau$ from $\mathcal{B}\times \mathcal{O}\times \mathcal{O}_1^{\perp}$ to $\mathcal{B}\times \mathcal{O}\times \mathcal{O}^{\perp}$ such that $\tau(x_0,0,0) {=} (x_0,0,0)$ and decomposing $h$ locally into a quadratic component and a singular one. More precisely, for any $(x,r,w)\in \mathcal{B}\times \mathcal{O}\times \mathcal{O}_1^{\perp}$, the map $r$ satisfies $\tau(x,r,w) {=} (x,r,w')$ for some $w'\in \mathcal{O}^{\perp}$ and the following equation holds:
	\begin{align}\label{eq:local_parameterization}
	h(\tau(x,r,w)) = h(\tau(x,r,0)) + \frac{1}{2} w^{\top}J_0w.
		\end{align}
It remains to show that $\xi\mapsto h(\tau(x,r,0))$ is in fact constant for $(x,r)$ in an open neighborhood of $(x_0,0)\in \X\times \mathcal{O}$. To this end, define the sets $A$, $B$  and $C$ as follows:
\begin{align}
	A &:= \braces{(x,y) \in \mathcal{B}\times \mathcal{V}\quad \middle| \quad \partial_y g(x,y)=0},\\
	B &:= \braces{(x,r,w)\quad \middle| \quad (x,r) \in \mathcal{B}\times \mathcal{O}, \quad \partial_{r} h(\tau(x,r,w))=0},\\
	C &:= \braces{(x,r,0)\quad \middle| \quad (x,r) \in \mathcal{B}\times \mathcal{O}, \quad \partial_{r} h(\tau(x,r,0))=0}.
\end{align}
Then by \cref{eq:local_parameterization}, it holds that $B=C$. Moreover, $A$ and $B$ are homeomorphic. Indeed to see this, we introduce the notation $\tilde{s}(x,r,w) := (x,s(r,w))$ which defines a diffeomorphism from $\mathcal{B}\times \mathcal{O}\times \mathcal{O}^{\perp}$ to $\mathcal{B}\times \mathcal{V}$. Hence, $g\circ \tilde{s}\circ \tau = h\circ \tau$. This ensures $\tilde{s}\circ \tau(B) {=} A $,  which means precisely that $A$ and $B$ are homeomorphic since $\tilde{s}\circ \tau$ is a homeomorphism. Moreover, by definition of $g$ as a parametric Morse-Bott function, we also know that $A$ is a sub-manifold of $\X\times \Y$ of dimension $\dim(\X) + \dim\parens{Ker \parens{\partial_{yy}^2 g(x_0,y_0)}}$ provided the neighborhoods $\mathcal{B}$ and $\mathcal{V}$ are small enough. Hence, we can deduce that $B$ and $C$ must also be sub-manifolds of the same dimension. In particular, $C$ is a sub-manifold of $\mathcal{B}\times \mathcal{O}\times \{0\}$ which is of dimension $\dim(\X) + \dim\parens{Ker \parens{\partial_{yy}^2 g(x_0,y_0)}}$. Therefore, $C$ is an open sub-manifold of $\mathcal{B}\times \mathcal{O}\times \{0\}$. Hence, since $(x_0,0,0)\in C$, there must exists an open connected neighborhood $\mathcal{B}_1\times \mathcal{O}_1\times \{0\}$ of $(x_0,0,0)$ in $\mathcal{B}\times \mathcal{O}\times \{0\}$ that is contained in $C$. Hence, we  deduce that for any $(x,r)\in \mathcal{B}_1\times \mathcal{O}_1 $, the function $h$ satisfies $\partial_{r}h(\tau(x,r,0))=0$ so that $h(\tau(x,r,0)) = h(\tau(x,0,0))$ on such neighborhood. Finally, we have shown that there exits 
$$
g \circ \tilde{s}\circ \tau(x,r,w) = g \circ \tilde{s}\circ \tau(x,0,0) + \frac{1}{2}w^{\top}J_0w.
$$
We conclude the proof by setting $\psi(x,r,w) = \tilde{s}\circ \tau(x,r,w)$ which is the desired diffeomorphism.
\end{proof}

\begin{prop}\label{corr:similar_hessian}
Let $g$ be a real-valued function such that \cref{assumpt:morse-bott} holds. Consider an augmented critical point $(x_0,y_0){\in} \mathcal{M}$,  with $\mathcal{M}$ defined in \cref{eq:augmented_critical_points}. 
Then there exists a neighborhood $\mathcal{V}$ of $(x_0,y_0)$ and a continuous map $(x,y)\mapsto P(x,y)$ defined on $\mathcal{V}$ with values in $\R^{d\times d}$ such that:
\begin{itemize}
	\item $P(x,y)$ is invertible for any $(x,y)\in \mathcal{V}$ with singular values contained in an interval $[\sigma_{\min},\sigma_{\max}]$ for some positive constants $\sigma_{\min}$ and $\sigma_{\max}$.
	\item For any augmented critical point $(x,y)\in \mathcal{V}$, the Hessian of $g$ is given by:
	$$
	\partial_{yy}^2 g(x,y) = P(x,y)^{\top}\partial_{yy}^2 g(x_0,y_0)P(x,y).
	$$ 
\end{itemize}

\end{prop}
\begin{proof}
Denote by $\mathcal{K}$ the null space of the Hessian $A_0 = \partial_{yy}^2 g(x_0,y_0)$ and by $\mathcal{K}^{\perp}$ its orthogonal complement in $\Y$. Let $J_0$ be a diagonal matrix with diagonal elements given by the sign of the non-zero eigenvalues of $A_0$. Since $g$ satisfies \cref{assumpt:morse-bott}, we apply \cref{thm:morse-bott_lemma} which ensures the existence of a diffeomorphism $\psi$ defined on an open neighborhood $\mathcal{U}$ of $(x_0,0,0)\in \X\times \mathcal{K}\times \mathcal{K}^{\perp} $ with values in an open neighborhood $\mathcal{V}$ of $(x_0,y_0)$ in $\X\times \Y$, s.t. $\psi(x_0,0,0)=(x_0,y_0)$ and for all $(x,r,w)\in \mathcal{U}$, $\psi$ satisfies $\psi(x,r,w){=} (x,y)$ and
\begin{align}\label{eq:parm_morse_bott}
	g(\psi(x,r,w)) &= g(\psi(x,0,0))+ \frac{1}{2}w^{\top}J_0w, \\
	&= g(\psi(x,0,0)) + \frac{1}{2}(r^{\top} w^{\top})\tilde{J}_0\begin{pmatrix}
		r\\
		w
	\end{pmatrix},
\end{align}
where we defined $\tilde{J}_0$ to be the matrix of dimension $d\times d$ given by:
\begin{align}
	\tilde{J}_0 =\begin{pmatrix}
		0 & 0\\
		0 & J_0 
	\end{pmatrix}.
\end{align}
Since $\psi$ is a diffeomorphism satisfying $\psi(x,r,w) = (x,y)$,  we can equivalently write \cref{eq:parm_morse_bott} as:
\begin{align}\label{eq:parm_morse_bott_1}
	g(x,y) = g(\psi(x,0,0)) + \frac{1}{2}\psi^{-1}_{2,3}(x,y)^{\top}\tilde{J}_0\psi^{-1}_{2,3}(x,y),\qquad \forall  (x,y)\in \mathcal{V}, 
\end{align}
where $\psi_{2,3}^{-1}(x,y)$ are last two components of $\psi^{-1}(x,y)$ (i.e. $\psi^{-1}(x,y) = (x,\psi_{2,3}^{1}(x,y))$). 
By differentiating \cref{eq:parm_morse_bott_1} w.r.t. $y$ we obtain:
\begin{align}\label{eq:grad_morse_bott}
	\partial_y g(x,y) = \partial_y \psi^{-1}_{2,3}(x,y)\tilde{J}_0 \psi^{-1}_{2,3}(x,y). 
\end{align}
$\partial_y \psi^{-1}_{2,3}(x,y)$ must be  invertible since $(\partial_x \psi^{-1}, \partial_y \psi^{-1})$ is invertible and of the form:
\begin{align}
	\begin{pmatrix}
		\partial_x \psi^{-1}\\ \partial_y \psi^{-1}
	\end{pmatrix}
 = \begin{pmatrix}
		I & \partial_x \psi_{2,3}^{-1}\\
		0 & \partial_y \psi_{2,3}^{-1}
	\end{pmatrix}.
\end{align}
Therefore, if $y$ is a critical point of $g(x,.)$, then  \cref{eq:grad_morse_bott} implies that $\tilde{J}_0 \psi^{-1}_{2,3}(x,y) = 0$. Let $(x,y)$ be an augmented critical point of $g$,  $\epsilon>0$ and $u$ be vector in $\Y$, then the following holds:
\begin{align}
	\frac{1}{\epsilon}\partial_y g(x,y+\epsilon u) = \parens{\partial_y \psi_{2,3}^{-1}(x,y+\epsilon u)}^{\top}\tilde{J}_0\parens{\frac{1}{\epsilon} \psi_3^{-1}(x,y+\epsilon u)-\psi_{2,3}^{-1}(x,y)}.
\end{align}
Hence, by taking the limit when $\epsilon$ approaches $0$, it follows that:
\begin{align}
	\partial_{yy}^2 g(x,y) = \parens{\partial_y \psi_{2,3}^{-1}(x,y)}^{\top}\tilde{J}_0\partial_y \psi_{2,3}^{-1}(x,y).
\end{align}
Define $P_0 := \partial_y \psi_{2,3}^{-1}(x_0,y_0)\partial_y \psi_{2,3}^{-1}(x_0,y_0)^{\top}$ which is invertible. Then, we can write:
\begin{align}
	\partial_{yy}^2 g(x,y) &= \partial_y \psi_{2,3}^{-1}(x,y)^{\top}P_0^{-1}P_0 \tilde{J}_0 P_0 P_0^{-1}\partial_y \psi_{2,3}^{-1}(x,y)\\
	   &= \partial_y \psi_{2,3}^{-1}(x,y)^{\top}P_0^{-1}\partial_y \psi_{2,3}^{-1}(x_0,y_0) A_0 \partial_y \psi_{2,3}^{-1}(x_0,y_0)^{\top}P_0^{-1}\partial_y \psi_{2,3}^{-1}(x,y)\\
	   &= P(x,y)^{\top}\partial_{yy}^{2}g(x_0,y_0)P(x,y),
\end{align}
where we defined $P(x,y){=} \partial_y \psi_{2,3}^{-1}(x_0,y_0)^{\top}P_0^{-1}\partial_y \psi_{2,3}^{-1}(x,y)$. The matrix $P(x,y)$ is invertible for any $(x,y)\in \mathcal{V}$ and the map $(x,y)\mapsto P(x,y)$ is continuous. Hence, by considering compact neighborhood of $(x_0,y_0)$ contained in $\mathcal{V}$, we can ensure that the singular values of $P(x,y)$ are contained in an interval $[\sigma_{\min},\sigma_{\max}]$ where $\sigma_{\min}$ and $\sigma_{\max}$ are positive numbers. Further considering the restriction of such map on an open neighborhood $\mathcal{V}'\subset K$ of $(x_0,y_0)$ yields the desired result. 
\end{proof}

\begin{corr}\label{cor:existence_diffeomorphism}
Let $g$ be a real-valued function such that \cref{assumpt:morse-bott} holds. Consider an augmented critical point $(x_0,y_0){\in} \mathcal{M}$,  with $\mathcal{M}$ defined in \cref{eq:augmented_critical_points}.   
Then, there exists a open neighborhoods $\mathcal{B}$ and $\mathcal{V}$ of $x_0$ and $y_0$ in $\X$ and $\Y$ and a continuously differentiable map $\tau$ from $\mathcal{B}\times \mathcal{V}$ to $\mathcal{V}$  such that: 
\begin{itemize}
	\item For any $x\in \mathcal{B}$, the map $\tau_x: y\mapsto \tau(x,y)$ is a diffeomorphism from $\mathcal{V}$ to itself satisfying $\tau_{x_0}(y)=y$ for any $y\in \mathcal{V}$. Moreover, $(x,y)\mapsto \tau_x^{-1}(y)$ is continuous.
		\item For any $(x,y)\in \mathcal{B}\times\mathcal{V}$, the function $g$ satisfies $g(x,y) {=}  g(x_0,\tau(x,y))+C(x)
$,  where $x\mapsto C(x)$ is a function independent of $y$.
	\item There exists positive numbers $\ell$ and $L$ s.t for any $(x,y)\in \mathcal{B}\times\mathcal{V}$:
		\begin{align}\label{eq:non_degeneracy_diffeo_1}
			\ell^2 I \leq  \partial_y\tau(x,y)^{\top}\partial_y\tau(x,y)\leq (L')^2 I.
		\end{align}

\end{itemize}
\end{corr}
\begin{proof}
	We use the notations of \cref{thm:morse-bott_lemma} where $\mathcal{K}$ is the null subspace of the Hessian $\partial_{yy}^2 g(x_0,y_0)$ and $\mathcal{K}^{\perp}$ its orthogonal complement in $\Y$.  By \cref{thm:morse-bott_lemma} $g$ satisfies: 
	\begin{align}
		g(\psi(x,r,w)) = g(\psi(x,0,0)) + \frac{1}{2}\bar{y}^{\top}J_0\bar{y},
	\end{align}
	with $\psi$ and $J_0$ being the diffeomorphism and matrix defined in \cref{thm:morse-bott_lemma}. Recall that $\psi$ is defined on an open neighborhood $\mathcal{B}\times \mathcal{O}\times \mathcal{O}^{\perp}$ of $(x_0,0,0)\in \mathcal{X}\times \mathcal{K}\times\mathcal{K}^{\perp}$ and whose image by $\psi$ is an open neighborhood $\mathcal{B}\times\mathcal{V}$ of $(x_0,y_0)$.  Hence, we can write:
	$$
	g(\psi(x,r,w)) = C(x) + g(\psi(x_0,r,w)),
	$$
	with $C(x):= g(\psi(x,0,0))-g(\psi(x_0,0,0))$. 	
	We also know that $\psi$ preserves $x$, meaning that $\psi(x,r,w) {=} (x,y)$. Hence, we can define  $(x,r,w)\mapsto \tilde{\tau}_x(r,w)\in \Y$, s.t.  $ \psi(x,r,w) {=} (x,\tilde{\tau}_x(r,w))$. For any $x\in \mathcal{B}$, $(r,w)\mapsto \tilde{\tau}_x(r,w)$ defines a diffeomorphism from $\mathcal{O}\times \mathcal{O}^{\perp}$ onto its image. Moreover, its image must be equal to $\mathcal{V}$. Indeed, since $\psi(\mathcal{B}\times\mathcal{O}\times\mathcal{O}^{\perp})=\mathcal{B}\times \mathcal{V}$, it follows that for any $(x,y)\in \mathcal{B}\times \mathcal{V}$, there exists $(r,w)\in \mathcal{O}\times\mathcal{O}^{\perp}$ such that $\psi(x,r,w) = (x,\tilde{\tau}_x(r,w))= (x,y)$. 
	In particular, if $(x,y)\in \mathcal{B}\times \mathcal{V}$ and $(r,w)= \tilde{\tau}_x^{-1}(y)$,  we can write $\psi(x_0,r,w) {=} (x_0,\tilde{\tau}_{x_0}(r,w)) {=} (x_0,\tilde{\tau}_{x_0}\tilde{\tau}_x^{-1}(y))$. Therefore, the following expression holds for any $(x,y)\in\mathcal{B}\times \mathcal{V}$:
$$
g(x,y)= C(x) + g(x_0,\tau(x,y)),
$$
where we defined $\tau(x,y) {=} \tilde{\tau}_{x_0}\circ \tilde{\tau}_x^{-1}(y)$. For any $x\in \mathcal{B}$, the map $\tau_x: y\mapsto \tau(x,y)$ is a diffeomorphism satisfying $\tau(x_0,y){=}y$.
Moreover,  $(x,r,w)\mapsto \tilde{\tau}_x(r,w)$ and $(x,y)\mapsto \tilde{\tau}_x^{-1}(y)$ are continuously differentiable since $\psi$ is a diffeomorphism. As a result, $\tau$ is continuously differentiable as well and $(x,y)\mapsto\tau_x^{-1}(y)$ is continuously differentiable.  
Finally, since $\partial_y\tau(x,y)$ is jointly continuous in $x$ and $y$ and  $\partial_y\psi_{x}(y)$ is invertible, then, provided that $\mathcal{B}$ and $\mathcal{V}$ are small enough, there must exist two positive numbers $\ell$ and $L'$ such that for any $(x,y)\in \mathcal{B}\times \mathcal{V}$:
		\begin{align}
			\ell^2 I \leq  \partial_y\psi(y)^{\top}\partial_y\psi(y)\leq (L')^2 I.
		\end{align}
\end{proof}

\begin{proof}[Proof of \cref{prop:uniform_KL} ]
	Recall $\mathcal{M} = \braces{(x,y)\in \X\times \Y \middle| \partial_y g(x,y)=0}$ the set of augmented critical points of $g$ and let $(x_0,y_0)$ be in $\mathcal{M}$. 
	First, since \cref{assumpt:morse-bott} holds, we know by \cref{prop:pointwise_morse_bott} that $g(x_0,.)$ is a Morse-Bott function. Hence, by \citep[Theorem 1]{Feehan:2020a}, it follows that $g(x_0,.)$ satisfies a {\L}ojasiewicz inequality near $y_0$. In other words, there exists a neighborhood $\mathcal{V}$ of $y_0$ and a positive constant $\mu'>0$ such that:
	\begin{align}
		\mu'\verts{g(x_0,y)-g(x_0,y_0)}\leq \frac{1}{2}\Verts{\partial_y g(x_0,y)}^2,\qquad \forall y\in \mathcal{V}. 
	\end{align}	
	By \cref{cor:existence_diffeomorphism}, there exists a continuous function $\tau$ defined on an open neighborhood $\mathcal{B}\times\mathcal{V}$ of $(x_0,y_0)$ whose image is $\mathcal{V}$ and for which $g(x,y){=} g(x_0,\tau(x,y)) {+} C(x)$ for any $(x,y)\in \mathcal{B}\times\mathcal{V}$, where $C(x)$ is a  function of $x$ independent of $y$. Moreover, for any $x\in \mathcal{B}$, $y\mapsto \tau(x,y)$ is a diffeomorphism from $\mathcal{V}$ to itself whose inverse is written as $\tau^{-1}(x,y)$ by an abuse of notion. 
	In particular, for $y {=} \tau^{-1}(x,y_0)$ we set $G(x) {:=} g(x,\tau^{-1}(x,y_0)) = g(x_0,y_0) + C(x)$. Note that $\tau^{-1}(x,y_0)$ is critical point of $g(x,.)$ since $\partial_y g(x,\tau^{-1}(x,y_0)) \partial_y \tau^{-1}(x,y_0) {=} \partial_y g(x_0,y_0) = 0 $ and  $\partial_y \tau^{-1}(x,y_0)$ is invertible. Hence, the following holds for any $(x,y)\in \mathcal{B}\times \mathcal{V}$.
	\begin{align}\label{eq:KL_unif}
		\mu'\verts{g(x,y)-G(x)} = \mu'\verts{g(x_0,\tau(x,y)) - g(x_0,y_0)}\leq \frac{1}{2}\Verts{\partial_y g(x_0,\tau(x,y)}^2.
	\end{align}
Moreover, by construction of $\tau$, we know that $\partial_y \tau(x,y)$ satisfies \cref{eq:non_degeneracy_diffeo_1} for any $(x,y)\in \mathcal{B}\times \mathcal{V}$.  Therefore, we deduce that: 
\begin{align}
	\Verts{\partial_y g(x,y)}^2 =\Verts{\partial_y g(x_0,\tau(x,y))\partial_y \tau(x,y)}^2\geq \ell^2 \Verts{\partial_y g(x_0,\tau(x,y))}^2,
\end{align}
Finally, combining the above inequality with \cref{eq:KL_unif},  we get that, for any $(x,y)\in \mathcal{B}\times \mathcal{V}$:
\begin{align}
	\ell^2\mu'\verts{g(x,y)-G(x)}\leq \frac{1}{2}\Verts{\partial_y g(x,y)}^2,\qquad \forall (x,y)\in \mathcal{U}.
\end{align}
The result follows by setting $\mu=\ell^2\mu'>0$ and $\mathcal{U}= \mathcal{B}\times \mathcal{V}$. 
\end{proof}

\section{Asymptotic Properties of Gradient Flows}\label{sec:asymptotic_properties_flow}
\subsection{Convergence of the gradient flow.}
Recall that the gradient flow $\phi_t(x,y)$ satisfies the differential equation
$$
\frac{d\phi_t(x,y)}{dt} = -\partial_y g(x,\phi_t(x,y)), \quad \phi_0(x,y)=y.
$$
The next proposition shows that the gradient flow $\phi_t(x,y)$ converges towards a well-defined selection map $\phi(x,y)$. 
\begin{prop}[\bf Convergence of $\phi_t$.]\label{prop:convergence_flow}
Let $x,y$ be in $\X\times \Y$. Under \cref{assumpt:morse-bott,assumpt:smootness,assumpt:Coercivity}, $ (t,x,y)\mapsto\phi_t(x,y)$ is continuous and for any $(x,y)\in \X\times \Y$, $\phi_t(x,y)$ converges towards a unique critical point $\phi(x,y)$ of $y\mapsto g(x,y)$ as $t$ goes to $+\infty$.
\end{prop}
\begin{proof}
First, \cref{assumpt:smootness} ensures that the gradient flow $\phi_t(x,y)$ is uniquely defined at all times $t$ \citep{Daneri:2010}. $\phi_t(x,y)$ is jointly continuous in $(t,x,y)$ by Cauchy-Lipschitz theorem. 
Moreover,  $t\mapsto \phi_t(x,y)$ remains bounded thanks to \cref{assumpt:Coercivity}. Otherwise, there exists a subsequence $\phi_{t_n}(x,y)$  such that $g(x,\phi_{t_n}(x,y))$ diverges to $+\infty$. This contradicts the fact that $g(x,\phi_{t_n}(x,y))$ is decreasing since $\phi_t(x,y)$ is a gradient flow of $g$. Hence, we deduce that $\phi_t(x,y)$ must have at least one accumulation point $y^{\star}$.   Moreover, $y^{\star}$ must be a critical point of $g(x,.)$. To see this, note that $g(x,\phi_t(x,y))$ is a decreasing function in time and is lower-bounded. Hence, it admits a finite limit $l$. Moreover, by differentiating $g(x,\phi_t(x,y))$ is time, it follows that:
\begin{align}
	\frac{\diff}{\diff t}g(x,\phi_t(x,y)) = -\Verts{\partial_y g(x,\phi_t(x,y))}^2
\end{align}
This implies that $ \int_0^{+\infty} \Verts{\partial_y g(x,\phi_s(x,y))}^2  \diff s = g(x,y)- l$ is finite. Since, $g$ is $L$-smooth by \cref{assumpt:smootness}, this is only possible if $\partial_y g(x,\phi_s(x,y))$ converges to $0$. In particular, by continuity of  $\partial_y g(x,y)$, it follows that $\partial_y g(x,y^{\star})=0$. We only need to show that $y^{\star}$ is the unique accumulation point of $\phi_t(x,y)$. To show this, we apply  \cref{prop:uniform_KL}, which implies, in particular, that $g$ satisfies a {\L}ojasiewicz inequality in a neighborhood $\mathcal{V}$ of $y^{\star}$:
\begin{align}
	\mu\verts{g(x,y)-G(x)}\leq \Verts{\partial_y g(x,y)}^2, \forall y\in \mathcal{V}.
\end{align}
We can therefore apply \cite[Theorem 2.7]{Merlet:2013} which ensure that $y^{\star}$ is the unique accumulation point of $\phi_t(x,y)$ and that $\phi_t(x,y)$ converges towards $y^{\star}$. We can therefore defined the map $\phi(x,y) = \lim_{t\rightarrow \infty} \phi_t(x,y)$ which constitues a selection.

\end{proof}

\subsection{Stability of the Gradient Flow Near Local Minima}\label{sec:stability_gradient_flow}

In this section, we provide a general result establishing the stability of gradient flows to perturbations. This result shows that deforming a gradient flow by a family of diffeomorphisms yields trajectories that are not too far from the unperturbed flow. We will use this result later in \cref{sec:continuity_flow_selection} in conjunction with 
the formulation of $y\mapsto g(x,y)$ as a perturbation of $y\mapsto g(x_0,y)$ provided in \cref{cor:existence_diffeomorphism} to prove that the gradient flow $\phi_t(x,y)$ remain stable as the parameter $x$ varies.

\begin{prop}[\bf Stability near local minima]\label{prop:stability_local_minima}
Let $h$ be a real valued differentiable function defined on $\Y$ and $y_0$ be a local minimizer of $h$. We assume that $h$ satisfies the {\L}ojasiewicz inequality near $y_0$, meaning that there exists $\mu>0$ and $R>0$ s.t.:

\begin{align}\label{eq:kl_ineq}
\mu(h(y)-h(y_0))\leq \frac{1}{2}\Verts{\partial_y h(y)}^2,\qquad \forall y\in B(y_0,R). 
\end{align}
Let $\mathcal{V}$ be an open neighborhood of $y_0$, $R'>0$ such that $B(y_0,2R')\subset \mathcal{V}$ and  $\mathcal{P}$ a family of diffeomorphisms defined from $\mathcal{V}$ to itself and satisfying: 
	\begin{enumerate}
		\item For any $\tau {\in} \mathcal{P}$, the pre-image $y_{\psi}{:=} \psi^{-1}(y_0)$ of $y_0$ by $\psi$ belongs to $B(y_0,R')$.
		\item There exists positive numbers $\ell$ and $L'$ s.t. for any $\tau\in \mathcal{P}$ and any $y\in \mathcal{V}$:
		\begin{align}\label{eq:non_degeneracy_diffeo}
			\ell^2 I \leq  \partial_y\tau(y)^{\top}\partial_y\tau(y)\leq (L')^2 I.
		\end{align}
	\end{enumerate}
For some $\tau \in \mathcal{P}$, consider a maximal solution $(z_t)$ of the following ODE:
	\begin{align}\label{eq:ODE_boundedness}
		\dot{z_t} = -\partial_y  h(\tau(z_t))\partial_z \tau(z_t) , \qquad z_0\in B\parens{y_{\tau},R'}. 
	\end{align}

Then, there exists $0<C\leq R'$, such that for any $0< \epsilon \leq C$, there exists $0<\eta\leq \frac{\epsilon}{2}$ with the following property:

For any $\tau\in \mathcal{P}$ and any $z_0$ s.t. $\Verts{z_0-y_{\tau}}\leq \eta$:
\begin{enumerate}
	\item The solution $z_t$ to \cref{eq:ODE_boundedness} is well-defined at all times $t\geq 0$.
	\item For all  $t\geq 0$, it holds that $\Verts{z_t-y_{\tau}}\leq \epsilon $.
\end{enumerate}

\end{prop}
\begin{proof}
	The proof is inspired from the the abstract stability result in \cite{Li:2019d}.  We know that $y_0$ is a local minimizer of $h$, therefore there exists $R">0$ such that for any $y$ satisfying $\Verts{y-y_0}\leq R"$, it holds that $\mathcal{L}(y) := h(y)-h(y_0)\geq 0$. Moreover, by \cref{eq:kl_ineq}, we also have that:
\begin{align}\label{eq:KL_stability}
			2\mu \mathcal{L}(y)\leq \Verts{\partial_y \mathcal{L}(y)}^2,\qquad \forall y\in B(y_0,R).
\end{align}
Take $\epsilon < \frac{1}{L'}\min(R,R",L'R'):= C$. To simplify subsequent calculations, we will choose $y$ close enough to $y_0$ so that $2\ell^{-1}\sqrt{\frac{2}{\mu}}\mathcal{L}(y)^{\frac{1}{2}}\leq \epsilon$, where $\mu$ is the positive constant appearing in \cref{eq:KL_stability} and $\ell$ is the positive constant in \cref{eq:non_degeneracy_diffeo}. This is possible by continuity of $\mathcal{L}$ which that there exists $0<\eta\leq \frac{\epsilon}{2}$ for which any $y\in B(y_0,L'\eta)$ satisfies: 
\begin{align}\label{eq:small_L}
	\mathcal{L}(y)^{\frac{1}{2}}\leq \frac{1}{2}\ell\sqrt{\frac{\mu}{2}}\epsilon.
\end{align}
Consider now $\tau\in \mathcal{P}$. Equation \cref{eq:non_degeneracy_diffeo} implies that $\tau$ is $L'$-Lipschitz on $B(y_0,2R')$. Moreover, for any $z$ in  $B(y_{\tau},\eta)$, it holds that $z\in B(y_0,2R')$ since $\eta\leq R'$ and $y_{\tau}\in B(y_0,R')$ by definition of $y_\tau$. Therefore, we can write the following inequality: 
\begin{align}
	\Verts{\tau(z)-y_0} = \Verts{\tau(z)-\tau(y_{\tau})}\leq L' \Verts{z-y_{\tau}}\leq L'\eta,  
\end{align}
We have shown that $\tau(z)\in B(y_0,L'\eta)$ for any $z\in B(y_{\tau},\eta)$, so that \cref{eq:small_L} holds for $\tau(z)$:
\begin{align}
		  \mathcal{L}(\tau(z))^{\frac{1}{2}}\leq \frac{1}{2}\ell\sqrt{\frac{\mu}{2}}\epsilon,
		\qquad \forall z\in B(y_{\tau},\eta),
\end{align}
Additionally, by \cref{eq:KL_stability} and using that $\ell^2\Verts{\partial_y \mathcal{L}(\tau(z))}^2\leq  \Verts{\nabla \mathcal{L}\circ \tau(z)}^2$ by \cref{eq:non_degeneracy_diffeo}, it holds for any $\epsilon <C$ that:  
\begin{align}\label{eq:KL_stability}
	0 \leq 2\mu \mathcal{L}(\tau(z))\leq \Verts{\partial_y \mathcal{L}(\tau(z))}^2\leq \ell^{-2} \Verts{\nabla \mathcal{L}\circ \tau(z)}^2, \qquad \forall z\in B(y_{\tau},\epsilon),
\end{align}
From now on, we fix $\tau$, and consider  $z_t$ to the ODE \cref{eq:ODE_boundedness} with initial condition $z_0\in B(y_{\tau},\eta)$.  Define $\mathcal{T} = \{ t\in \R_{+}|. \Verts{z_s-y_{\tau}} < C\quad  \forall s\in [0,t) \}$ which is not empty by construction since $\Verts{z_0-y_{\tau}}< C$ $s\mapsto z_s$ is continuous. Hence, $t_1 := \sup \mathcal{T}$ is positive. We will show that $t_1=+\infty$. We will also consider the time until which  $\mathcal{L}(\tau(z_t))$ remains positive:  $t^{+} := \sup \{ t\in \R_{+}| \mathcal{L}(\tau(z_s))>0 \forall s\in [0,t) \} $. 
We may assume that $\mathcal{L}(\tau(z_0))>0$ so that $t^{+}>0$ by continuity of the solution $z_t$. The case where $\mathcal{L}(\tau(z_0))=0$ will be treated separately. Denote by $t_1^+ := \min(t_1,t^{+})$ so that, for any $t\in [0,t_1^+)$ the following holds:
\begin{align}
-\frac{\diff \mathcal{L}(\tau(z_t))^{\frac{1}{2}} }{\diff  t} =& \frac{1}{2}\mathcal{L}(\tau(z_t))^{-\frac{1}{2}} \Verts{\nabla \mathcal{L}\circ\tau(z_t)}^2\geq  \ell\sqrt{\frac{\mu}{2}}\Verts{\nabla  \mathcal{L}\circ(\tau(z_t))},
\end{align}
where the first equality follows by differentiating $z_t$ in time and using the ODE equation \cref{eq:ODE_boundedness},  while the last inequality uses the inequality  \cref{eq:KL_stability} which holds since $\Verts{z_t-y_{\tau}} < C$. Integrating between $0$ and $t\in [0,t_1^+)$, we get:
\begin{align}
	\mathcal{L}(\tau(z_0))^{\frac{1}{2}} -\mathcal{L}(\tau(z_t))^{\frac{1}{2}}\geq \ell\sqrt{\frac{\mu}{2}}\int_0^t\Verts{\nabla  \mathcal{L}\circ\tau(z_s)}\diff s. 
\end{align}
Since $\Verts{z_0-y_{\tau}}\leq \eta$ and using \cref{eq:small_L}, it holds that $\mathcal{L}(z_0)^{\frac{1}{2}}\leq \frac{\ell}{2}\sqrt{\frac{\mu}{2}}\epsilon$. We can therefore deduce that $\int_0^t\Verts{\nabla  \mathcal{L}\circ\tau(z_s)}\diff s\leq \frac{\epsilon}{2}$. This allows to write for all $t\in [0,t_1^+)$
\begin{align}\label{eq:main_ineq_stability}
	\Verts{z_t-y_{\tau}} \leq & \Verts{z_t-z_0} + \Verts{z_0-y_{\tau}},\\ 
	\leq & \int_0^t \Verts{\nabla  \mathcal{L}\circ\tau(z_s)}\diff s + \eta \leq \epsilon.
\end{align}

We distinguish two cases depending on whether $t^+<t_1$ or $t_1\leq t^+$.

{\bf Case 1: $t^+<t_1$ or}. 
In this case we have $t_1^+=t^+ < +\infty$. This case also accounts for when $\mathcal{L}\circ\tau(z_0){=}0$ which implies that $t^+ = 0<t_1$. If $t^+ {=} 0$, then $\Verts{z_{t^+}-y_{\tau}}\leq \epsilon$ by construction. Otherwise, we still have that $\Verts{z_{t^+}-y_{\tau}}\leq \epsilon$ by \cref{eq:main_ineq_stability} and the continuity of $z_t$ at $t^+$. Moreover, by definition of $t^+$, it must also hold that $\mathcal{L}\circ\tau(z_{t^+})=0$. We only need to show that $\nabla   \mathcal{\mathcal{L}}\circ\tau(z_{t^+})=0$. By contradiction, if $\nabla   \mathcal{\mathcal{L}}\circ\tau(z_{t^+})\neq 0$, then we would have $\mathcal{L}\circ\tau(z_{t^+ + s})<0$ for $s>0$ small enough. However, since $t^+< t_1$, then $t^+ +s<t_1$ for $s$ small enough, so that $\Verts{z_{t^+ +s}-y_{\tau}}< C$. The latter means that $\mathcal{L}\circ\tau(z_{t^+ +s})\geq 0$ since $y_0$ is a local minimizer of $\mathcal{L}$. This contradicts $\mathcal{L}\circ(z_{t^+ + s})<0$. Therefore $\nabla \mathcal{\mathcal{L}}\circ\tau(z_{t^+}){=} 0$  which implies that $\tau(z_{t^+})$ is a critical point of $y\mapsto\mathcal{L}(y)$ so that $z_{t}= z_{t^+}$ for any $t\geq t^{+}$. This directly means that $\Verts{z_t-y_{\tau}}\leq \epsilon$ for any $t\geq 0$, hence $t_1=+\infty$.

{\bf Case 2: $t^+\geq t_1$.} In this case, $t_1^+=t_1$. If by contradiction we had $t_1<+\infty$, then we would directly get $\Verts{z_{t_1}-y_{\tau}}\leq \epsilon$ by continuity of $t$ at $t_1$ and maximality of the solution $z_t$. However, by defintion of $t_1$, we also have $\Verts{z_{t_1}-y_{\tau}}=C$. This contradicts the condition $\epsilon < C$ and therefore  means that $t_1=+\infty$. Hence,  it holds that $\Verts{z_t-y_{\tau}}\leq \epsilon$ for any $t\geq 0$ and that the solution $z_t$ is well-defined at all times.

\end{proof}

\subsection{Continuity of the Flow Selection}\label{sec:continuity_flow_selection} 

\cref{prop:continuity_at_point}
shows that $x\mapsto\phi(x,y)$ is continuous at $x_0$ whenever $\phi(x_0,y)$ is a local minimum of $g(x_0,.)$. 
\cref{prop:stability_of_local_minima}
shows that, near $x_0$, $\phi(x,y)$ are local minima as well provided $\phi(x_0,y)$ is a local minimum of $g(x_0,.)$. 
\begin{prop}[\bf Continuity near local minima]\label{prop:continuity_at_point}
Let $x_0\in \X$ and $y\in \Y$. 
Let $g$ be such that \cref{assumpt:Coercivity,assumpt:morse-bott,assumpt:smootness} hold. Assume that $y_0 = \phi(x_0,y)$ is a local minimizer of $y\mapsto g(x_0,y)$. 
Then, for any $\epsilon>0$ small enough, there exists $T>0$ and $\eta>0$, s.t.:
\begin{align}
	\Verts{\phi_t(x,y)-\phi(x_0,y)}\leq \epsilon, \qquad \forall t\geq T, \quad \forall x\in B(x_0,\eta).
\end{align}
In particular, $x\mapsto \phi(x,y)$ is continuous at $\phi(x_0,y)$.
\end{prop}
\begin{proof}
We will apply \cref{prop:stability_local_minima} to the function $h(y) = g(x_0,y)$ and the well-chosen family $\mathcal{P}$ of local diffeomorphisms on $\Y$. By application of \cref{cor:existence_diffeomorphism}, 
there exists a open neighborhoods $\mathcal{B}$ and $\mathcal{V}$ of $x_0$ and $y_0$ in $\X$ and $\Y$ and a continuously differentiable map $\tau$ from $\mathcal{B}\times \mathcal{V}$ to $\mathcal{V}$ such that $y\mapsto \tau(x,y)$ is a diffeomorphism from $\mathcal{V}$ onto itself and for which $g$ satisfies for any $(x,y)\in \mathcal{B}\times \mathcal{V}$:
\begin{align}
	g(x,y) = g(x_0,\tau(x,y)) + C(x).
\end{align}
For simplicity, we write $\tau_x: y\mapsto \tau(x,y)$ by an abuse of notations. 
We know, by \cref{cor:existence_diffeomorphism},  that $x\mapsto \tau_x^{-1}(y_0)$ is continuous and converges to $\tau_{x_0}^{-1}(y_0)=y_0$. Hence, by restricting $x$ to a smaller neighborhood $\mathcal{B}'\subset \mathcal{B}$, we can ensure that $\tau_x^{-1}(y_0)$ belongs to $B(y_0,R')$ with $R'$ small enough so that $B(y_0,2R')\subset \mathcal{V}$. Consider now the family of diffeomorphisms $\mathcal{P}$
\begin{align}
	\mathcal{P} = \braces{ \mathcal{V} \ni y\mapsto \tau(x,y)\in \mathcal{V}  \middle| x\in \mathcal{B}'}.
\end{align}
We have constructed $\mathcal{P}$ satisfying the conditions of \cref{prop:stability_local_minima}. Moreover, by \cref{prop:uniform_KL}, the function $h(y){:=} g(x_0,y) $ satisfies 
a {\L}ojasiewicz inequality in an open neighborhood $\mathcal{V}'$ of $y_0$:
\begin{align}
	\mu\verts{h(y)-G(x_0)}\leq \Verts{\partial_y h(y)}^2, \forall y\in \mathcal{V}'.
\end{align}
We can always choose the neighborhood $\mathcal{V}'$ to be an open ball $B(y_0,R)$ of radius $R>0$ centered in $y_0$.
Therefore, we have shown so far that $h$ and $\mathcal{P}$ satisfy the conditions of \cref{prop:stability_local_minima}. 

For any $\tau\in \mathcal{P}$, consider the ODE:
\begin{align}\label{eq:ode_z_t}
	z_t = -\partial_y g(x,\tau(z_t)),\qquad z_0\in B(y_0,R').
\end{align}
Following the notation in \cref{prop:stability_local_minima}, we define $y_{\tau} :=\tau^{-1}(y_0)$ for any $\tau \in \mathcal{P}$. We apply \cref{prop:stability_local_minima} which ensures stability of $z_t$. More precisely, there exists a positive constant $C$ smaller than $R'$ so that  for any $0<\epsilon<C$, the solution $z_t$ is well-defined at all times and satisfies $\Verts{z_t-y_{\tau}}\leq \epsilon$ for any $t\geq 0$, provided that the initial condition $z_0$ satisfies $\Verts{z_0-y_{\tau}}\leq \eta$ for some positive $\eta<\frac{\epsilon}{2}$ that is independent of the choice of $\tau$ is $\mathcal{P}$:
\begin{align}\label{eq:main_stability_ineq}
\forall \tau\in \mathcal{P}: \Verts{z_0-y_{\tau}}\leq \eta\implies \Verts{z_t-y_{\tau}}\leq \epsilon.
\end{align}
We will apply this result to a particular choice for $z_0$. From now on, we fix $0<\epsilon<C$ and let $0<\eta\leq \frac{\epsilon}{2}$ be as in \cref{prop:stability_local_minima}.  Using  \cref{prop:convergence_flow}, we know that $\phi_t(x_0,y)$ converges to $y_0=\phi(x_0,y)$, hence there exits $T>0$ s.t. $\Verts{\phi_T(x_0,y)-y_0}\leq \frac{\eta}{3}$. 
Moreover, since the maps $x\mapsto \phi_T(x,y)$ and $x\mapsto y_{\tau_x}$ are continuous at $x_0$ with $y_{\tau_{x_0}} {=} y_0$, there exits $\eta'$ satisfying $0<\eta'$   such that $B(x_0,\eta')\subset \mathcal{B}'$ and $\Verts{\phi_T(x,y)- \phi_T(x_0,y)}\leq \frac{\eta}{3}$ and $\Verts{y_{\tau_x}-y_0}\leq \frac{\eta}{3}$ for any $x\in B(x_0,\eta')$. Therefore:
\begin{align}
	\Verts{\phi_T(x,y)- y_{\tau_x}}\leq 
		\Verts{\phi_T(x,y)- \phi_T(x_0,y)} + \Verts{\phi_T(x_0,y)- y_0} + \Verts{y_0- y_{\tau_x}}\leq \eta. 
\end{align}
For any $x\in B(x_0,\eta')$, by choosing $z_0 = \phi_T(x,y)$, we have that $\Verts{z_0-y_{\tau_x}}\leq \eta$. 
Therefore, we deduce by \cref{eq:main_stability_ineq} that $\Verts{z_t-y_{\tau_x}}\leq \epsilon$ and subsequently that 
\begin{align}
	\Verts{z_t-y_0}\leq \Verts{z_t-y_{\tau_x}}+\Verts{y_{\tau_x}-y_0}\leq \epsilon +\frac{\eta}{3}\leq \frac{7}{6}\epsilon,
\end{align} 
since we imposed that $\eta<\frac{\epsilon}{2}$. Recall now that $z_t$ satisfies the ODE:
\begin{align}
	\dot{z}_t = -\partial_y g(x_0,\tau_x(z_t))\partial_z \tau_x(z_t).
\end{align}
By definition of $\tau_x$, we have $g(x,y) = g(x_0,\tau_x(y))$ for any $y\in B(y_0,2R')\subset \mathcal{V}$. In particular, as we have shown that $\Verts{z_t-y_0}\leq \frac{7}{6}\epsilon <2R'$, it follows  that $z_t$ satisfies the ODE:
\begin{align}
	\dot{z}_t = -\partial_y g(x,z_t) = -\partial_y g(x_0,\tau_x(z_t))\partial_z \tau_x(z_t).
\end{align}
By Cauchy-Lipschtz theorem, the solution of the above ODE is unique. Moreover, since we know that $\phi_{T+t}(x,y)$ is a solution to the above ODE, then we deduce that $z_t = \phi_{T+t}(x,y)$. We have shown that for any $\epsilon<C$, there exists $T>0$ and $\eta'$ such that:
\begin{align}\label{eq:stability_1}
	\Verts{\phi_t(x,y)-y_0}\leq \frac{7}{6}\epsilon, \qquad \forall t\geq T, \quad\forall x\in B(x_0,\eta').
\end{align}
Since $\phi_t(x,y)$ converges towards $\phi(x,y)$ by \cref{prop:convergence_flow}, taking the limit $t{\rightarrow} \infty$ in \cref{eq:stability_1}, we obtain:
\begin{align}
	\Verts{\phi(x,y)-y_0}\leq \frac{7}{6}\epsilon, \forall x\in B(x_0,\eta').
\end{align}
The above inequality imply in particular that $x\mapsto\phi(x,y)$ is continuous at $x_0$.
\end{proof}

\begin{prop}[\bf Stability of local minimizers]\label{prop:stability_of_local_minima}
	Let $x_0\in \X$ and $y\in \Y$ and $g$ be such that \cref{assumpt:Coercivity,assumpt:morse-bott,assumpt:smootness} hold. Assume that $y_0 = \phi(x_0,y)$ is a local minimizer of $y\mapsto g(x_0,y)$. Then, for any $x_1$ in a neighborhood of $x_0$, $\phi(x_1,y)$ is a local minimizer of $g(x_1,.)$.
\end{prop}
\begin{proof}
By assumption, $y_0:=\phi(x_0,y)$ is a local minimizer of $g(x_0,.)$ ensuring that $\partial_{yy}^2 g(x_0,y_0)$ is positive semi-definite. Moreover, by \cref{cor:existence_diffeomorphism}, there exists a neighborhood $B(x_0,\eta)\times B(y_0,2R')$ of $(x_0,y_0)$ such that for any augmented critical point $(x_1,y_1)\in \mathcal{M}\cap B(x_0,\eta)\times B(y_0,2R')$, the  Hessian $\partial_{yy}^2 g(x_1,y_1)$ is similar to $\partial_{yy}^2 g(x_0,y_0)$. Hence, for any $(x_1,y_1)\in \mathcal{M}\cap B(x_0,\eta)\times B(y_0,2R')$, $\partial_{yy}^2 g(x_1,y_1)$ must be positive semi-definite so that $y_1$ is a local minimizer of $g(x_1,.)$.

We can then apply \cref{prop:continuity_at_point} which ensures that $x\mapsto \phi(x,y)$ is continuous at $x_0$. Therefore, there exists $\eta'<\eta$ so that, for any $x_1\in B(x_0,\eta')$,  $y_1:=\phi(x_1,y)$ belongs to $B(y_0,2R')$. As a result, $y_1$ must be a local minimizer of $g(x_1,.)$ since the augmented critical point $(x_1,y_1)$ belongs to $\mathcal{M}\cap B(x_0,\eta)\times B(y_0,2R')$.
\end{proof}

\subsection{Uniform Convergence of the Gradient Flow}\label{sec:unif_convergence_flow}
The result bellow shows that the gradient flow $\phi_t(x,y)$ converges locally uniformly in $x$ near $x_0$ at an exponential rate, whenever $\phi(x_0,y)$ is a local minimum. It relies on the locally uniform convergence result in  \cref{prop:continuity_at_point} and the locally uniform {\L}ojasiewicz inequality in \cref{prop:uniform_KL}.  

\begin{prop}\label{prop:unif_convergence_flow}
	Let $x_0\in \X$ and $y\in \Y$ and $g$ be such that \cref{assumpt:Coercivity,assumpt:morse-bott,assumpt:smootness} hold. Assume that $y_0:=\phi(x_0,y)$ is a local minimum. Then there exists positive constants $\eta$, $T$, $\mu$ and $C$ such that: 
\begin{align}
	\Verts{\phi_t(x,y)-\phi(x,y)}\leq Ce^{-t\mu}, \qquad \forall t\geq T, x\in B(x_0,\eta).
\end{align}
A fortiori, $\phi(x,y)$ is continuous on $B(x_0,\eta)$. 
\end{prop}
\begin{proof}

\cref{prop:uniform_KL} ensures the existence of  $\epsilon>0$ and $\eta>0$ be such that the following  inequality holds:
\begin{align}\label{eq:kl_gen}
	\mu\verts{g(x,y')-G(x)}\leq \frac{1}{2}\Verts{\partial_yg(x,y')}^2, \qquad \forall x,y'\in B(x_0,\eta')\times B(y_0,\epsilon).
\end{align}
By \cref{prop:continuity_at_point} and for $\epsilon{>}0$ small enough, there exists $T{>}0$ and $\eta'>\eta{>}0$ for which:
\begin{align}\label{eq:stab_1}
	\Verts{\phi_t(x,y)-y_0}\leq \epsilon, \qquad \forall t\geq T,\quad \forall x\in B(x_0,\eta'). 
\end{align}
Therefore, choosing $y' {=} \phi_t(x,y)$ in \cref{eq:kl_gen} implies:
\begin{align}\label{eq:KL_ineq}
	\mu\verts{g(x,\phi_t(x,y))-G(x)}\leq \frac{1}{2}\Verts{\partial_yg(x,\phi_t(x,y))}^2, \qquad \forall x\in B(x_0,\eta). 
\end{align}
Note that $G(x)$ is the common value of $g(x,y)$ when $y$ is a critical point of $g(x,.)$ in $B(y_0,\epsilon)$. In particular, since $\phi(x,y)\in B(y_0,\epsilon)$,  it holds that  $G(x)= g(x,\phi(x,y))$. Moreover, by \cref{prop:stability_of_local_minima}, $\phi(x,y)$ is a local minimum for $y\mapsto g(x,y)$. Hence, we must have $g(x,\phi_t(x,y))-G(x)\geq 0$. We may assume that the inequality is strict otherwise the $\phi_t(x,y)$ would be a fixed point and we would have $\phi_t(x,y)=\phi(x,y)$. The following inequality holds for any $t\geq T$:
\begin{align}
	\Verts{\phi_t(x,y)-\phi(x,y)}\leq \int_t^{+\infty}\Verts{\partial_y g(x,\phi_s(x,y))}\diff s \leq -\frac{2}{\mu}\int_t^{+\infty} 
	\dot{H}(s)\diff s =  \frac{2}{\mu}H(t).
\end{align}
where we introduced $H(t) = \parens{g\parens{x,\phi_t(x,y)}- G(x)}^{\frac{1}{2}}$. Thus, we only need to study the evolution of $H(t)$ in time. Computing the derivatives of $H(t)$ and using the inequality in \cref{eq:KL_ineq} yields
$$
\dot{H}(t) = -\frac{1}{2}H(t)^{-1} \Verts{\partial_y g(x,\phi_t(x,y))}^2\leq -\mu H(t). 
$$
By integrating the above inequality, it follows that $H(t)\leq H(T)e^{-\mu(t-T)}$. 
Moreover, using the smoothness of $y\mapsto g(x,y)$, we know that 
$$
H(T)\leq \sqrt{\frac{L}{2}}\Verts{\phi_T(x,y)-\phi(x,y)}\leq \sqrt{\frac{L}{2}}\epsilon, \forall x\in B(x_0,\eta).
$$
Finally, we have shown that $\Verts{\phi_t(x,y){-}\phi(x,y)}\leq \sqrt{\frac{L}{\mu}}\epsilon e^{-(t-T)\mu}$ for any $x\in B(x_0,\eta)$ and $t{\geq} T$. Since $\phi_t$ are continuous in $x$ and converge uniformly in $x$ on $B(x_0,\eta)$, then their limit must be continuous on $B(x_0,\eta)$.  
\end{proof}

\section{Differentiability of the Flow Selection}\label{sec:differentiability-selection}
In this section, we study the differentiability of $x{\mapsto}\phi(x,y)$ through the evolution of $\partial_x\phi_t(x,y)$. The following result establishes that $\partial_x\phi_t(x,y)$ is well-defined and satisfies a linear differential equation.
\begin{prop}\label{eq:continuity_U_t}
Assume $g$ is twice continuously differentiable and satisfies \cref{assumpt:smootness}.  Then, $(x,t)\mapsto\phi_t(x,y)$ is continuously differentiable with $\partial_x \phi_t(x,y) := U_t(x,y)$ satisfying the differential equation:
\begin{align}\label{eq:evolution_U}
	\dot{U}_t(x,y) = -B_t(x,y) -A_t(x,y)U_t(x,y),
\end{align}
where $B_t$ and $A_t$  are given by:
\begin{align}
	B_t(x,y)= \partial_{xy}^2 g(x,\phi_t(x,y)),\qquad A_t(x,y) = \partial_{yy}^2 g(x,\phi_t(x,y)).
\end{align}
\end{prop}
\begin{proof}
	The differentiability of the flow $\phi_t(x,y)$ in $x$ follows by the application of Cauchy-Lipschitz theorem. It suffices to differentiate the equation defining the flow w.r.t. to obtain \cref{eq:evolution_U}.
\end{proof}
Note that, by \cref{prop:convergence_flow} and continuity of $\partial_{xy}^2g(x,y)$ and $\partial_{yy}^2 g(x,y)$, the matrices $A_t(x,y)$ and $B_t(x,y)$ must converge to the following matrices $A_{\infty}$ and $B_{\infty}$ for any $(x,y)\in \X\times \Y$:
\begin{align}
	A_{\infty}(x,y) := \partial_{yy}^2 g(x,\phi(x,y)),\qquad B_{\infty}(x,y) := \partial_{xy}^2 g(x,\phi(x,y))
\end{align}
The following proposition shows that the pseudo-inverse of $A_{\infty}(x,y)$ remain bounded near $x_0$ provided that $\phi(x_0,y)$ is a local minimum.
\begin{prop}\label{prop:boundedness_U}
	Let $(x_0,y)$ be in $\X\times \Y$ and set $y_0$ and $g$ be such that \cref{assumpt:Coercivity,assumpt:morse-bott,assumpt:smootness} hold. Assume that $y_0{:=}\phi(x_0,y)$ is a local minimum of $g(x_0,.)$. Then there exists an open neighborhood $\mathcal{U}$ of $x_0$ and a positive constant $\lambda>0$ such that:
	\begin{align}
		\lambda \Verts{A_{\infty}(x,y)}_{op}\leq 1,\qquad \forall x\in \mathcal{U}, 
	\end{align}
	where $A_{\infty}(x,y) = \partial_{yy}^2 g(x,\phi(x,y))$. 
\end{prop}
\begin{proof}
	We apply \cref{corr:similar_hessian} which ensures the existence of an open neighborhood $\mathcal{V}$ of $(x_0,y_0)$ for which:
	\begin{align}
		\partial_{yy}^2g(x,y) = P(x,y)^{\top}\partial_{yy}^2g(x_0,y_0)P(x,y).
	\end{align}
where $(x,y)\mapsto P(x,y)$ is continuous map with values in $\R^{d\times d}$, and $P(x,y)$ is invertible for any $(x,y)\in \mathcal{V}$ with singular values in $[\sigma_{\min},\sigma_{\max}]$ for 	$\sigma_{\min}>0$ and $\sigma_{\max}<+\infty$. Moreover, since $y_0:= \phi(x_0,y)$ is a local minimum of $g(x_0,.)$, we know, by \cref{prop:continuity_at_point}, that $x\mapsto \phi(x,y)$ is continuous at $x_0$. Hence, there exists a neighborhood $\mathcal{U}$ of $x_0$ for which $(x,\phi(x,y))\in \mathcal{V}$ for any $x\in \mathcal{U}$. Therefore, it follows that:
\begin{align}
	A_{\infty}(x,y) = P(x,\phi(x,y))^{\top}\partial_{yy}^2 g(x_0,y_0)P(x,\phi(x,y)),\qquad \forall x\in \mathcal{U}.
\end{align} 
In particular, it follows that:
\begin{align}
	A_{\infty}(x,y)^{\dagger} = P(x,\phi(x,y))^{-1}\partial_{yy}^2 g(x_0,y_0)^{\dagger}P(x,\phi(x,y))^{-\top}.
\end{align}
Hence, we easily deduce that the operator norm of $A_{\infty}(x,y)^{\dagger}$ satisfies:
\begin{align}
	\Verts{A_{\infty}(x,y)^{\dagger}}_{op}\leq \sigma_{\min}^{-2}\Verts{\partial_{yy}^2 g(x_0,y_0)^{\dagger}}_{op}.
\end{align}
The result follows by setting $\lambda = \sigma_{\min}^2  \Verts{\partial_{yy}^2 g(x_0,y_0)^{\dagger}}_{op}^{-1}$.
\end{proof}
We will need to introduce the following matrix $U^{\star}(x,y)$ defined as:
\begin{align}
	U^{\star}(x,y) := -\parens{A_{\infty}(x,y)}^{\dagger}B_{\infty}(x,y).
\end{align}
The following proposition shows, under mild conditions, that $U_t(x,y)$ converges towards a limiting element $U_{\infty}(x,y)$ satisfying the equation: $A_{\infty}U_{\infty}= A_{\infty}U^{\star}$. 
\begin{prop}\label{prop:convergence_U}
	Let $(x_0,y)$ be in $\X\times \Y$ and set $y_0$ and $g$ be such that \cref{assumpt:Coercivity,assumpt:morse-bott,assumpt:smootness} hold. Assume that $y_0{:=}\phi(x_0,y)$ is a local minimum of $g(x_0,.)$. 
	Then there exists $\eta>0$ such that, for any $x\in B(x_0,\eta)$,  $U_t(x,y)$ converges towards an element $U_{\infty}(x,y)$ satisfying 
	$$A_{\infty}(x,y)U_{\infty}(x,y) {=} A_{\infty}(x,y)U^{\star}(x,y).
	$$
In paticular, if $y$ is a critical point of $y\mapsto g(x,.)$ then $U_{\infty}(x,y):= U^{\star}(x,y)$. 
Moreover, there exists a time $T>0$ and constants $C>0$, $\mu$ such that for any $x\in B(x_0,\eta)$ and $t\geq T$:
\begin{align}
	\Verts{U_t(x,y)- U_{\infty}(x,y)}
	\leq & Ce^{-\mu t},
\end{align}
\end{prop}
\begin{proof}
For simplicity, we omit the dependence on $(x,y)$ as they remain fixed.
Let $P$ be a projection matrix that commutes with $A_{\infty}$, i.e. : $PA_{\infty}=A_{\infty}P$. We will choose $P$ to be either $P_{\infty} = A_{\infty}A_{\infty}^{\dagger}$ or $P = I-P_{\infty}$. Define $V_t = P(U_t-U^{\star})$. By differentiating in time, it is easy to see that $V_t$ satisfies:
\begin{align}\label{eq:projected_dynamics}
\dot{V}_t = \tilde{B}_t - PA_t V_t.
\end{align}
where $\tilde{B}_t := P\parens{B_{\infty}-B_t + (A_{\infty}-A_t)U^{\star}}$. Denote by $(s,t)\mapsto R_s^t$ the resolvant of the linear system \cref{eq:projected_dynamics}, i.e. the squared matrix satisfying $\frac{\diff R_s^t}{\diff t} = -PA_t R_s^t$ for $t\geq s$ and $R_s^s = I$ . Standard results for linear differential equations ~\citep[Chapter 2]{Robinson:2012}  ensure that $R_s^t$ is always invertible at any time and that $V_t$ can be expressed in terms of $R_s^t$ as follows:
$$
V_t  = -R_0^tPU^{\star} + \int_0^t R_s^t\tilde{B}_s \diff s.
$$
{\bf Controlling $\Verts{R_s^t}_{op}$: }

We will show the following inequality:
\begin{align}\label{eq:bound_Resolvant}
	\log\parens{\Verts{R_s^t}_{op}}\leq \int_s^t \parens{-\lambda_P + \Verts{A_{\infty}-A_u}}\diff u,
\end{align}
 where $\Verts{.}_{op}$ refers to the operator norm and $\lambda_P$ is the smallest eigenvalue of $PA_{\infty}P$. To achieve this, we define $\mathcal{L}_t = \frac{1}{2}\Verts{R_s^tu}^2$ for $t\geq s$ and  $u$ a vector in $\Y$. We then   differentiate $\mathcal{L}_t$  in time to get:
\begin{align}
\dot{\mathcal{L}}_t &= -\langle R_s^tu,PA_t R_s^tu \rangle 
= -\langle R_s^tu,PA_{\infty} R_s^tu \rangle + \langle R_s^tu,P\parens{A_{\infty}-A_t} R_s^tu \rangle,
\\
&= -\langle R_s^tu,PA_{\infty}P R_s^tu \rangle + \langle R_s^tu,P\parens{A_{\infty}-A_t} R_s^tu \rangle,\\
	&\leq 2\parens{ -\lambda_P + \Verts{A_{\infty}-A_t}_{op}} \mathcal{L}_t,
\end{align}
where we used that $PA_{\infty}P{=}P^2A_{\infty}{=}PA_{\infty}$ since $P$ and $A_{\infty}$ commute. We also used elementary properties of the trace of product of matrices to get the last inequality. By integrating the above inequality, we obtain:
\begin{align}
\frac{1}{2}\Verts{R_s^t u}^2 \mathcal{L}_t \leq & \frac{1}{2}\Verts{R_s^s u}^2 e^{2\int_s^t \parens{-\lambda_P + \Verts{A_{\infty}-A_u}_{op}}\diff u},\\ 
\leq & \frac{1}{2}\Verts{u}^2 e^{2\int_s^t \parens{-\lambda_P + \Verts{A_{\infty}-A_u}_{op}}\diff u},
\end{align}
where we used that $R_s^s =I$. The desired bound on $\Verts{R_s^t}_{op}$ follows   by taking the supremum over $u$ in the unit ball.

{\bf Controlling $\Verts{B_{\infty}-B_t}_{op}$ and $\Verts{A_{\infty}-A_t}_{op}$:}

By \cref{prop:unif_convergence_flow}, there exists $\eta>0$ and $T>0$ such that:
\begin{align}
	\Verts{\phi_t(x,y)-\phi(x,y)}\leq Ce^{-t\mu}, \qquad \forall t\geq T, x\in B(x_0,\eta).
\end{align}
Moreover, since $\phi(x,y)$ is continuous at $x_0$ by \cref{prop:continuity_at_point}, we can always choose $\eta$ small enough so that $\phi(x,y)$ remains bounded. Hence, there exists a compact set $K$ containing $\phi_t(x,y)$ for any $t\geq T$ and $x\in B(x_0,\eta)$. Denote by $\verts{K}$ its diameter. By continuity of $\phi_t(x,y)$, we can also take $K$ large enough so that $\phi_t(x,y)\in K$ for any $0\leq t\leq T$ and $x\in B(x_0,\eta)$. 
Since $g$ is three-times continuity differentiable by \cref{assumpt:morse-bott}, there exists a positive constant $L$ s.t. for all $x\in B(x_0,\eta)$ and $y,y'\in K$: 
\begin{align}\Verts{\partial_{xy}^2 g(x,y)-\partial_{xy}^2 g(x,y')}_{op},&\leq L\Verts{y-y'},\\
\Verts{\partial_{yy}^2 g(x,y)-\partial_{yy}^2 g(x,y')}_{op}&\leq L\Verts{y-y'}. 
\end{align}
As a result, we can write 
\begin{align}\label{eq:error_AB}
	\max\parens{\Verts{B_{\infty}-B_t}_{op},\Verts{A_{\infty}-A_t}_{op}}\leq L\Verts{\phi(x,y)-\phi_t(x,y)}\leq c_t.
\end{align} 
where, we defined $c_t$ to be:
\begin{align}
	c_t = \begin{cases}
		LCe^{-t\mu},\qquad t\geq T,\\
		2L\verts{K}, \qquad t<T.
	\end{cases}
\end{align}

{\bf Controlling $V_t$:}
For simplicity define $C_t = \int_0^t c_u \diff u \leq C_{\infty}:= 2L\verts{K}T + LCe^{-T\mu}/\mu $.
We will first control the error term $\int_0^t R_s^t \tilde{B}_s\diff s$. For $t>T$, the following holds:

\begin{align}\label{eq:bound_RB}
	\int_0^{t}\Verts{R_s^t\tilde{B}_s}_{op}\diff s\leq & 
	\int_0^t \Verts{R_s^t}_{op}\parens{\Verts{B_{\infty}-B_t} + \Verts{A_{\infty}-A_t}\Verts{U^{\star}}_{\infty}}\diff s,\\
	\leq & \int_0^t e^{\int_s^t -\lambda_P + c_u\diff u} \parens{1+\Verts{U}_{op}^{\star}} c_s\diff s,\\
	\leq & e^{C_{\infty}} \parens{1+\Verts{U}_{op}^{\star}}\int_0^t c_s e^{-(t-s)\lambda_P}\diff s, 
\end{align}
where we used elementary linear algebra inequalities for the first line and \cref{eq:bound_Resolvant,eq:error_AB} for the second line. 
We need to control $\Verts{U^{\star}(x,y)}_{op} {=}  \Verts{A_{\infty}(x,y)^{\dagger}B_{\infty}(x,y)}_{op}$. To achieve this, we use \cref{prop:boundedness_U} which ensures that $\Verts{A_{\infty}(x,y)^{\dagger}}_{op}\leq \lambda^{-1}$ for some positive $\lambda$ provided $x$ is close enough to $x_0$. Thus, we can choose $\eta$ small enough so that $\Verts{A_{\infty}(x,y)^{\dagger}}_{op}\leq \lambda^{-1}$ for any $x\in B(x_0,\eta)$. Moreover, by \cref{prop:unif_convergence_flow}, we know that $x\mapsto \phi(x,y)$ is continuous on $B(x_0,\eta)$ provided $\eta$ is small enough. Therefore, we can ensure that $B_{\infty}(x,y)$ is bounded by some value $B_{max}$ on $B(x_0,\eta)$. Hence, we deduce that $\Verts{U^{\star}(x,y)}_{op}\leq M = \lambda^{-1}B_{\max}$ for any $x\in B(x_0,\eta)$. We can finally write the upper-bound bellow:
 \begin{align}\label{eq:bound_RB}
	\int_0^{t}\Verts{R_s^t\tilde{B}_s}_{op}\diff s
	\leq & e^{C_{\infty}} \parens{1+M}\underbrace{\int_0^t c_s e^{-(t-s)\lambda_P}\diff s}_{E_t}. 
\end{align}
We distinguish two cases depending on the choice of $P$:

\begin{itemize}
	\item {\bf Case $P = A_{\infty}A_{\infty}^{\dagger}$.} 
\end{itemize}
In the case where $A_{\infty}(x_0,y){=}0$, then by \cref{corr:similar_hessian} and for $\eta>0$ small enough, it holds that $A_{\infty}(x,y)=0$ for any $x\in B(x_0,\eta)$. In this case, the dynamics is trivial. 
Instead, if $A_{\infty}(x_0,y){\neq}0$, then by \cref{corr:similar_hessian} and for $\eta>0$ small enough, $A_{\infty}(x,y)\neq 0$  for any $x\in B(x_0,\eta)$. In this case, we know that $\Verts{A_{\infty}(x,y)}_{op}$  is the inverse of the smallest positive eigenvalue of $A_{\infty}(x,y)$ which is also equal  to $\lambda_P$ by definition. Moreover, by \cref{prop:boundedness_U}, there exists $\eta>0$ small enough and $\lambda>0$ such that $\lambda \Verts{A_{\infty}(x,y)}_{op}\leq 1$ for any $x\in B(x_0,\eta)$. We then deduce that  $\lambda<\lambda_P$. Hence, for $t\geq T$, we have:
\begin{align}
E_t  =& c_T\int_0^{T} e^{-\lambda(t-s)} + LC\int_{T}^{t} e^{-\lambda(t-s) -(s-T)\mu},\\
= & \frac{c_T}{\lambda}e^{-\lambda(t-T)} + 
\frac{LC}{\lambda - \mu}\parens{e^{-\mu(t-T)} - e^{-\lambda(t-T)}}. 
\end{align}
By abuse of notation, we still write $\frac{1}{\lambda - \mu}\parens{e^{-\mu(t-T)} - e^{-\lambda(t-T)}}$ even when  when $\lambda=\mu$, to refer to the limit $(t-T)e^{-\lambda(t-T)}$ when $\mu$ approaches $\lambda$. By introducing $\tilde{\mu} = \frac{1}{2}\min(\lambda,\mu)$, we get the simpler bound:
\begin{align}
	E_t\leq \frac{(c_T+LC)}{\tilde{\mu}}e^{-\tilde{\mu}(t-T)}.
\end{align}
On the other hand, recalling the upper-bound on $\Verts{R_s^t}_{op}$ we deduce that $\Verts{R_0^t}_{op}\leq e^{C_{\infty} -\lambda_P t}$. Hence, we can write for any $t\geq T$:
\begin{align}
	\Verts{V_t} \leq &  e^{ C_{\infty}}\parens{1+M}\parens{e^{-\lambda_P t}  + \frac{c_T+LC}{\tilde{\mu}} e^{-\tilde{\mu}(t-T)}},\\
	\leq & e^{C_{\infty}}\parens{1+M}\parens{1+\frac{c_T+LC}{\tilde{\mu}}e^{\tilde{\mu}T} }e^{-\tilde{\mu}t}.
\end{align}
Hence, $V_t$ converges towards $0$ at an exponential rate.
\begin{itemize}
	\item {\bf Case $P = I-A_{\infty}A_{\infty}^{\dagger}$.}
\end{itemize}
In this case, $\lambda_P=0$ and $PU^{\star}=-PA_{\infty}^{\dagger}B_{\infty} = 0$. Therefore, $V_t$ simplifies to $V_t{=} \int_0^t R_s^t \tilde{B}_s\diff s$. We will simply show that such integral is absolutely convergent.  To achieve this, we consider $t\geq T$ and compute $E_t$:
\begin{align}
	E_t = & \int_0^t c_s \diff s = \int_0^{T} c_s \diff s  + \int_{T}^t c_s \diff s,\\
	=& T c_T + LC\int_{T}^t e^{-\mu(s-T)}\diff s,\\
	=& Tc_T + \frac{LC}{\mu}\parens{1- e^{-\mu(t-T)}}\leq Tc_T + \frac{LC}{\mu}:=E_{\infty}.
\end{align}
Hence, $E_t$ converges to a finite quantity $E_{\infty}$.  
Using \cref{eq:bound_RB}, we deduce that $\int_0^t R_s^t \tilde{B}_s\diff s$ is absolutely convergent so that $V_t$ converges to an element $V_{\infty}$. Moreover, we have:
\begin{align}
	\Verts{V_t-V_{\infty}} \leq   \int_t^{\infty}\Verts{R_s^t\tilde{B}_s}\diff s \leq & e^{C_{\infty}}\parens{1+M}\parens{E_{\infty}-E_t},\\
	 \leq & C e^{C_{\infty}}\frac{L}{\mu}\parens{1+M} e^{-\mu(t-T)}.
	\end{align}
Hence, we have shown that there exists  $\eta>0$ small enough such that for any $x\in B(x_0,\eta)$,  $U_t(x,y)$ converges to an element $U_{\infty}(x,y)$ satisfying $A_{\infty}(x,y)U_{\infty}(x,y)=A_{\infty}(x,y)U^{\star}(x,y)$. Moreover, the there exists a time $T$ and positive constants $C'$ and $\mu'$  such that:
\begin{align}
	\Verts{U_t(x,y)-U_{\infty}(x,y)}\leq C'e^{-\mu't},\forall t\geq T, \forall x\in B(x_0,\eta).
\end{align}
	
\end{proof}

\begin{proof}[Proof of \cref{prop:diff_flow_selection}.]
By \cref{prop:convergence_flow}, we have that $\phi_t(x,y)$ converges to $\phi(x,y)$. Moreover, since $\phi(x_0,y)$ is a local minimizer of $g(x_0,.)$, \cref{prop:continuity_at_point}  ensures that $\phi(x,y)$ is continuous at $x_0$. Finally, we know by \cref{eq:continuity_U_t} that $\phi_t(x,y)$ is differentiable in $x$ and by \cref{prop:convergence_U} that   $\partial_x\phi_t(x,y):=U_t(x,y)$  converges uniformly towards $U_{\infty}(x,y)$. Therefore, by \cite[Theorem 7.17]{Rudin:1976}, we conclude that $\phi(x,y)$ is differentiable in a neighborhood of $x_0$ with differential given by $\partial_x\phi(x,y) = U_{\infty}(x,y)$. 
If in addition, $y$ is a local minimizer, then,  by \cref{prop:convergence_U},  $\partial_x\phi(x,y) {=} -\partial_{xy}g(x,y)\parens{\partial_{yy}g(x,y)}^{\dagger}$.
\end{proof}

\section{Limits Points of Bilevel Optimization Algorithms}\label{sec:proof_algorithms}
\begin{prop}\label{lem:sufficient_decrease}
Let $g$ be a real-valued function on $\X\times \Y$ such that \cref{assumpt:smootness} holds. Consider the maps $\varphi_T$ and $\mathcal{I}_M$ defined in \cref{eq:example_varphi} and let $T$ and $M$ be non-negative integers, such that $T+M>0$. Let $(x,y)\in \X\times \Y$ such that $\varphi_T(x,\mathcal{I}_M(x,y))=y$. Then, $\partial_y g(x,y)=0$.
\end{prop}
\begin{proof}
	Let us fix $(x,y)\in \X\times \Y$ and consider the iterates $y^T = \varphi_T(x,y)$. 
	We will show that $y^T $ satisfy a sufficient decrease condition for some positive constant $a$:
	\begin{align}\label{eq:sufficient_decrease}
		g(x,y^{T+1} ) + \frac{L}{2}\Verts{y^T -y^{T+1} }^2\leq g(x,y^T).
	\end{align}
To see this, we can use the smoothness of $g$ to write:
\begin{align}
	g(x,y^{T+1})- g(x,y^T) \leq -d^{\top}H_Td + \frac{L}{2} \Verts{H_Td}^2,
\end{align}
where $d = \partial_y g(x,y^T)$ and we write $H_T {=} H_T(x,y)$ by abuse of notation. Hence, it follows that:
\begin{align}
	g(x,y^{T+1})- g(x,y^T) + \frac{L}{2}\Verts{y^{T+1}-y^T}^2 \leq - d^{\top}\parens{H_T-LH_T^2}d\leq 0,
\end{align}
where we used that $H_T\leq \frac{1}{L}I$. Similarly, we obtain a sufficient decrease condition for the iterates defined by  $\mathcal{I}_M$. Consider now $T$ and $M$, such that $T+M>0$, and let $(x,y)$ be such that $\varphi_T(x,\mathcal{I}_M(x,y)){=}y$. Consider the iterates $y^k {=} \mathcal{I}_k(x,y)$ for $m\leq M$, and $y^{k} {=} \varphi_t(x,y^M)$ for $t\leq T$. Then the iterates $y^{k}$ define a non-increasing sequence $g(x,y^{k})$. Moreover, since $y^{T+M}=y^0=y$, it must be that $g(x,y^{k})=g(x,y)$. The sufficient decrease condition in  \cref{eq:sufficient_decrease} implies that the iterates are all constant $y^{k}=y^{0}$. In particular, if $M>0$, this implies that $H_M(x,y)\partial_y g(x,y)=0$ so that $\partial_y g(x,y)=0$ since $H_M(x,y)$ is invertible. On the other hand, if $M=0$, then the condition $T+M>0$ implies that $T>0$, so that $y = y^1 = y-H_T(x,y)\partial_y g(x,y)$. Similarly, since $H_T(x,y)$, we deduce that $\partial_y g(x,y) = 0$.
\end{proof}

\begin{prop}[\bf Properties of the maps $\varphi_T$ and $\mathcal{I}_M$]\label{prop:properties_varphi}
	Let $g$ be a function satisfying \cref{assumpt:smootness} with a smoothness constant $L$. Consider $\varphi_T(x,y)$ and $\mathcal{I}_M(x,y)$ defined by the following recursion which holds for any $x,y\in \X\times \Y$:
	\begin{align}\label{eq:example_varphi}
		\varphi_{T+1}(x,y) &= \varphi_{T}(x,y)-H_T(x,y)\parens{\partial_y g(x,\varphi_T(x,y))} ,\qquad \varphi_{0}(x,y)= y\\
		\mathcal{I}_{M+1}(x,y) &=  \mathcal{I}_{M}(x,y)-H_M'(x,y)\parens{\partial_y g(x,\mathcal{I}_M(x,y))}, \qquad \mathcal{I}_0(x,y)=y,
	\end{align}
	where $H_T(x,y)$ and $H_M'(x,y)$ are positive symmetric matrices satisfying $H_M'(x,y)\leq \frac{1}{L} I$ and $H_T(x,y)\leq \frac{1}{L} I $ for any $(x,y)\in \X\times \Y$ and non-negative integers $T,M$.
	Moreover, assume that $H_T(x,y)$ is continuously differentiable. Then $\varphi_T$ and $\mathcal{I}_M$ satisfy \cref{assumpt:map_varphi}.	
\end{prop}
\begin{proof}
	It is clear that for any $(x,y)\in \X\times \Y$, s.t. $y$ is a critical point of  $g(x,.)$, we have that $\mathcal{I}_M(x,y)=\varphi_T(x,y)=y$. Moreover, if $T,M$ are such that $T+M>0$ and $(x,y)\in \X\times \Y$ satisfy $\varphi_T(x,\mathcal{I}_M(x,y))=y$, then \cref{lem:sufficient_decrease} ensures that $\partial_y g(x,y){=}0$. It remains to obtain an expression for $\partial_x\varphi_T$ and $\partial_y \varphi_T$ in terms of second-order derivatives of $g$. We proceed by recursion. For $T=0$, by setting $D=0$, we have that:
	\begin{align}
		\partial_x \varphi_0(x,y) = 0 = \partial_{xy}^2 g(x,y) D, \qquad \partial_y \varphi_0(x,y) = I = I + \partial_{xy}^2 g(x,y) D.
	\end{align}
Let $(x,y)$ be an augmented critical point of $g$. 
Assume now that for some $T\geq 0$, there exists  a matrix $D_T$, such that:
\begin{align}\label{eq:recursion_varphi}
	\partial_x\varphi_T(x,y) = \partial_{xy}^2 g(x,y) D_T,\qquad \partial_y \varphi_T(x,y) = I + \partial_{yy}^2 g(x,y) D_T.  
\end{align}
Differentiating the expression of $\varphi_{T+1}(x,y)$ w.r.t.  $x$ and $y$ yields:
\begin{align}
	\partial_x\varphi_{T+1}(x,y) =& \partial_x \varphi_T(x,y)-\parens{\partial_{xy}g(x,\varphi_T(x,y)) + \partial_x\varphi_T(x,y) \partial_{yy}^2g(x,\varphi_T(x,y))}H_T(x,y) \\
	&- \partial_x H_T(x,y)\partial_y g(x,\varphi_T(x,y)).\\
	=& \partial_{xy}g(x,y)\parens{D_T - \parens{I+D_T \partial_{yy}^2g(x,y)}H_T(x,y)} = \partial_{xy}g(x,y) D_{T+1},
\end{align}
Where we defined $D_{T+1}(x,y) {=} D_T - \parens{I+D_T \partial_{yy}^2g(x,y)}H_T(x,y)$. In the above expression, the last line follows by recalling that $\varphi_T(x,y){=}0$ and $\partial_x H_T(x,y)\partial_y g(x,\varphi_T(x,y))=0$ since $(x,y)$ is an augmented critical point of $g$ and by using the recursion assumption on $\partial_x \varphi_T(x,y)$. 

Similarly, for $\partial_y \varphi_{T+1}(x,y)$, the following holds:
\begin{align}
	\partial_y\varphi_{T+1}(x,y) =& \partial_y \varphi_T(x,y)-\partial_y\varphi_T(x,y) \partial_{yy}^2g(x,\varphi_T(x,y))H_T(x,y) \\
		&- \partial_y H_T(x,y)\partial_y g(x,\varphi_T(x,y)),\\
		=&
		I + \partial_{yy}^2g(x,y)\parens{D_T-\parens{I+D_T\partial_{yy}^2g(x,y)}H_T(x,y)}= I+\partial_{yy}^2g(x,y)D_{T+1}.
\end{align}
Hence, by recursion, $\varphi_T(x,y)$ satisfies the equation \cref{eq:recursion_varphi} for any $T\geq 0$. We have shown that $\varphi_T$ and $\mathcal{I}_M$ satisfy \cref{assumpt:map_varphi}.
\end{proof}

\begin{proof}[Proof of \cref{prop:critical_point_no_correction}] Fix $T\geq$ and consider the iterates $(x_k,y_k)$ of \cref{alg:abg_alg} using $\varphi_T$.
	By assumption $(x_k,y_k)_{k\geq 0}$ converges to an element $(x_T^{\star},y_T^{\star})$ in $\X\times \Y$. By continuity of the maps $\varphi_T$, $\mathcal{I}_M$ and $\partial_x \mathcal{L}_T$, we have that:
	\begin{align}
		y_T^{\star} = \lim_k y_k =  \lim_k \varphi_T(x_{k-1},\mathcal{I}_M(x_{k-1},y_{k-1})) &= \varphi_T(x_T^{\star},\mathcal{I}_M(x_T^{\star},y_T^{\star})),\\
		\lim_k \partial_x \mathcal{L}_T(x_{k-1},\mathcal{I}_M(x_{k-1},y_{k-1}))&= \partial_x \mathcal{L}_T(x_T^{\star},\mathcal{I}_M(x_T^{\star},y_T^{\star})):= d^{\star}.
	\end{align}
	By \cref{assumpt:map_varphi}, the first equation implies that $y_T^{\star}$ is a critical point of $g(x_T^{\star},.)$ (i.e. $\partial_y g(x_T^{\star},y_T^{\star})=0$). Moreover, taking the limit in the update equation $x_k = x_{k-1}-\gamma d_k$ yields $d^{\star}=0$. Hence, we also have that $\partial_x \mathcal{L}_T(x_T^{\star},\mathcal{I}_M(x_T^{\star},y_T^{\star}))=0$. Finally, recall that $\mathcal{I}_M(x_T^{\star},y_T^{\star}) {=}y_T^{\star}$ by \cref{assumpt:map_varphi} since $(x_T^{\star},y_T^{\star})$ is an augmented critical point of $g$. Thus we have shown that:
	\begin{align}
		\partial_x \mathcal{L}_T(x_T^{\star},y_T^{\star})=0,\qquad \partial_y g(x_T^{\star},y_T^{\star}). 
	\end{align}
	Assume now that $y_T^{\star}$ is a local minimum of $g(x_T^{\star},.)$ and that $(x_T^{\star},y_T^{\star})_{T\geq 0}$ is bounded. Hence, there exists a subsequence of $(x_T^{\star},y_T^{\star})_{T\geq 0}$ converging towards an accumulation point $ (x^{\star},y^{\star})$. By abuse of notation, we denote $ (x_T^{\star},y_T^{\star})_{T\geq 0}$ such  subsequence. By continuity of the Hessian of $g$, it follows that $y^{\star}$ must also be a local minimum of  $g(x^{\star},.)$. We can now use \cref{assumpt:map_convergence_to_selection} which ensures that $\varphi_T$ converges to a selection $\phi$. Moreover, since $\partial_x\varphi_T$ converges uniformly near local minima, it follows by \citep[Theorem 7.17]{Rudin:1976} that $\phi(x,y)$ is differentiable w.r.t. $x$ near $(x^{\star},y^{\star})$ and that $\partial_{x}\varphi_T(x,y)$ converges uniformly near $(x^{\star},y^{\star})$ towards $\partial_x \phi(x,y)$. Hence, we can write for $T$ large enough:
	\begin{align}
		\partial_x \mathcal{L}_{\phi}(x_T^{\star},y_T^{\star}) =& \partial_x\mathcal{L}_T(x_T^{\star},y_T^{\star}) + \parens{\partial_x\phi(x_T^{\star},y_T^{\star})- \partial_x\varphi_T(x_T^{\star},y_T^{\star})}\partial_y f(x_T^{\star},y_T^{\star}),\\
		=& \parens{\partial_x\phi(x_T^{\star},y_T^{\star})- \partial_x\varphi_T(x_T^{\star},y_T^{\star})}\partial_y f(x_T^{\star},y_T^{\star}).
	\end{align}
	By uniform convergence of $\partial_{x}\varphi_T(x,y)$ to $\partial_x \phi(x,y)$ and recalling that $(x_T^{\star},y_T^{\star})$ is bounded, we deduce that $\Verts{\partial_x \mathcal{L}_{\phi}(x_T^{\star},y_T^{\star})}$ converges to $0$. In particular, this holds true for a subsequence satisfying $\lim\sup_T \Verts{\partial_x \mathcal{L}_{\phi}(x_T^{\star},y_T^{\star})} {=} \lim_T \Verts{\partial_x \mathcal{L}_{\phi}(x_T^{\star},y_T^{\star})}$, which proves the desired result.
\end{proof}

\begin{proof}[Proof of \cref{prop:gradient_correction}]
	Let $(x,y)\in \X\times \Y$ be such that $y$ is a local minimum of $g(x,.)$. Define $d$ to be:
	\begin{align}
		d = \partial_x \mathcal{L}_T(x,y) - \partial_{xy}^2 g(x,y)\parens{\partial_{yy}^2g(x,y)}^{\dagger}\partial_y \mathcal{L}_T(x,y). 
	\end{align}
	By \cref{prop:diff_flow_selection}, $x\mapsto \phi(x,y)$ is differentiable at $x$ and since $y$ is a critical point of $g(x,.)$, the differential of $\phi(x,y)$ is given by $\partial_x\phi(x,y) = -\partial_{xy}^2 g(x,y)\parens{\partial_{yy}^2g(x,y)}^{\dagger}$. Hence, $d$ is equal to:
	\begin{align}
		d = \partial_x \mathcal{L}_T(x,y) + \partial_x \phi(x,y)\partial_y \mathcal{L}_T(x,y). 
	\end{align} 
  Using the definition of $\mathcal{L}_T$ and recalling that $\varphi_T$ satisfies \cref{assumpt:map_varphi}, the following holds:
	\begin{align}
		d &= \partial_x f(x,\varphi_T(x,y)) + \partial_x \varphi_T(x,y)\partial_y f(x,\varphi_T(x,y)) + \partial_x \phi(x,y)\partial_y\varphi_T(x,y)\partial_y f(x,\varphi_T(x,y)),\\
		&=
		\partial_x f(x,y) + \parens{\partial_x \varphi_T(x,y) +\partial_x \phi(x,y)\partial_y\varphi_T(x,y)}\partial_y f(x,y),\\
		&= \partial_x f(x,y) + \parens{\partial_{xy}^2g(x,y)D +\partial_x \phi(x,y)\parens{I + \partial_{yy}^2g(x,y)D }}\partial_y f(x,y),\\
		&= \partial_x f(x,y) +\partial_x \phi(x,y)\partial_y f(x,y) +\parens{\partial_{xy}^2g(x,y) +\partial_x \phi(x,y) \partial_{yy}^2g(x,y) }D\partial_y f(x,y),\\
		&= \partial_x \mathcal{L}_{\phi}(x,y) + \parens{\partial_{xy}^2g(x,y) +\partial_x \phi(x,y) \partial_{yy}^2g(x,y) }D\partial_y f(x,y).
	\end{align}
	The last term of the above equation vanishes, since by definition of $\partial_x\phi(x,y)$, it holds that $\partial_{xy}^2g(x,y) +\partial_x \phi(x,y) \partial_{yy}^2g(x,y) = 0$. Therefore, we have shown that $d= \partial_x \mathcal{L}_{\phi}(x,y)$, which concludes the proof.
\end{proof}

\begin{proof}[Proof of \cref{prop:critical_point_correction}]
By continuity of the maps $\varphi_T$ and $\mathcal{I}_M$ and since $(x_k,y_k,z_k)\rightarrow (x^{\star},y^{\star},z^{\star})$,  it holds that $y^{\star}= \varphi_T(x^{\star},\mathcal{I}_M(x^{\star},y^{\star}))$. Hence, by  \cref{assumpt:map_varphi}, it follows that $y^{\star}$ must be a critical point of $y\mapsto g(x^{\star},y^{\star})$, i.e. $\partial_y g(x^{\star},y^{\star}){=}0$. 
Moreover, we have that $\tilde{y}_k= \mathcal{I}_M(x_k,y_k)\xrightarrow[k]{} \mathcal{I}_M(x^{\star},y^{\star})=y^{\star}$ by continuity of $\mathcal{I}_M$ and the condition in \cref{assumpt:map_varphi}. 
Since $f$ and $\varphi_T$ are continuously differentiable we get that $u_k,v_k \xrightarrow[k]{} \partial_x \mathcal{L}_{T}(x^{\star},y^{\star}), \partial_y \mathcal{L}_{T}(x^{\star},y^{\star})$. Moreover, recalling that $\mathcal{L}_{T}(x,y) {=} f(x,\varphi_T(x,y))$, by application of  the chain rule and using that $\varphi_T(x^{\star},y^{\star}){=} y^{\star}$ it follows that:
\begin{align}\label{eq:def_u_v_star}
	u_k &\xrightarrow[k]{} u^{\star}:= \partial_x f(x^{\star},y^{\star}) + \partial_x\varphi_T(x^{\star},y^{\star})\partial_y f(x^{\star},y^{\star}), \\
	v_k &\xrightarrow[k]{} v^{\star} :=   \partial_y\varphi_T(x^{\star},y^{\star})\partial_y f(x^{\star},y^{\star}).
\end{align}
By continuity of the higher-order derivatives of $g$, it holds that:
\begin{align}
	\partial_{yy}^2 g(x_k,y_{k+1})&\xrightarrow[k]{} A^*:= \partial_{yy}^2 g(x^{\star},y^{\star}),\\
	\partial_{xy}^2 g(x_k,y_{k+1})&\xrightarrow[k]{} B^*:= \partial_{xy}^2 g(x^{\star},y^{\star}).
\end{align} 
Recall that $z_k$ is given by the update equation $z_k = \mathcal{P}(A_k,v_k,z_{k-1})$, 
where $\mathcal{P}$ is a continuous map for which $z=\mathcal{P}(A,v,z)$ if and only if $z\in \arg\min_z \Verts{A^2z+v}^2$. By continuity of $\mathcal{P}$, it follows that $z^{\star}$ satisfies:
\begin{align}
	z^{\star} = \mathcal{P}(A^{\star},v^{\star},z^{\star}).
\end{align}
Therefore, $z^{\star}$ minimizes $z\mapsto\Verts{(A^{\star})^2z + v^{\star}}^2$ and satisfies the fixed point equation $(A^{\star})^3z^{\star} +A^{\star}v^{\star} {=} 0 $ so that $A^{\star}z^{\star}= -(A^{\star})^{\dagger}v^{\star}$. Moreover, recall that $\xi_k {=} A_kz_k$, hence $\xi_k$ converges towards $\xi^{\star}:= A^{\star}z^{\star}$. Therefore, $\xi^{\star} {=} -(A^{\star})^{\dagger}v^{\star}$. 
Taking the limit as $k$ goes to $+\infty$, we get that $d_k$ defined in \cref{alg:abg_alg} converges towards $d^{\star}$ defined by:
\begin{align}
	d^{\star} :&= u^{\star} + B^{\star}\xi^{\star},\\
				&= u^{\star} - B^{\star}(A^{\star})^{\dagger}v^{\star}.
\end{align}
By \cref{prop:gradient_correction}, it is easy to see that $d^{\star} = \partial_x\mathcal{L}_{\phi}(x^{\star},y^{\star})$. Finally, recalling the update equation $x_{k+1}{=}x_k-\gamma d_k$ and that $x_k\xrightarrow[k]{} x^{\star}$, we directly deduce that $d_k\xrightarrow[k]{} 0$, so that $d^{\star}=0$. 
This shows that $(x^{\star},y^{\star})$ is an equilibrium point of \cref{eq:ABG} and satisfies \cref{eq:SC}.  
\end{proof}

\subsection{Warm-start Strategy}\label{sec:warm-start}
In this section, we provide simple examples for the map $\mathcal{P}(A,v,z)$ to find approximate solutions minimizing $Q(z):=\frac{1}{2}\Verts{A^2z+v}^2$, where  $A$ is a symmetric matrix in $\R^{d\times d}$ satisfying $A\leq LI$, with $L$ being the smoothness constant of $g$ in \cref{assumpt:smootness}.  
The algorithm $\mathcal{P}$ can be as simple as $N$-step of conjugate gradient descent on $Q$ with a step-size $\alpha\leq \frac{1}{L^4}$ where $L$ is the smoothness constant of $g$ in \cref{assumpt:smootness}. More formally, $\mathcal{P}(A,v,z) {=} z^{N}$ where $z^{N}$ is the $N$ iterate of the following recursion:
\begin{align}\label{eq:gradient_descent_Q}
	z^{n+1} = z^n -\alpha \partial_z Q(z^n),\qquad z^0=z. 
\end{align}
It is clear that $\mathcal{P}(A,v,z)$ is continuous in its arguments. Moreover, using a similar argument as in \cref{lem:sufficient_decrease}, one can prove that whenever  $z$ is a fixed point of $\mathcal{P}(A,v,z)$, then $z$ must be a critical point of $Q$ and therefore satisfies the equation $A^3z+Av=0$. The update equation in \cref{eq:gradient_descent_Q} depends however on the step-size $\alpha$ which needs to be smaller than $\frac{1}{L^4}$. to avoid the dependence on such step-size, A more efficient choice for the map $\mathcal{P}$ which does not require using a step-size, is to perform $N$ conjugate gradient iterations on $Q$ starting from an initial condition $z$. 

\subsection{Recovering Existing Algorithms }\label{sec:recovering_existing_alg}

\cref{table:summary_alg} below summarizes how to recover well-known gradient-based algorithms for bilevel optimization from \cref{alg:abg_alg}. 
\begin{table}
\centering
\begin{tabular}{|l|l|l|l|}
\hline
Algorithm     & T & M & Correction  \\
\hline
ITD  \citep{Baydin:2018}          & $T>0$ & $M=0$ & False       \\
\hline 
Corrected ITD            & $T>0$ & $M=0$ & True       \\
\hline
Truncated ITD \citep{Shaban:2019} & $T>0$ & $M>0$ & False       \\
\hline
Corrected Truncated ITD  & $T>0$ & $M>0$ & True       \\
\hline
AID \citep{Pedregosa:2016}           & $T=0$ & $T>0$ & True \\     
\hline
\end{tabular}
\vspace{0.5cm}
\caption{Recovering bilevel optimization algorithms from \cref{alg:abg_alg}. }\label{table:summary_alg}
\end{table}
Hence, \cref{alg:abg_alg} recovers the most popular bilevel optimization algorithms but also introduces a corrected version to them to ensure that they recover the equilibria of \cref{eq:ABG}.

\section{Experiments}\label{sec:experiments}
To illustrate the effect of the corrective term introduced in \cref{sec:algorithms} , we consider two sets of experiments: a synthetic problem for which the optimal solutions can be computed in closed form and a dataset distillation task on Cifar10 \citep{Krizhevsky:2009} using a ResNet18 architecture \citep{He:2015}. 
\subsection{Synthetic Problem}
Motivated by the instrumental variable regression problem \citep{Singh:2019} which solves a bilevel problem with quadratic objectives for both levels, we consider lower and upper-level objectives of the form:
\begin{align}
	f(x,y) &:= \frac{1}{2}x^{\top}A_f x + C_f^{\top}y\\
	g(x,y) &:= \frac{1}{2} y^{\top}A_g y + y^{\top}B_g x
\end{align}
where $A_f$ and $A_g$ are symmetric positive matrices of size $d_x{\times} d_x$ and $d_y{\times} d_y$,  $B_g$ is a $d_y{\times} d_x $ matrix and $C_f$ is a $d_y$ vector with $d_x{=}2000$ and $d_y{=}1000$. 
To allow for multiple solutions to the LL objective,  we choose $A_g$ to be non-invertible with a null-space of dimension $100$ while we choose $A_f$ to be invertible for simplicity. Furthermore, to ensure that $f$ admits a finite minimum value we choose $B_g$ to be of the form $A_g U$ for some randomly sampled matrix $U$. 
We construct the matrices $A_f$ and $A_g$ so that the highest eigenvalues of $A_f$ and $A_g$ are smaller than $1$ and their conditioning is equal to $10$. Here, we define the conditioning of a matrix to be the ratio between the highest and smallest non-zero eigenvalues. For a given $x$, the minimizers of $g$ are of the form: 
\begin{align}
	y = -A_g^{\dagger}B_gx + (I-A_gA_g^{\dagger})y_0,
\end{align}
where $y_0$ is any vector in $\R^{d_y}$. Replacing the optimal $y$ in the UL objective results in the expression which holds for any $y_0\in\R^{d_y}$.
\begin{align}
	\frac{1}{2}x^{\top}A_f x - C_f^{\top}A_g^{\dagger}B_gx + C_f^{\top}(I-A_gA_g^{\dagger})y_0.
\end{align}
At this point, it is easy to check that either maximizing or minimizing the above objective over  $y_0$ results in an infinite value of the objective whenever $C_{f}^{\top}(I-A_gA_g^{\dagger})$ is non-zero. This implies that the optimistic and pessimistic formulations of the bilevel problem result in an infinite optimal loss. However, the \cref{eq:ABG} has a well-defined solution. To see this, it is possible to define a selection of the form $\phi(x,y) {=} -A_g^{\dagger}B_gx + (I-A_gA_g^{\dagger})y$ which corresponds to the limit of a gradient flow of $g$ initialized at $y$. The upper objective of \cref{eq:ABG} is therefore given by:
\begin{align}
	\mathcal{L}_{\phi}(x,y)
	=  \frac{1}{2}x^{\top}A_f x - C_f^{\top}A_g^{\dagger}B_gx + C_f^{\top}(I-A_gA_g^{\dagger})y.
\end{align}
Instead of optimizing $\mathcal{L}_{\phi}(x,y)$ over $x$ and $y$ which would result in an infinite loss, \cref{eq:ABG} optimizes $\mathcal{L}_{\phi}(x,y)$ over $x$ only, while $y$ is optimized for $f(x,y)$, thus seeking an equilibrium $(x^{\star},y^{\star})$ satisfying \cref{eq:SC} which can be expressed in closed form as
\begin{align}
	x^{\star} := A_f^{-1}B_g^{\top}A_g^{\dagger}C_f,\qquad
	y^{\star} :=   -A_g^{\dagger}B_gx + (I-A_gA_g^{\dagger})y_0
\end{align}
where $y_0$ is any vector in $\R^{d_y}$. Hence, while there exist multiple equilibria, they all have the same value for $x^{\star}$ and yield a finite objective.

We solve the above problem using \cref{alg:abg_alg} either using the correction or not. When using the correction, we compute the approximate solution $\xi_k$ to the linear system \cref{eq:min_norm_least_square} using the following update rule:
\begin{align}\label{eq:update_eq_corrective_term}
	\xi_k = \xi_{k-1}-\beta (\partial_{yy}g(x_{k-1},y_k)\xi_{k-1} + v_k)
\end{align} 
where $\beta=0.9$ is a positive step-size. For the lower-level problem, we use $T$ steps of gradient descent with a step-size $\alpha=0.9$   while we set the upper-level step-size to $\gamma=1.$. We then set the warm-start parameter value $M$ to $0$ and vary $T$. 

\paragraph{Results.} We consider the distance of the iterate $x_k$ to the optimal equilibrium $x^{\star}$ as measured by the metric induced by $A_f$:  
\begin{align}
	\Verts{x_k-x^{\star}}^{2}_{A_f} := \frac{1}{2}\parens{x_k-x^{\star}}^{\top}A_f\parens{x_k-x^{\star}}
\end{align}
\cref{fig:toy_quad} (left)  shows the evolution of $\Verts{x_k-x^{\star}}^{2}_{A_f}$ as a function of time (in seconds) for different algorithmic choices, while \cref{fig:toy_quad}(right) shows the evolution of the approximate upper-level gradient $d_k$ used in \cref{alg:abg_alg}. 
We first observe that, without correction, and when using a small number of unrolled iterations $(T\leq 10)$, the algorithm does not converge towards $x^{\star}$, (the distance to the iterate is larger than $10^3$). Instead, the algorithm reaches a different equilibrium as suggested by the evolution of the gradient approximation $d_k$ towards $0$ (\cref{fig:toy_quad}-(right)). As the number of unrolling steps $T$ increases, the algorithm takes more time to converge as suggested by \cref{fig:toy_quad}-(right) (green trace $T=1000$). However, the limit gets closer to the equilibrium $x^{\star}$ (\cref{fig:toy_quad}-(right), green trace). This confirms our first convergence result in \cref{prop:critical_point_no_correction} stating that unrolled optimization finds an approximate solution to \cref{eq:ABG}.     

When using the correction, \cref{alg:abg_alg} is able to recover the equilibrium $x^{\star}$ while still using a small number of unrolling steps $T\leq 10$ and requiring less time to converge. This observation supports the result in \cref{prop:critical_point_correction}.

\begin{figure}
	\includegraphics[width=1.\linewidth]{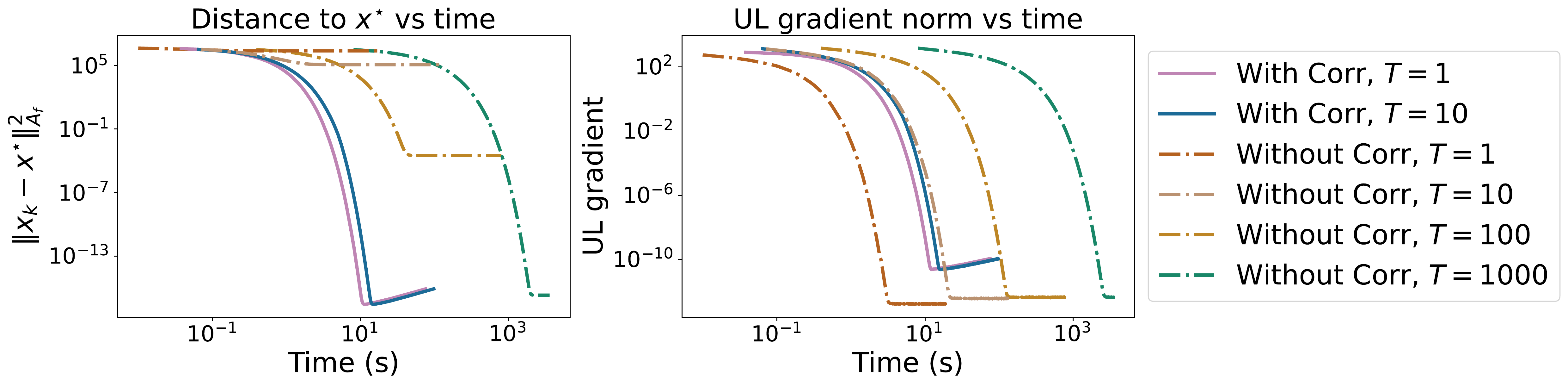}
	\caption{(left) Evolution of the distance of the UL iterate $x_k$ to the equilibrium $x^{\star}$  vs time (in seconds). (right) evolution of the norm of approximate gradient $d_k$ vs time in seconds. In all cases, algorithms are run until convergence, i.e. $\Verts{d_k}$ converges to $0$.}
\label{fig:toy_quad}
\end{figure}

\begin{figure}
\includegraphics[width=1.\linewidth]{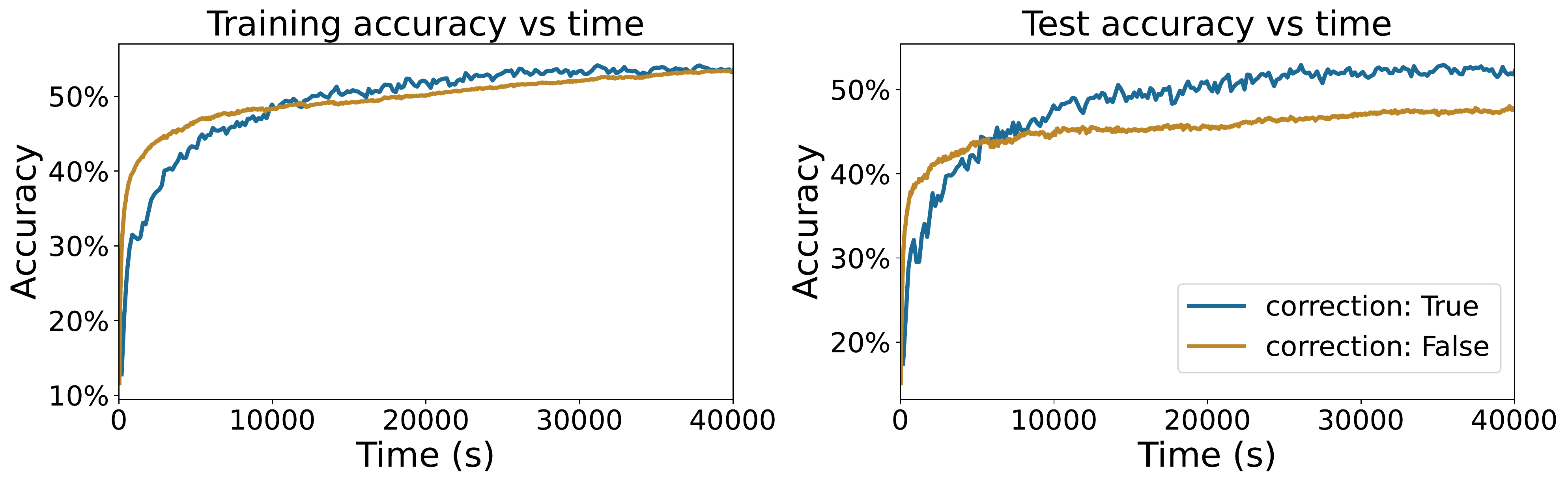}
	\caption{Evolution in time of the training and test accuracy of a ResNet18 model on Cifar10 dataset. Each iteration corresponds to the accuracy of the model with parameter $y_k$ trained on a synthetic dataset of $100$ points $x_k$ to minimize the LL objective.  The synthetic points $x_k$ are learned by minimizing the training error when using the running model $y_k$.}
\label{fig:distillation_cifar10}
\end{figure}

\subsection{Dataset Distillation on Cifar10}
We consider the task of learning a small synthetic dataset so that a classifier trained on such a dataset achieves a small error on a training set. More formally, we consider a classification problem with $C$ classes using a model with parameters $y$  and a training dataset $\mathcal{D}_{tr} = \{(\xi_i,c_i)\}$ consisting of $N$ i.i.d. samples $\xi_i$ and corresponding labels $c_i$. The goal is to learn a synthetic dataset of $F C$ points, where $F$ is a positive integer, such that each class $c$ contains $F$ representative samples. We can collect the synthetic points into a vector $x$ to be learned and denote by $\mathcal{D}_x$ the synthetic dataset. For a given dataset $\mathcal{D}$, denote by $\mathcal{L}_{\mathcal{D}}(y)$ the cross-entropy loss of a model with parameters $y$ evaluated on $\mathcal{D}$. The bi-level formulation of the distillation task consists in optimizing a lower-level objective $g(x,y) = \mathcal{L}_{\mathcal{D}_x}(y)$ to learn the model parameters $y$ that best predicts the classes of the synthetic dataset. The upper-level objective  $g(x,y) = \mathcal{L}_{\mathcal{D}_{tr}}(y)$ evaluates the optimal model on the training set and optimizes the synthetic samples. 

\paragraph{Setup}. We consider a setup similar to \citep{Wang:2018}  for distilling \verb+Cifar10+ \citep{Krizhevsky:2009} on $100$ synthetic points. We set $F{=}10$, thus requiring $10$ synthetic points for each of the $C{=}10$ classes of \verb+Cifar10+. We then use ResNet18 \citep{He:2015} as a classifier and apply \cref{alg:abg_alg} to learn the optimal synthetic points. For the lower level, we use gradient descent with $1$ unrolled iteration (i.e. $T=1$, $M=0$) and a step-size of $\alpha{=}0.001$. For the upper level, we use Adam optimizer \citep{Kingma:2014}, with the default parameters, a step-size of $\gamma=0.01$ and a batch-size of $1024$. When using the corrective term, we use the update equation \cref{eq:update_eq_corrective_term} with a step-size $\beta=0.0001$. 

\paragraph{Results.} \cref{fig:distillation_cifar10} shows the evolution of the training and test accuracy of the model as a function of time in two settings, either with or without correction. While the training accuracy for both versions of the algorithm is similar, the corrective term yields an improved final test accuracy ($54.19\%$ vs $48.6\%$). Note that these accuracies are of the same order as those obtained in \cite{Wang:2018a} suggesting that distilling Cifar10 in only $100$ samples is not sufficient to capture all variability in the dataset.
While the additional correction increases the computational cost per iteration, it provides a better gradient estimate which results in a faster/better performance overall.

\end{document}